\documentclass[12pt]{article}

\usepackage[utf8]{inputenc}
\usepackage[english]{babel}
\usepackage{multirow}
\usepackage{hyperref}

\hypersetup{
    bookmarks=true,         %
    unicode=false,          %
    pdftoolbar=true,        %
    pdfmenubar=true,        %
    pdffitwindow=false,     %
    pdfstartview={FitH},    %
    pdftitle={My title},    %
    pdfauthor={Author},     %
    pdfsubject={Subject},   %
    pdfcreator={Creator},   %
    pdfproducer={Producer}, %
    pdfkeywords={keyword1, key2, key3}, %
    pdfnewwindow=true,      %
    hidelinks=true,
}

\usepackage{graphicx}
\usepackage{tikz}
\usepackage{tikz-cd}

\usetikzlibrary{
  positioning,
  decorations.markings,
  arrows.meta,
  calc,
  angles,
  quotes
}

\definecolor{aqua}{rgb}{0.0, 1.0, 1.0}
\usepackage[background = aqua, arrow = aqua, text = black]{callouts}

\usepackage{amsfonts}
\usepackage{amsmath}
\usepackage{amssymb}
\usepackage{amsthm}
\usepackage{mathtools}

\usepackage{verbatim}

\usepackage{thmtools, thm-restate}

\theoremstyle{plain}
  \newtheorem{thm}{Theorem}[section]
  \newtheorem{lem}[thm]{Lemma}
  
  \newtheorem{cor}[thm]{Corollary}
  
  \newtheorem{conj}[thm]{Conjecture}

\theoremstyle{definition}
  \newtheorem{dfn}[thm]{Definition}

\theoremstyle{remark}
  \newtheorem{rem}[thm]{Remark}

\newcommand{\angb}[1]{\langle #1 \rangle}
\DeclareMathOperator{\Ima}{Im}
\DeclareMathOperator{\Int}{\text{Int}}
\DeclareMathOperator{\Sing}{\text{Sing}}

\title{$\mathbb{Z}_p$-folded Surfaces and the $\mathbb{Z}_p$-Thurston norm}
\author{William Hobkirk\\[1ex]}
\date{}

\begin{document}

\maketitle

\begin{abstract}
    $\mathbb{Z}_p$-folded surfaces are a class of oriented folded surfaces with a fundamental $H_2(\mathbb{Z}/p\mathbb{Z})$ class. In this paper we study the $\mathbb{Z}_p$-folded surfaces that embed in $3$-manifolds. This perspective naturally leads to the definition of the $\mathbb{Z}_p$-Thurston norm which generalises the $\mathbb{Z}_2$-Thurston norm in~\cite{jaco2009z2thurstonnormcomplexity3manifolds} and adapts the definition of $\theta$ in~\cite{Turaevfnhom3mfld2007}. We define \textit{canonical $\mathbb{Z}_p$-folded surfaces} which connect the theory of $\mathbb{Z}_p$-folded surfaces to the study of one vertex $3$-manifold triangulations and we relate the Euler characteristic of canonical $\mathbb{Z}_p$-folded surfaces to the $\mathbb{Z}_p$-Thurston norm. We generalise a classical result of Bredon and Wood in~\cite{bredonwood1969} by proving that the maximum Euler characteristic of a non-separating non-empty $\mathbb{Z}_3$-folded surface embedded in $L(3n,1)$ is $2-n$ for $n\in \mathbb{Z}^{>0}$, determining the $\mathbb{Z}_3$-Thurston norm in these cases. We conclude by demonstrating the motivating application of this theory to the study of generalised triangulations by proving that any generalised triangulation of the lens space $L(3n,1)$ must have at least $n$ tetrahedra. This final result is a preview of joint work with Spreer in which we determine the exact complexity of $L(3n,1)$ using more sophisticated techniques.
\end{abstract}

\paragraph{MSC 2020 classification:} Primary 57K31; Secondary 57K30, 57Q15.

\paragraph{Keywords:} triangulations of $3$-manifolds, triangulation complexity, folded surfaces, lens spaces, layered triangulations, Thurston norm, $\mathbb{Z}_2$-Thurston norm.

\tableofcontents

\section{Introduction}

Generally speaking, the contribution of this research is the application of similar techniques to study $3$-manifolds with non-trivial $H_2(\mathbb{Z}_p)$. To this end we develop a theory of $\mathbb{Z}_p$-folded surfaces and the $\mathbb{Z}_p$-Thurston norm that generalises the $\mathbb{Z}_2$-Thurston norm. We apply this theory to study $\mathbb{Z}_3$-folded surfaces that embed in the family of lens spaces $L(3n,1)$ for $n\in \mathbb{Z}^{>0}$ and prove Theorem~\ref{thm:l3n1norm} which states that the maximum Euler characteristic of a $\mathbb{Z}_3$-folded surface embedded in $L(3n,1)$ which does not separate the ambient space is $2-n$. This result, coupled with the construction of canonical $\mathbb{Z}_p$-folded surfaces yields Corollary~\ref{cor:naivecomplexitybound} which states that any triangulation of $L(3n,1)$ must have at least $n$ tetrahedra.

This research is inspired by the following four papers which all define invariants of 3-manifolds using embedded topological spaces which represent 2-homology classes:

\begin{itemize}
    \item Embedded oriented surfaces as representatives of $H_2(M;\mathbb{Z})$ classes and the Thurston norm $||\cdot||:H_2(M;\mathbb{Z})\rightarrow \mathbb{Z}^{\geq 0}$ defined in~\cite{Thurston1986ANF}.
    \item Unoriented surfaces as representatives of $H_2(M;\mathbb{Z}_2)$ classes and the $\mathbb{Z}_2$-Thurston norm $||\cdot||_2:H_2(M;\mathbb{Z}_2)\rightarrow \mathbb{Z}^{\geq 0}$ defined in~\cite{jaco2009z2thurstonnormcomplexity3manifolds}.
    \item Connected non-orientable surfaces as representatives of the unique non-trivial class of $H_2(M;\mathbb{Z}_2)$ where $M$ is a lens space, studied in~\cite{bredonwood1969} using combinatorial methods.
    \item Simple oriented folded surfaces as representatives of $H_2(M;\mathbb{Q}/\mathbb{Z})$ classes and Turaev's $\theta:H_2(M;\mathbb{Q}/\mathbb{Z})\rightarrow \mathbb{R}^{\geq 0}$ as defined in~\cite{Turaevfnhom3mfld2007}.
\end{itemize}

The functions $||\cdot||, ||\cdot||_2$ and $\theta$ are recalled in Sections~\ref{sec:Zthurstonnorm}-\ref{sec:theta} respectively.

In Section~\ref{sec:pfoldedsurfaces} we introduce $\mathbb{Z}_p$-folded surfaces. These $\mathbb{Z}_p$-folded surfaces are a refinement of the oriented folded surfaces described in~\cite{Turaevfnhom3mfld2007}. We describe their anatomy and in Section~\ref{sec:classification} we give a classification of $\mathbb{Z}_p$-folded surfaces and determine the $\mathbb{Z}_p$-folded surfaces that embed in orientable $3$-manifolds. In Section~\ref{sec:realisation} we define the fundamental class and the sense in which $\mathbb{Z}_p$-folded surfaces embedded in $3$-manifolds represent $H_2(M;\mathbb{Z}_p)$ classes. We also prove that any $H_2(\mathbb{Z}_p)$ class of an orientable $3$-manifold can be represented by such an embedded $\mathbb{Z}_p$-folded surface.

In Section~\ref{sec:zpthurstonnorm} we define a new invariant of $3$-manifolds which we call the $\mathbb{Z}_p$-Thurston norm $||\cdot||_{\mathbb{Z}_p}:H_2(M;\mathbb{Z}_p)\rightarrow \mathbb{Z}^{\geq 0}$. We determine the relationship between the $\mathbb{Z}_p$-Thurston norm and the $\mathbb{Z}$-Thurston norm and prove that the $\mathbb{Z}_p$-Thurston norm generalises the $\mathbb{Z}_2$-Thurston norm. In Section~\ref{sec:weightedfoldedsurfs} we generalise the weighted surfaces defined in~\cite[Chapter~IV.3.3]{Turaev2002torsions} to define \textit{weighted $\mathbb{Z}_p$-folded surfaces}. We further refine this object to define \textit{reduced $\mathbb{Z}_p$-folded surfaces} in Section~\ref{sec:reducedpfolded}. These reduced $\mathbb{Z}_p$-folded surfaces allow us to define another function on homology $\chi_p:H_2(M;\mathbb{Z}_p)\rightarrow \mathbb{Z}$ which is defined as the maximum Euler characteristic of a reduced $\mathbb{Z}_p$-folded surface. When $p\in \{2,3\}$ a reduced $\mathbb{Z}_p$-folded surface can be taken simply as a $\mathbb{Z}_p$-folded surface that does not separate the ambient space. We connect this back to the $\mathbb{Z}_p$-Thurston norm by proving that $||x||_p = \max(0, -\chi_p(x))$.

In Section~\ref{sec:normalpfoldedsurfacesconstruction} we give a construction of a $\mathbb{Z}_p$-folded surface embedded in any closed $3$-manifold that depends only on a one vertex triangulation and a dual $H^1(\mathbb{Z}_p)$ class as input. In Section~\ref{sec:normalisreduced} we show that the Euler characteristic of these \textit{canonical $\mathbb{Z}_p$-folded surface} is bounded from above by $\chi_p$, providing the desired connection between the $\mathbb{Z}_p$-Thurston norm and triangulation complexity. 

In Section~\ref{sec:3foldsurgeries} we restrict our focus to $\mathbb{Z}_3$-folded surfaces, defining analogues of compression surgeries and incompressibility in this context. In Section~\ref{sec:reducedsurgeries} we define reduced versions of these surgeries that preserve the crucial property of not separating the ambient manifold. 

In Sections~\ref{sec:decompositions} and~\ref{sec:mocktransversality} we give some necessary preliminaries and in Section~\ref{sec:meridionalintpats} we introduce \textit{meridional intersection patterns} which allow us to reduce the proof of Theorem~\ref{thm:l3n1norm} to a more combinatorial problem about graph embeddings in the disk. In Sections~\ref{sec:meridionalintpats} and~\ref{sec:snakesec} we simplify these meridional intersection patterns with a combination of combinatorial and topological arguments that rely heavily on the surgeries developed in Section~\ref{sec:3foldedsurfs}. Section~\ref{sec:recursionrelation} finishes the proof of Theorem~\ref{thm:l3n1norm}.

Section~\ref{sec:naivecomplexity} gives the promised Corollary~\ref{cor:naivecomplexitybound} which states that any generalised triangulation of the lens space $L(3n,1)$ for $n\in \mathbb{Z}^{>0}$ must have at least $n$ tetrahedra.

While the applications we have so far all use $\mathbb{Z}_3$-folded surfaces, we expect that the theory of $\mathbb{Z}_p$-folded surfaces can be used to determine similar results for the lens spaces $L(pn,1)$ in future research, motivating the level of generality in Sections~\ref{sec:pfoldedsurfsoverarch}-~\ref{sec:canonicalpfoldedsurfaces}.

\subsection{AI declaration}

Our use of AI when conducting this research was minimal. Most of it was completed in 2024 and 2025 prior to the public release of the more useful large language models. Consequently, all lemmas and theorems in this manuscript were proved without the use of AI and it was not used for the drafting of this paper. In the editing stage AI use was limited to a single proofread for typos the week prior to submission. 

\subsection{Acknowledgements}

I am very grateful for the endless guidance and encouragement of my supervisor Jonathan Spreer. The motivating idea of this project to generalise the techniques from \cite{Jaco_2009} using the folded surfaces described in \cite{Turaevfnhom3mfld2007} stemmed from discussions he had with Rubinstein and Tillmann. I want to thank Daniele Celoria and Finn Thompson for helpful discussions about this research during several trips to The University of Queensland in 2025. Special thanks to Nelson Odins-Jones and Max Mikkelsen for their constant support and for many interesting conversations about folded surfaces over the years.

\section{Background}\label{sec:background}

We recall definitions of the $\mathbb{Z}$-Thurston norm, $\mathbb{Z}_2$-Thurston norm, Turaev's $\theta$ function and generalised triangulations for the reader's benefit. We also introduce some notation for the maps in standard homology sequences that will be used throughout. Note that we will use the notation $\mathbb{Z}_p := \mathbb{Z}/p\mathbb{Z}$ for $p\in \mathbb{Z}^{>0}$. 

\subsection{The $\mathbb{Z}$-Thurston Norm}\label{sec:Zthurstonnorm}

Let $(\Sigma,\partial \Sigma)$ be a compact oriented surface properly embedded in a compact orientable $3$-manifold $(M,\partial M)$ (\textit{properly} means that $\partial M \cap \Sigma = \partial \Sigma$). Denote the image of the fundamental class of $\Sigma$ under the induced homomorphism $H_2(\Sigma,\partial \Sigma;\mathbb{Z})\rightarrow H_2(M,\partial M;\mathbb{Z})$ by $[\Sigma]$. We say that $\Sigma$ \textit{represents} the class $[\Sigma]$ and every element of $H_2(M,\partial M;\mathbb{Z})$ can be represented in this way if $M$.

The function $\chi_-$ is defined for a connected properly embedded surface $\Sigma$ as $\chi_-(\Sigma) := \max(0,-\chi(\Sigma))$. This function is extended to a function on the class of properly embedded surfaces with multiple components by summing over connected components such that $\chi_-(\Sigma) := \sum_{\Sigma_i\in \pi_0(\Sigma)} \chi_-(\Sigma_i)$. The Thurston norm on $x \in H_2(M,\partial M;\mathbb{Z})$ is defined by allowing $\Sigma$ to run over all such properly embedded oriented surfaces representing $x$.

\[||x||_{\mathbb{Z}} := \min_{[\Sigma]=x} \chi_-(\Sigma)\]

See~\cite{Thurston1986ANF} for more details on the Thurston norm.

\subsection{The $\mathbb{Z}_2$-Thurston Norm}\label{sec:Z2thurstonnorm}

Again, let $(\Sigma, \partial \Sigma)$ be a compact (unoriented) surface properly embedded in a compact $3$-manifold $(M,\partial M)$. Every manifold has a unique $\mathbb{Z}_2$-orientation and we can use this to define a fundamental class of $\Sigma$ in $H_2(\Sigma,\partial \Sigma;\mathbb{Z}_2)$. Let $[\Sigma]$ be the image of this fundamental class under the induced homomorphism $H_2(\Sigma,\partial \Sigma;\mathbb{Z}_2)\rightarrow H_2(M,\partial M;\mathbb{Z}_2)$. We say that $\Sigma$ represents the class $[\Sigma]$ and every element of $H_2(M,\partial M;\mathbb{Z}_2)$ can be represented in this way. We define $\chi_-$ exactly as in Section~\ref{sec:Zthurstonnorm} by allowing components of $\Sigma$ to be non-orientable. Definition~\ref{dfn:z2thurstonnorm} recalls the definition of the $\mathbb{Z}_2$-Thurston norm.

\begin{dfn}\label{dfn:z2thurstonnorm}
    Let $M$ be a compact $3$-manifold. The $\mathbb{Z}_2$-Thurston norm is a function $||\cdot||_{\mathbb{Z}_2}:H_2(M;\mathbb{Z}_2)\rightarrow \mathbb{Z}^{\geq 0}$ defined by the following expression for each $x\in H_2(M;\mathbb{Z}_2)$ where $\Sigma$ runs over all unoriented surfaces properly embedded in $M$ that represent $x$.

    \[||x||_{\mathbb{Z}_2} := \min_{[\Sigma]=x}\chi_-(\Sigma)\]
\end{dfn}

See~\cite{jaco2009z2thurstonnormcomplexity3manifolds} for more details on the $\mathbb{Z}_2$-Thurston norm.

\subsection{Turaev's $\theta$ function}\label{sec:theta}

To describe $\theta$ we need to first define \textit{folded surfaces}. Definition~\ref{dfn:foldedsurface} restricts to the definition given in~\cite{Turaevfnhom3mfld2007} when $k\neq 2$. The reason for stating it in this slightly different way is to allow for a sensible definition of a $\mathbb{Z}_2$-folded surface in Section~\ref{sec:pfoldedsurfaces} as a non-orientable surface with some extra structure.

\begin{dfn}\label{dfn:foldedsurface}
    A \textit{folded surface} is a compact 2-dimensional polyhedron and sub-polyhedron pair $(X,\Sing(X))$ with the following properties:

    \begin{itemize}
        \item \textit{The singular locus}, $\Sing(X)$ is a compact $1$-manifold. 
        \item \textit{The interior}, $\Int(X) = X\setminus \Sing(X)$ is a $2$-manifold.
        \item Every point $p$ in $\Sing(X)$ has a neighbourhood in $X$ homeomorphic to $\mathbb{R}\times K_{k,1}$ where $p = (0,v)$, $K_{k,1}$ is the cone on $k>0$ univalent vertices, with a single $k$-valent cone vertex $v$. This is homeomorphic to $k$ upper half planes identified along a common $x$-axis on which $p$ lies. We refer to components of $\Sing(X)$ as \textit{singular arcs}.
    \end{itemize}
    An \textit{oriented} folded surface is a folded surface $X$ with an orientation on $\Int(X)$.

\end{dfn}

Let $X$ be a oriented folded surface embedded in an oriented closed $3$-manifold. The orientation of $\Int(X)$ determines a fundamental class which we denote $[\Int(X)]\in H_2(X,\Sing(X);\mathbb{Z})$. If the image of this class under the connecting homomorphism $d:H_2(X,\Sing(X);\mathbb{Z})\rightarrow H_1(\Sing(X);\mathbb{Z})$ is non-zero and $\Sing(X) \cong S^1$ we call $X$ a \textit{simple} folded surface. Since $H_1(\Sing(X);\mathbb{Z}) =\mathbb{Z}$ there is a unique choice of generator $[\Sing(X)]\in H_1(\Sing(X);\mathbb{Z})$ and positive integer $i_X$ such that $d([\Int(X)]) = i_X[\Sing(X)]$. Now consider the class $i_X^{-1}[\Int(X)]\in H_2(X,\Sing(X);\mathbb{Q}/\mathbb{Z})$. 

A dimension argument implies that $H_2(\Sing(X);\mathbb{Q}/\mathbb{Z}) = 0$ and so the long exact sequence of the pair $X,\Sing(X)$ defines an injection 
\[H_2(X;\mathbb{Q}/\mathbb{Z}) \rightarrow H_2(X,\Sing(X);\mathbb{Q}/\mathbb{Z}).\] 
This map sends $i_X^{-1}[\Int(X)] \mapsto [\Sing(X)]$ which is trivial with $\mathbb{Q}/\mathbb{Z}$ coefficients. Hence we can lift $i_X^{-1}[\Int(X)]$ to a unique class of $H_2(X;\mathbb{Q}/\mathbb{Z})$ which can be thought of as the $\mathbb{Q}/\mathbb{Z}$ \textit{fundamental class} of $X$. We denote the image of this fundamental class in $H_2(M;\mathbb{Q}/\mathbb{Z})$ under the homomorphism induced by the embedding by $[X]$ and say that $X$ represents $[X]$. Turaev proves that every element of $H_2(M;\mathbb{Q}/\mathbb{Z})$ can be represented in this way. A folded surface $X$ can be described as a quotient of the compact $2$-manifold $X_{cut}$ obtained by cutting a simplicial complex representing $X$ along the subcomplex $\Sing(X)$. The quotient map $q:X_{cut}\rightarrow X$ restricts to a homeomorphism $X_{cut}\setminus \partial X_{cut} \xrightarrow{\sim} \Int(X)$ and a surjection $\partial{X_{cut}} \twoheadrightarrow \Sing(X)$. Since $X_{cut}$ is a compact 2-manifold with boundary we can use the definition of $\chi_-$ from Section~\ref{sec:Zthurstonnorm} to define $\chi_-(X) := \chi_-(X_{cut})$ for a general folded surface $X$. The value of $\theta$ on $x\in H_2(M;\mathbb{Q}/\mathbb{Z})$ is defined by allowing $X$ to run over all simple folded surfaces embedded in $M$ representing $x$:

\[\theta(x):= \inf_{[X]=x} i_X^{-1} \chi_-(X)\]

See~\cite{Turaevfnhom3mfld2007} for more details on $\theta$ and folded surfaces.

\subsection{Triangulations}\label{sec:triangulations}

The motivation for studying $\mathbb{Z}_p$-folded surfaces is their potential for proving results on the complexity of orientable $3$-manifolds as defined below. We will be using the definition of a triangulation given in~\cite{jaco20020efficienttriangulations3manifolds} which we highly recommend the unfamiliar reader to refer to for some of the terminology used in Section~\ref{sec:canonicalpfoldedsurfaces}. For convenience we paraphrase it here. Let $\mathbf{\Delta} = \{\Delta_1,\ldots,\Delta_n\}$ be a pairwise disjoint union of (linear) tetrahedra and $\Phi$ be a family of affine orientation reversing isomorphisms between pairs of triangles of the tetrahedra of $\mathbf{\Delta}$. Note that in $\Phi$ there may be pairings of triangles that are faces of the same tetrahedron. A \textit{triangulation} $\mathcal{T}$ is the data $(\mathbf{\Delta},\Phi)$. We denote by $|\mathcal{T}|$ or $\mathbf{\Delta}/\Phi$ the topological space given by the quotient of the disjoint union of the $\Delta_i$ by the relation $\phi(x)\sim x$ for all $\phi:\sigma \rightarrow \sigma' \in \Phi$ and $x\in \sigma$ with the identification topology. We say that $\mathcal{T}$ is a triangulation of an oriented $3$-manifold $M$ if $|\mathcal{T}| \simeq M$. The complexity of $M$ denoted $c(M)$ is the minimal number of tetrahedra in any such triangulation of $M$. This number is known to agree with the notion of complexity given by Matveev in~\cite{Matveev1990complexity} for closed irreducible $3$-manifolds except when $M$ is $S^3$, $\mathbb{R}\mathrm{P}^3$ or $L(3,1)$.

\subsection{Sequences of homology}\label{sec:sequences}

We will make extensive use of two long exact sequences of homology. For a pair $B\subset A$ we will denote the maps in the \textit{long exact sequence of the pair $(A,B)$} as in Equation~\eqref{eq:longpair} where $G$ is $\mathbb{Z}$ or $\mathbb{Z}_p$ in our case.

\begin{equation}\label{eq:longpair}
    \resizebox{0.9\linewidth}{!}{$\displaystyle
        \cdots
        \rightarrow H_i(B;G)
        \xrightarrow{i_{(A,B)}} H_i(A;G)
        \xrightarrow{j_{(A,B)}} H_i(A,B;G)
        \xrightarrow{d_{(A,B)}} H_{i-1}(B;G)
        \rightarrow \cdots
    $}
\end{equation}

The other sequence we will use is the \textit{coefficient long exact sequence} induced by the following short exact sequence:

\[ 0 \rightarrow \mathbb{Z}\xrightarrow{\times p} \mathbb{Z}\xrightarrow{\pi} \mathbb{Z}/p\mathbb{Z} \rightarrow 0.\]

We will denote the maps of this long exact sequence defined by the snake lemma as in Equation~\eqref{eq:longcoeff}.

\begin{equation}\label{eq:longcoeff}
    \cdots\rightarrow H_i(A;\mathbb{Z})\xrightarrow{i_{\mathbb{Z}_p}} H_i(A;\mathbb{Z})\xrightarrow{j_{\mathbb{Z}_p}} H_i(A;\mathbb{Z}_p)\xrightarrow{d_{\mathbb{Z}_p}} H_{i-1}(A;\mathbb{Z})\rightarrow \cdots
\end{equation}

These two long exact sequences are natural with respect to one another.

\section{$\mathbb{Z}_p$-folded Surfaces}\label{sec:pfoldedsurfsoverarch}

We start by describing a subclass of oriented folded surfaces called \textit{$\mathbb{Z}_p$-folded surfaces} as our set of representatives of $H_2(\mathbb{Z}_p)$ classes. We will refine this definition to describe \textit{weighted}, \textit{non-separating} and \textit{reduced} $\mathbb{Z}_p$-folded surfaces. The last of these is closely related to the \textit{normal} $\mathbb{Z}_p$-folded surfaces described in Section~\ref{sec:normalpfoldedsurfacesconstruction}.

\subsection{Definition}\label{sec:pfoldedsurfaces}

\begin{dfn}\label{dfn:pfoldedsurface}
    A \textit{$\mathbb{Z}_p$-folded surface} is an oriented folded surface from Definition~\ref{dfn:foldedsurface} with two additional properties:
    \begin{itemize}
        \item Every point of $\Sing(X)$ has a neighbourhood in $X$ homeomorphic to $\mathbb{R}\times K_{p,1}$. 
        \item The orientation of $X_{cut}$ induces an orientation of $\partial{X_{cut}}$ on the boundary and let $[\partial{X_{cut}}]\in H_1(\partial{X_{cut}};\mathbb{Z})$ be the corresponding fundamental class. The quotient map $q:X_{cut}\rightarrow X$ restricts to a map $\partial{X_{cut}}\rightarrow \Sing(X)$ inducing a homomorphism $H_1(\partial{X_{cut}};\mathbb{Z})\rightarrow H_1(\Sing(X);\mathbb{Z})$. There exists a unique orientation of $\Sing(X)$ such that if $[\Sing(X)]\in H_1(\Sing(X);\mathbb{Z})$ is the corresponding fundamental class then the induced homomorphism sends $[\partial{X_{cut}}] \mapsto p\cdot [\Sing(X)]$. In other words the $p$ sheets incident to each arc of the singular locus all induce a consistent orientation along their identified boundaries. 
    \end{itemize}

\end{dfn}

Note that unlike the simple folded surfaces of~\cite{Turaevfnhom3mfld2007}, the singular locus of a $\mathbb{Z}_p$-folded surface may not be connected. When $p>2$ these $\mathbb{Z}_p$-folded surfaces are a subclass of the \textit{quasi-simple} folded surfaces described in Section 7.1 of~\cite{Turaevfnhom3mfld2007} with $i_X =p$. In general, quasi-simple folded surfaces are related to $\mathbb{Z}_p$-folded surfaces by a surgery in a neighbourhood of $\Sing(X)$. An oriented closed surface $X$ is trivially a $\mathbb{Z}_p$-folded surface for any $p$ with $\Sing(X) = \emptyset$. When $p=1$, a \textit{trivially} folded surface $X$ embedded in a $3$-manifold is a Seifert surface of the oriented link $\Sing(X)$. A closed unoriented surface can be given the structure of a $\mathbb{Z}_2$-folded surface $X$ by making a choice of a closed $1$-manifold $\Sing(X) \subset X$ such that $X\setminus \Sing(X)$ is orientable (and making some choice of orientation). This procedure is described in more detail in the proof of Lemma~\ref{lem:nonorientable}.

\begin{rem}
    A neighbourhood of a singular arc of a $\mathbb{Z}_4$-folded surface appears in a 1965 woodcut print by M.C. Escher titled \textit{Knots}. 
\end{rem}

In this paper we focus on $\mathbb{Z}_p$-folded surfaces where $p$ is prime because this simplifies things, but this restriction is not necessary in general. Definition~\ref{dfn:realisation} and Lemma~\ref{lem:existence} do not require $p$ to be prime and hence the $\mathbb{Z}_p$-Thurston norm is defined for all $p\in \mathbb{Z}^{>0}$. However, Lemma~\ref{lem:components} relies on the primeness of $p$. The definition of a reduced $\mathbb{Z}_p$-folded surface in Section~\ref{sec:reducedpfolded} and the construction of canonical $\mathbb{Z}_p$-folded surface in Section~\ref{sec:canonicalpfoldedsurfaces} both require $p$ to be odd. For this reason we assume that $p$ is a prime greater than $2$ unless otherwise stated from Section~\ref{sec:zpthurstonnorm} onwards.

\subsection{Regular neighbourhoods of singular arcs}\label{sec:regnbhds}

Let $X$ be an oriented folded surface and $\Sing(X)$ be oriented as a $1$-manifold. Since $X$ is oriented then $\partial X_{cut}$ is oriented as described in Definition~\ref{dfn:pfoldedsurface}. Since $\Sing(X)$ is also oriented then for each boundary component $B\in \pi_0(\partial{X_{cut}})$ the restriction $q_B:B\rightarrow q(B)$ of $q: X_{cut}\rightarrow X$ is a covering map between oriented circles and so it has a degree $\deg(q_B)$. Lemma~\ref{lem:alternatedefinition} gives an alternative definition for a $\mathbb{Z}_p$-folded surface.

\begin{lem}\label{lem:alternatedefinition}
    An oriented folded surface $X$ is a $\mathbb{Z}_p$-folded surface if and only if there exists an orientation of $\Sing(X)$ such that:
    \begin{enumerate}
        \item $\sum_{B \in \pi_0(q^{-1}(K))} \deg(q_B) = p$ for all $K\in \pi_0(\Sing(X))$. \label{altdef:sumdegp}
        \item $\deg(q_B)>0$ for all $B\in \pi_0(\partial X_{cut})$. \label{altdef:degpositive}
    \end{enumerate} 
\end{lem}

\begin{proof}
    Fixing any orientation of $\Sing(X)$, the number of sheets meeting along a singular arc $K\in \pi_0(\Sing(X))$ of an oriented folded surface $X$ is $\sum_{B \in \pi_0(q^{-1}(K))} |\deg(q_B)|$ (equal to the number of sheets of the $S^1$ covering $q^{-1}[K]\xrightarrow{q} K$). Let $[B]\in H_1(\partial{X_{cut}})$ be the fundamental class of the boundary component $B$ (with induced orientation) and $[K]$ be the fundamental class of the component $K$ in $H_1(\Sing(X);\mathbb{Z})$. The induced homomorphism $H_1(\partial{X_{cut}};\mathbb{Z})\rightarrow H_1(\Sing(X);\mathbb{Z})$ sends $[B] \mapsto \deg(q_B)[K]$. Note that clearly $[\partial{X_{cut}}] = \sum_{B \in \pi_0(\partial{X_{cut}})}[B]$ and $[\Sing(X)] = \sum_{K\in \pi_0(\Sing(X))} [K]$. Also, $\partial{X_{cut}}$ is a disjoint union of the sets $q^{-1}(K)$ for all $K\in \pi_0(\Sing(X))$. Together this implies that the induced homomorphism sends $[\partial{X_{cut}}] \mapsto \sum_{K\in \pi_0(\Sing(X))}\left[\sum_{B\in \pi_0(q^{-1}(K))} \deg(q_B)\right][K]$. It is now easy to verify from these facts that conditions~\eqref{altdef:degpositive} and~\eqref{altdef:sumdegp} of Lemma~\ref{lem:alternatedefinition} are equivalent to the two properties in Definition~\ref{dfn:pfoldedsurface}.
\end{proof}

\begin{lem}\label{lem:regularneighbourhoods}
    If $X$ is a $\mathbb{Z}_p$-folded surface then for all $K\in \pi_0(\Sing(X))$ a regular neighbourhood of $K$ in $X$ is homeomorphic to $[0,1]\times K_{p,1} / \sim$ where $(1,x) \sim (0,f(x))$ where $f\in \text{Aut}(K_{p,1})= \text{Sym}(p)$. The homeomorphism type is classified by the cycle type of $f$ and equivalently the multiset $\{\deg(q_B)\}_{B\in \pi_0(q^{-1}(K))}$.
\end{lem}

\begin{proof}
    A regular neighbourhood $N(K)$ of $K\in \Sing(X)$ must be homeomorphic to a PL fibre bundle of the graph $K_{p,1}$ over $S^1$. Such a fibre bundle can be assembled using the fact that each boundary component $B$ of $X_{cut}$ collars in $X_{cut}$ and that if $q_{X}(B) = K$ then the restriction of $q_{X}$ to $B$ is a covering map of the circle. Thus, there exists a homeomorphism $f\in \text{Homeo}(K_{p,1})$ with isotopy class $[f] \in \text{Aut}(K_{p,1}) = \pi_0(\text{Homeo}(K_{p,1}))= \text{Sym}(p)$ such that $N(K) \simeq [0,1]\times K_{p,1} / \sim$ where $(1,x) \sim (0,f(x))$. The isomorphism class of this fibre bundle over $S^1$ is determined by the conjugacy class of $[f]\in \pi_0(\text{Homeo}(K_{p,1}))$ which is equivalent to its cycle type as an element of $\text{Sym}(p)$. The identification with the symmetric group is given by labelling the $p$ valence $1$ vertices $\{v_1,\ldots v_p\}$ which are permuted by the graph isomorphism by $v_i\mapsto v_{f(i)}$. We define the cycle type to be the partition of $p$ given by writing an element of $\text{Sym}(p)$ as a product of cycles featuring each of $\{1,\ldots, p\}$ exactly once and taking the multiset of cycle lengths. We will now show that this multiset is independent of the choice of fibre bundle and so determines the homeomorphism type. Let $N_{cut}(K)$ be the disjoint union of annuli given by cutting $N(K)$ along $K$. For each $A \in \pi_0(N_{cut}(K))$ the map $A \rightarrow N(K)$ induces the homomorphism $H_1(A;\mathbb{Z}) \rightarrow H_1(N(K);\mathbb{Z})$ given by $a\mapsto k_A[K]$ for some choice of generating class $\angb{a} = H_1(A;\mathbb{Z})$ and a positive integer $k_A$. The multiset $\{k_A\}_{A\in \pi_0(N\setminus K)}$ is equivalent to the cycle type of $f$ since $k_A$ is determined by the size of $\{v_1,\ldots v_p\} \cap \text{im}(A)$. $N_{cut}(K)$ is homeomorphic to a collar neighbourhood of $q^{-1}(K) \subset \partial{X_{cut}}$ and the multiset $\{k_A\}_{A\in \pi_0(N\setminus K)}$ can also be computed from the degrees of the maps $B \xrightarrow{q} q(B)$ for each $B$ in the preimage of $K$. This gives the desired classification of the homeomorphism type of $N(K)$ by the multiset $\{\deg(q_B)\}_{B\in \pi_0(q^{-1}(K))}$.

\end{proof}

\subsection{Classification using decorated graphs}\label{sec:classification}

We will now give a classification of connected $\mathbb{Z}_p$-folded surfaces up to homeomorphisms. Note that homeomorphisms of $\mathbb{Z}_p$-folded surfaces are homeomorphisms of the pair $X,\Sing(X)$. When $p\neq 2$, points on the singular locus have different local neighbourhoods to points in $\Int(X)$ and so any homeomorphism of the underlying space of $X$ will restrict to a homeomorphism of $\Sing(X)$ and this qualification can be dropped. Homeomorphisms of connected $\mathbb{Z}_p$-folded surfaces are either orientation preserving or orientation reversing on all of $\Int(X)$ and $\Sing(X)$. This depends on Definition~\ref{dfn:pfoldgraph} which describes a type of decorated graph which we call a \textit{$p$-fold graph}.

\begin{dfn}\label{dfn:pfoldgraph}
    A \textit{$p$-fold graph} $G$ is a finite bipartite multi-graph 
    \[ G = (I,S,E,\phi,\psi,d,g)\]
    with vertex sets $I,S$, edges $E$ and edge endpoints given by the functions $\phi:E\rightarrow I$ and $\psi:E\rightarrow S$ together with edge labels $d:E\rightarrow \mathbb{Z}^{>0}$ and vertex labels $g:I\rightarrow \mathbb{Z}^{\geq 0}$ such that $\sum_{e\in \psi^{-1}(s)} d(e) =p$. An isomorphism of $p$-fold graphs is an isomorphism of bipartite multi-graphs that preserves labels and partitions.
\end{dfn}

Now we associate a $p$-fold graph $G(X)$ to any $\mathbb{Z}_p$-folded surface $X$ in the following way:

\begin{dfn}\label{def:assocgraph}
    Let $G(X) := (I,S,E,\phi,\psi,d,g)$ be the $p$-fold graph where:
    \begin{itemize}
        \item $I := \pi_0(X_{cut})$
        \item $S := \pi_0(\Sing(X))$
        \item $E = \pi_0(\partial X_{cut})$, 
        \item For every boundary component $B$ of $X_{cut}$ define the following:
        \begin{itemize}
            \item $\phi(B)$ is the unique component of $X_{cut}$ which cobounds $B$
            \item $\psi(B) := q_B(B) \in \pi_0(\Sing(X))$
            \item $d(B):= \deg(q:B\rightarrow q(B))$
        \end{itemize}
        \item For every component $Y$ of $X_{cut}$ let $g(Y)$ be the (orientable) genus of $Y$.
    \end{itemize}
\end{dfn}

\begin{lem}
    Let $p>0$. The homeomorphism class of any $\mathbb{Z}_p$-folded surface $(X,\Sing(X))$ is determined by the isomorphism class of $G(X)$.
\end{lem}

\begin{proof}

    The associated graph described in Definition~\ref{def:assocgraph} depends only on the topology of a $\mathbb{Z}_p$-folded surface and is independent of the choice of orientation. Hence homeomorphic $\mathbb{Z}_p$-folded surfaces have isomorphic $p$-fold graphs.

    To prove the converse, suppose that we have two $\mathbb{Z}_p$-folded surfaces $X, \Sing(X)$ and $Y,\Sing(Y)$ and an isomorphism $f:G(X) \rightarrow G(Y)$. Let $G(X) := (I,S,E,\phi,\psi,d,g)$ and $G(Y) := (I',S',E',\phi',\psi',d',g')$. Fix an orientation preserving homeomorphism $h:\Sing(X)\rightarrow \Sing(Y)$ such that for all $K\in \pi_0(\Sing(X))$ we have $f(K) = \text{im}(h|_K)$. Now let $q_X:X_{cut}\rightarrow X$ and $q_Y:Y_{cut}\rightarrow Y$ be the two projection maps described in Section~\ref{sec:regnbhds}. Since $h$ is orientation preserving and $f$ is a $p$-fold graph isomorphism we have that $\deg(q_X|_{B}) = \deg(q_Y|_{f(B)})$ and $\text{im}(q_Y|_{f(B)}) = \text{im}(h\circ q_X|_{B})$ so we can use the lifting criterion for each component to construct a homeomorphism $\tilde{h}:\partial X_{cut}\rightarrow \partial{Y_{cut}}$ which commutes with the projection maps for all $B\in \pi_0(\partial{X_{cut}})$:

\begin{center}
    \begin{tikzcd}
        B \ar[r, "\tilde{h}"] \ar[d, "q_{B}"] 
        & 
        f(B) \ar[d, "q_{f(B)}"] 
        \\
        \psi(B) \ar[r, "h"] 
        & 
        \psi'(f(B)) = f(\psi(B))
    \end{tikzcd}
\end{center}

    Note that the definition of a $p$-fold graph isomorphism implies that $\psi'\circ f= f\circ \psi $, $\phi'\circ f  = f\circ \phi$, $g' \circ f = g$ and $d' \circ f = d$. The second fact tells us that if we take $X_i\in \pi_0(X_{cut})$ and $Y_i = f(X_i) \in \pi_0(Y_{cut})$ then $\tilde{h}$ restricts to a homeomorphism $\tilde{h}:\partial{X_i}\rightarrow \partial{Y_i}$ preserving the induced orientation on the boundary. The third fact tells us that the genus of $X_i$ is the same as the genus of $Y_i$. It is well known (for example see Theorem 1 of~\cite{youngs:1948}) that $\tilde{h}$ can be extended to a homeomorphism from $X_i\rightarrow Y_i$ and hence a homeomorphism $\tilde{h}:X_{cut}\rightarrow Y_{cut}$. By construction this descends to a homeomorphism $h:X\rightarrow Y$ which extends the original $h:\Sing(X)\rightarrow \Sing(Y)$ over the interior.
    
\end{proof}

Figure~\ref{fig:pcrosscaps} shows a gluing diagram and also the associated graph for what is arguably the simplest non-trivial example of a $\mathbb{Z}_p$-folded surface which we will refer to as the \textit{$\mathbb{Z}_p$-crosscap} which generalises the usual crosscap $\mathbb{R}\mathrm{P}^2$ when $p=2$. Figure~\ref{fig:example3fold} shows another example of a $\mathbb{Z}_3$-folded surface and its associated graph.

\begin{figure}
    \centering
    \includegraphics[width=0.4\textwidth]{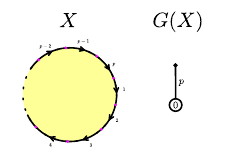}
    \caption{The $\mathbb{Z}_p$-crosscap and its associated $p$-fold graph}\label{fig:pcrosscaps}
\end{figure}

\begin{figure}
    \centering
    \includegraphics[width=0.4\textwidth]{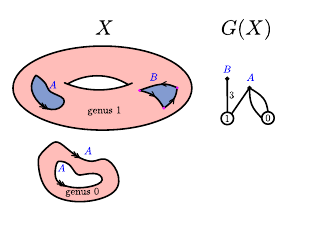}
    \caption{An example $\mathbb{Z}_3$-folded surface with two interior components and two singular arcs labelled $A$ and $B$. $A$ is \textit{straight} and $B$ is \textit{twisted}.}\label{fig:example3fold}
\end{figure}

$\mathbb{Z}_p$-folded surfaces that can be embedded in $3$-manifolds will be of particular interest in this research. Lemma~\ref{lem:classificationin3mflds} characterises these \textit{embeddable} $\mathbb{Z}_p$-folded surfaces. Throughout the remainder of this paper all spaces and embeddings will be \textit{piecewise-linear} unless otherwise stated.

\begin{lem}\label{lem:classificationin3mflds}
    Let $X$ be a $\mathbb{Z}_p$-folded surface such that $p$ is prime and let $G(X) = (I,S,E,\phi,\psi,d,g)$ be its associated graph. $X$ can be embedded in an orientable $3$-manifold if and only if $d(B) \in \{1,p\}\ \forall\ B\in E$.
\end{lem}

\begin{proof}

    First, let $X$ be a $\mathbb{Z}_p$-folded surface embedded in a $3$-manifold $M$ and let $S$ be any singular arc of $X$. We will prove that for every component $B$ of $q^{-1}(S)$ the degree of $B\rightarrow S$ is $1$ or $p$. 

    Let $N(S)$ be a regular neighbourhood of $S$ in $M$ such that $N(S)\cap X$ is a regular neighbourhood of $S$ in $X$. By gluing together the product structure $(K_{p,1},\mathbb{D}^2)\times \mathbb{R}$ of neighbourhoods of points on $S$ we can assemble a fibre bundle $\mathbb{D}^2 \rightarrow N(S) \rightarrow S^1$ with sub-bundle $K_{p,1}\rightarrow N(S)\cap X \rightarrow S^1$. The restriction of the embedding $N(S)\cap X \rightarrow N(S)$ on fibres is the embedding $K_{p,1}\subset \mathbb{D}^2$ in Figure~\ref{fig:subbundle}.
    
    The embedding $N(S)\cap X \hookrightarrow N(S)$ is therefore homeomorphic to an embedding of mapping tori $K_{p,1}\times[0,1]/\sim \hookrightarrow \mathbb{D}^2 \times [0,1]/\sim$ where $(1,x) \sim (0,f(x))$ and $f$ is a homeomorphism of $\mathbb{D}^2$ which restricts to a homeomorphism of $K_{p,1}$. The orientability of $M$ forces $f$ to be orientation preserving. The embedding $K_{p,1}\hookrightarrow \mathbb{D}^2$ forces $f|_{K_{p,1}}$ to preserve the cyclic order of the vertices of $K_{p,1}\cap \partial{\mathbb{D}^2}$. The graph isomorphism class $[f|_{K_{p,1}}]\in \text{Sym}(p)$ must therefore be a cyclic permutation of order dividing $p$. If $[f|_{K_{p,1}}]=id$ then $\forall B \in \pi_0(q^{-1}(S))$ the map $q_B:B \rightarrow S$ has degree $1$, and we will call $S$ \textit{straight} if this is the case. If $[f|_{K_{p,1}}]\neq id$ then its order is $p$ and so $\pi_0(q^{-1}(S)) = \{B\}$, $q_B:B \rightarrow S$ has degree $p$ and we call $S$ \textit{twisted}.

    Now suppose $X$ is such a $\mathbb{Z}_p$-folded surface with the property that $d(e) \in \{1,p\}\ \forall\ e\in E$. Let $S\in \pi_0(\Sing(X))$ and $N(S)$ be a regular neighbourhood of $S$ in $X$. The fact that $\sum_{e\in \psi^{-1}(S)} d(e) =p$ implies that either $|\psi^{-1}(S)| = p$ and $d(B)= 1\ \forall\ B\in \psi^{-1}(S)$ ($S$ is \textit{straight}) or $\psi^{-1}(S) = \{B\}$ and $d(B)=p$ ($S$ is \textit{twisted}). Consider the embedding $K_{p,1} \simeq \{te^{2i\pi k/p}, k \in \{0,1,\ldots,p-1\}, t \in [0,1]\}$ in the unit disk $\mathbb{D}^2$. Lemma~\ref{lem:regularneighbourhoods} implies that $N(S)$ is homeomorphic to $[0,1]\times K_{p,1}/\sim$ where $(1,x) \sim (0,f(x))$ and $f:K_{p,1}\rightarrow K_{p,1}$ is the restriction of $id\in \text{Aut}(\mathbb{C})$ if $S$ is straight or $z \mapsto ze^{2i\pi/p}$ if $S$ is twisted. In both cases we can use the embedding of $K_{p,1}\hookrightarrow \mathbb{D}^2$ to construct an embedding of $N(S)\hookrightarrow T(S)$ where $T(S)$ is the solid torus $[0,1]\times \mathbb{D}^2/\sim$ where $(1,x) \sim (0,f(x))$. Now repeat this for all $S\in \pi_0(\Sing(X))$ and let $N(X) := \bigsqcup_{S\in \pi_0(\Sing(X))}N(S)$ and $T(X):= \bigsqcup_{S\in \pi_0(\Sing(X))}T(S)$. Let $\Sigma(X)$ be the compact orientable $2$-manifold which is the closure in $X$ of $X \setminus \bigsqcup_{S\in \pi_0(\Sing(X))} N(S)$. $\partial{\Sigma(X)}$ is a $1$-manifold identified in $X$ with $\partial{N(X)} = \partial{T(X)} \cap N(X)$. Thicken $[-1,1]\times \Sigma(X)$ such that $\Sigma(X) = \{0\} \times \Sigma(X)$. We then extend the homeomorphism identifying $\{0\} \times \partial \Sigma(X)$ with $N(X) \cap \partial T(X)$ to a homeomorphism identifying $[-1,1]\times \partial \Sigma(X)$ with the disjoint union of annuli that form a regular neighbourhood of $N(X) \cap \partial T(X)$ in the disjoint union of tori $\partial T(X)$. The result is an embedding of $X$ in the compact orientable $3$-manifold with boundary $T(X)\sqcup [-1,1]\times \Sigma(X)/\sim$ with image $N(X)\sqcup \Sigma(X)/\sim$, where $\sim$ denotes the identification described above.  

\end{proof}

\begin{figure}
    \centering
    \includegraphics[width=0.6\textwidth]{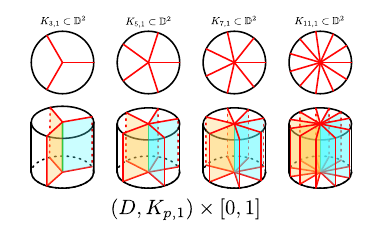}
    \caption{The embedding of $K_{p,1}\hookrightarrow \mathbb{D}^2$ that is the fibrewise restriction of the embedding $N(S)\cap X \subset N(S)$ respecting the fibre bundle structures $\mathbb{D}^2 \rightarrow N(S)\rightarrow S^1$ and $K_{p,1}\rightarrow N(S)\cap X \rightarrow S^1$.}\label{fig:subbundle}
\end{figure}

Section~\ref{sec:realisation} defines the class in $H_2(M;\mathbb{Z}_p)$ represented by a $\mathbb{Z}_p$-folded surface embedded in an orientable 3-manifold $M$ (Definition~\ref{dfn:realisation}) and proves the existence of such representatives for all $x \in H_2(M;\mathbb{Z}_p)$ (Lemma~\ref{lem:existence}).

\subsection{Representing $H_2(\mathbb{Z}_p)$ classes as $\mathbb{Z}_p$-folded surfaces}\label{sec:realisation}

Let $X$ be a $\mathbb{Z}_p$-folded surface. The map $q:(X_{cut},\partial X_{cut})\rightarrow (X,\text{Sing}(X))$ induces a map $q_*$ on homology. Since $X_{cut}$ is a $2$-manifold with boundary $\partial X_{cut}$ that has an orientation given by lifting the orientation of $\Int(X)$, it has a fundamental class $[X_{cut}]\in H_2(X_{cut},\partial X_{cut};\mathbb{Z})$. Similarly since $\Sing(X)$ is an oriented $1$-manifold it too has a fundamental class $[Sing(X)]$. Let $d_{(X,\Sing(X))}$ be the connecting homomorphism of the long exact sequence of relative homology. It follows from the naturality of this sequence that $d_{(X,\Sing(X))}(q_*([X_{cut}])) = q_*(d_{(X_{cut},\partial X_{cut})}([X_{cut}])) = q_*([\partial X_{cut}])$. Lemma~\ref{lem:alternatedefinition} implies that $q_*([\partial X_{cut}]) = p[\Sing(X)]$. It follows that if we change coefficients via the standard quotient $j_{\mathbb{Z}_p}:\mathbb{Z}\rightarrow \mathbb{Z}_p$ then $j_{\mathbb{Z}_p}\circ q_*([\partial X_{cut}]) = 0$. Hence it follows from the naturality of the change of coefficient long exact sequence and the relative homology long exact sequence that $j_{\mathbb{Z}_p} \circ q_*([X_{cut}])$ is in the kernel of $d_{(X,\Sing X)}:H_2(X,\Sing X ;\mathbb{Z}_p)\rightarrow H_1(\Sing X;\mathbb{Z}_p)$ and so must be in the image of $H_2(X;\mathbb{Z}_p)\rightarrow H_2(X,\Sing(X);\mathbb{Z}_p)$. A dimension argument implies that $H_2(\Sing(X);\mathbb{Z}_p)=0$ and so this map is injective. Let $[X]$ be the unique element of $H_2(X;\mathbb{Z}_p)$ such that $[X]\mapsto j_{\mathbb{Z}_p}\circ q_*([X_{cut}])$. We define $[X]$ to be the \textit{fundamental class} of $X$. If $X$ is embedded in $M$ then we may abuse notation by using $[X]$ to denote the image of the fundamental class of $X$ under the map on homology $H_2(X;\mathbb{Z}_p)\rightarrow H_2(M;\mathbb{Z}_p)$ induced by the inclusion $X\hookrightarrow M$. Similarly we may use $[\Sing(X)]$ to denote the image of the fundamental class of $\Sing(X)$ under the map $H_1(X;\mathbb{Z})\rightarrow H_1(M;\mathbb{Z})$ induced by the restriction of the embedding map $X\hookrightarrow M$.

\begin{dfn}\label{dfn:realisation}
    A $\mathbb{Z}_p$-folded surface $X$ embedded in $M$ \textit{represents} $x\in H_2(M;\mathbb{Z}_p)$ if $[X] = x$.
\end{dfn}

The coefficient short exact sequence $\mathbb{Z}\xrightarrow{i} \mathbb{Z} \xrightarrow{j} \mathbb{Z}_p$ where $i(x) = px$ and $j$ is the standard quotient $\mathbb{Z}\rightarrow \mathbb{Z}/p\mathbb{Z}$ induces \textit{the coefficient long exact sequence} by the snake lemma. We will in general use $d_{\mathbb{Z}_p}:H_i(\underline{};\mathbb{Z}_p)\rightarrow H_{i-1}(\underline{};\mathbb{Z})$ to denote the connecting homomorphism of this sequence. 

If $X$ is a $\mathbb{Z}_p$-folded surface representing a class $x\in H_2(M;\mathbb{Z}_p)$ then it follows from Definition~\ref{dfn:pfoldedsurface} that $d_{\mathbb{Z}_p}(x) = [\Sing(X)]$. We will now prove the following statement by constructing a $\mathbb{Z}_p$-folded surface representing an arbitrary class $x$.

\begin{lem}[Existence]
    Let $M$ be a compact orientable $3$-manifold. Every homology class $x\in H_2(M;\mathbb{Z}_p)$ can be represented by a $\mathbb{Z}_p$-folded surface.
    \label{lem:existence}
\end{lem}

\begin{proof}

    The main idea for this proof comes from Lemma 2.1 of~\cite{Turaevfnhom3mfld2007}, which we supplement with a surgery on a regular neighbourhood of $\Sing(X)$ required to obtain a $\mathbb{Z}_p$-folded surface rather than a general simple folded surface. We also relax part of Turaev's construction by allowing $\Sing(X)$ to be any oriented link representing $d_{\mathbb{Z}_p}(x)$ rather than requiring a knot. This change is not strictly necessary for existence but it allows Lemma~\ref{lem:thurstonnorm} in Section~\ref{sec:zpthurstonnorm} to follow relatively easily. Existence can also be shown constructively using a technique analogous to normal surface theory which we develop in Section~\ref{sec:canonicalpfoldedsurfaces}.
    
    We start by choosing an oriented link $L$ embedded via $l:L\rightarrow M$ such that $i_{(M,L)}([L]) = d_{\mathbb{Z}_p}(x)$. Note that the coefficient long exact sequence and the long exact sequence of the pair are both natural. For the reader's convenience we assemble the part of the large commutative diagram that they form together which is relevant to the proof in Equation~\eqref{cd:bigcd}.

\begin{equation}\label{cd:bigcd}
    \resizebox{0.88\linewidth}{!}{%
        $\begin{tikzcd}[
            cramped,
            column sep=small,
            row sep=normal,
            ampersand replacement=\&
        ]
            H_2(L;\mathbb{Z})
                \ar[d,"i_{\mathbb{Z}_p}"]
                \ar[r,"d_{(M,L)}"]
            \&
            H_2(M;\mathbb{Z})
                \ar[d,"i_{\mathbb{Z}_p}"]
                \ar[r,"j_{(M,L)}"]
            \&
            H_2(M,L;\mathbb{Z})
                \ar[d,"i_{\mathbb{Z}_p}"]
                \ar[r,"d_{(M,L)}"]
            \&
            H_1(L;\mathbb{Z})
                \ar[d,"i_{\mathbb{Z}_p}"]
                \ar[r,"i_{(M,L)}"]
            \&
            H_1(M;\mathbb{Z})
                \ar[d,"i_{\mathbb{Z}_p}"]
            \\
            H_2(L;\mathbb{Z})
                \ar[d,"j_{\mathbb{Z}_p}"]
                \ar[r,"d_{(M,L)}"]
            \&
            H_2(M;\mathbb{Z})
                \ar[d,"j_{\mathbb{Z}_p}"]
                \ar[r,"j_{(M,L)}"]
            \&
            H_2(M,L;\mathbb{Z})
                \ar[d,"j_{\mathbb{Z}_p}"]
                \ar[r,"d_{(M,L)}"]
            \&
            H_1(L;\mathbb{Z})
                \ar[d,"j_{\mathbb{Z}_p}"]
                \ar[r,"i_{(M,L)}"]
            \&
            H_1(M;\mathbb{Z})
                \ar[d,"j_{\mathbb{Z}_p}"]
            \\
            H_2(L;\mathbb{Z}_p)
                \ar[r,"d_{(M,L)}"]
            \&
            H_2(M;\mathbb{Z}_p)
                \ar[r,"j_{(M,L)}"]
                \ar[uurrr,
                    rounded corners=4pt,
                    "d_{\mathbb{Z}_p}"{pos=0.75},
                    to path={
                        (\tikztostart.south)
                        -- ++(0,-8mm)
                        -| ([xshift=10mm]\tikztotarget.east)
                        \tikztonodes
                        -- (\tikztotarget.east)
                    }
                ]
            \&
            H_2(M,L;\mathbb{Z}_p)
                \ar[r,"d_{(M,L)}"]
            \&
            H_1(L;\mathbb{Z}_p)
                \ar[r,"i_{(M,L)}"]
            \&
            H_1(M;\mathbb{Z}_p)
        \end{tikzcd}$%
    }
\end{equation}

    The exactness of the coefficient long exact sequence implies that $0 = i_{\mathbb{Z}_p}\circ d_{\mathbb{Z}_p}(x) = i_{\mathbb{Z}_p} \circ i_{(M,L)}([L])$.
    The coefficient long exact sequence and the relative homology long exact sequence are natural and so $i_{(M,L)} \circ i_{\mathbb{Z}_p}([L]) = 0$. The exactness of the long exact sequence of the pair implies that there exists a class $x' \in H_2(M,L;\mathbb{Z})$ such that $d_{(M,L)}(x') = i_{\mathbb{Z}_p}([L]) = p\cdot [L]$. A dimension argument implies that $H_2(L;\mathbb{Z}_p)=0$ so it follows that $j_{(M,L)}:H_2(M;\mathbb{Z}_p)\rightarrow H_2(M,L;\mathbb{Z}_p)$ is injective. The naturality of the two sequences implies that $d_{(M,L)}(j_{\mathbb{Z}_p}(x')) = j_{\mathbb{Z}_p}(d_{(M,L)}(x')) = j_{\mathbb{Z}_p}\circ i_{\mathbb{Z}_p}([L]) = 0$ and so there is a unique lift $x'' := j_{(M,L)}^{-1}(j_{\mathbb{Z}_p}(x')) \in H_2(M;\mathbb{Z}_p)$ of $j_{\mathbb{Z}_p}(x')$. 
    
    It follows from a chase over the commutative diagram~\eqref{cd:bigcd} or alternatively from the definition of the connecting homomorphism that $d_{\mathbb{Z}_p}(x'') = d_{\mathbb{Z}_p}(x) = [L]$. Therefore there must exist some $y\in H_2(M;\mathbb{Z})$ such that $x = j_{\mathbb{Z}_p}(y) + x''$. Let $x''' := j_{(M,L)}(y) + x' \in H_2(M,L;\mathbb{Z})$ such that $j^{-1}_{(M,L)}(j_{\mathbb{Z}_p}(x''')) = x$. 

    Take an open regular neighbourhood $N$ of $L$. By excision and deformation we can identify $H_2(M,L;\mathbb{Z})= H_2(M,\overline{N};\mathbb{Z}) = H_2(M\setminus N,\partial(N);\mathbb{Z})$. It is well known that every integral 2-homology class of a compact oriented 3-manifold such as $M \setminus N$ is represented by a properly embedded oriented surface. Let $S$ be such a surface where $[S] = x''' \in H_2(M\setminus N(L),\partial(N(L));\mathbb{Z})$. Consider the maps on homology induced by the embedding $f:(M\setminus N,\partial(N)) \hookrightarrow (M,\overline{N})$ and the naturality of the long exact sequence of relative homology.

    \begin{center}
        \begin{tikzcd}
        H_2(M\setminus N,\partial(N);\mathbb{Z}) \ar[r,"d"] \ar[d,"f_*"] 
        & H_1(\partial(N);\mathbb{Z}) \ar[d,"f_*"] 
        \\
        H_2(M,\overline{N};\mathbb{Z}) \ar[r,"d"]
        & H_1(\overline{N};\mathbb{Z})
        \end{tikzcd}
    \end{center}

    $\partial(N)$ is a disjoint union of tori and we can choose generators $\lambda_k$ and $\mu_k$ represented by simple closed curves on $\partial(N_k)$ for each component $N_k$ of $N$. Let $L_k$ be the component of $L$ that is a core circle of $N_k$. Let $\mu_k, \lambda_k$ be generators of $H_1(\partial(N_k);\mathbb{Z})$ such that $\mu_k$ generates $\ker(f_*)\cap H_1(\partial(N_k);\mathbb{Z})$ and $f_*(\lambda_k) = [L_k]$. The above commutative square can be used to show that $\partial(S)\cap \partial(N_k)$ is a union of oriented simple closed curves with homology class $p_k\lambda_k + q_k\mu_k \in H_1(\partial(N_k);\mathbb{Z})$. Null-homologous curves on the torus bound a disk and so can be eliminated by pasting a disk parallel to $\partial N_k$. After all null-homologous curves are removed $\partial(S)\cap \partial(N_k)$ is a disjoint union of simple closed curves of equal slope, a consequence of the fact that the intersection form is trivial on primitive non-trivial $s_1,s_2 \in H_1(\mathbb{T}^2;\mathbb{Z})$ classes only if $s_1 = \pm s_2$. If the curves are not oriented consistently then there must exist an annulus embedded in $\partial(N_k)$ cobounding two curves of opposite orientation. We can again perform a local surgery to $S$ by gluing an annulus with the required orientation in the interior of $M\setminus N$ parallel to the annulus in the boundary to remove these two components of $\partial(S)$ without changing the homology class represented by $S$. Once these surgeries are complete it follows that $\partial S \cap \partial N_k$ is a $(p_k,q_k)$-torus link where every component represents the same homology class. This torus link is unique up to an isotopy of $\partial N_k$ which the PL isotopy extension theorem allows us to extend to an isotopy of $N_k$. We can therefore obtain a $\mathbb{Z}_p$-folded surface $X$ by setting $X\cap M\setminus N := S$ and $X\cap N_k \subset N_k$ to be a PL-approximation of the embedding where $N_k = \mathbb{C} \times [0,1]/\sim$ where $(z,0)\sim (z,1)$ and $X\cap N_k = \bigcup_{j=0}^{p} \{x e^{(j+q_{k}t)/2ip_{k}\pi}\ |\ x \in \mathbb{R}, t\in [0,1]\}/\sim$. This model neighbourhood of a singular arc is illustrated in Figure~\ref{fig:modelsingularnbhd}. 

    $x'''$ is the image of the fundamental class of $(X_{cut},\partial X_{cut})$ in $(M,\Sing(X))$. By definition the fundamental class of $X$ is the unique lift of $j_{(M,L)}^{-1}(j_{\mathbb{Z}_p}([X_{cut}]))$. It follows that the image of this fundamental class in $M$ is $j_{(M,L)}^{-1}(j_{\mathbb{Z}_p}(x''')) = x$ completing the proof.

\end{proof}

\begin{figure}[htbp]
    \centering
    \setlength{\tabcolsep}{2pt}
    \renewcommand{\arraystretch}{0.8}

    \begin{tabular}{c@{}|@{}c@{}|@{}c@{}|@{}c@{}|@{}c@{}|@{}c@{}|@{}c@{}|@{}c}
        $(3,-2)$ & $(3,0)$ & $(3,1)$ & $(3,7)$ &
        $(5,-2)$ & $(5,0)$ & $(5,1)$ & $(5,7)$ \\ \hline
        \includegraphics[width=0.12\textwidth]{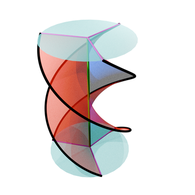} &
        \includegraphics[width=0.12\textwidth]{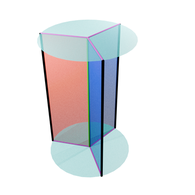} &
        \includegraphics[width=0.12\textwidth]{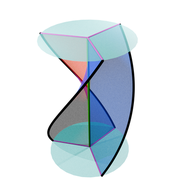} &
        \includegraphics[width=0.12\textwidth]{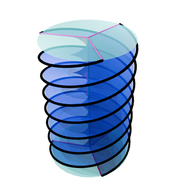} &
        \includegraphics[width=0.12\textwidth]{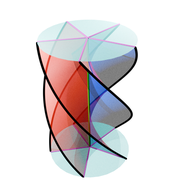} &
        \includegraphics[width=0.12\textwidth]{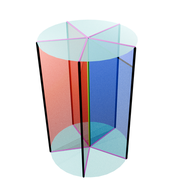} &
        \includegraphics[width=0.12\textwidth]{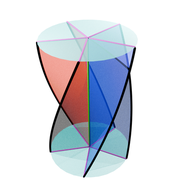} &
        \includegraphics[width=0.12\textwidth]{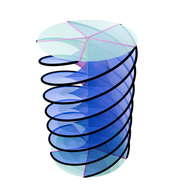} \\

        \includegraphics[width=0.12\textwidth]{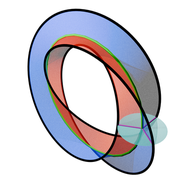} &
        \includegraphics[width=0.12\textwidth]{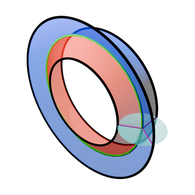} &
        \includegraphics[width=0.12\textwidth]{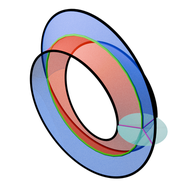} &
        \includegraphics[width=0.12\textwidth]{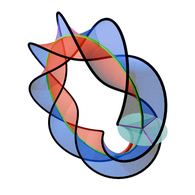} &
        \includegraphics[width=0.12\textwidth]{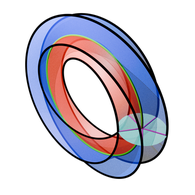} &
        \includegraphics[width=0.12\textwidth]{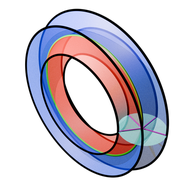} &
        \includegraphics[width=0.12\textwidth]{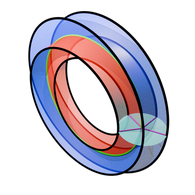} &
        \includegraphics[width=0.12\textwidth]{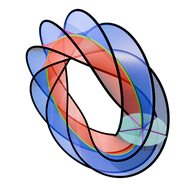} \\

        \includegraphics[width=0.12\textwidth]{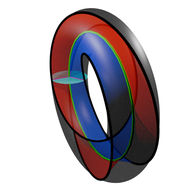} &
        \includegraphics[width=0.12\textwidth]{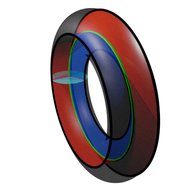} &
        \includegraphics[width=0.12\textwidth]{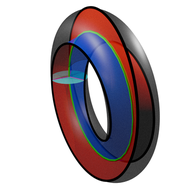} &
        \includegraphics[width=0.12\textwidth]{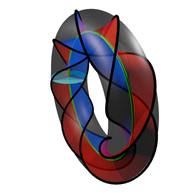} &
        \includegraphics[width=0.12\textwidth]{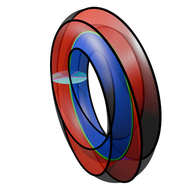} &
        \includegraphics[width=0.12\textwidth]{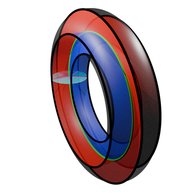} &
        \includegraphics[width=0.12\textwidth]{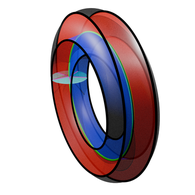} &
        \includegraphics[width=0.12\textwidth]{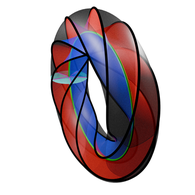}
    \end{tabular}

    \caption{Solid torus neighbourhoods of singular arcs in $\mathbb{Z}_p$-folded surfaces that intersect the boundary torus in a $(p,q)$-torus link (in black). Examples where $p \in \{3,5\}$ and $q \in \{-2,0,1,7\}$. The first row shows the lift of the embedding $X\cap N_k\subset N_k$ to the solid cylinder $[0,1]\times \mathbb{C}$. The second row shows the embedding $X\cap N_k\hookrightarrow N_k$ where the solid torus $N_k$ is embedded in $\mathbb{R}^3$ as a regular neighbourhood of the unknot singular arc in green. $\partial N_k$ is shaded in the third row.}
    \label{fig:modelsingularnbhd}
\end{figure}

\section{The $\mathbb{Z}_p$-Thurston norm}\label{sec:zpthurstonnorm}

We are inspired by the $\mathbb{Z}$-Thurston norm, $\mathbb{Z}_2$-Thurston norm and Turaev's $\theta$ to define the \textit{$\mathbb{Z}_p$-Thurston norm} in Definition~\ref{dfn:zpthurstonnorm}. Here, $X$ runs over all $\mathbb{Z}_p$-folded surfaces embedded in $M$ realising $x\in H_2(M;\mathbb{Z}_p)$ and $\chi_-$ is defined for folded surfaces as in Section~\ref{sec:theta}.

\begin{dfn}\label{dfn:zpthurstonnorm}
    Let $M$ be a closed $3$-manifold. For each $p\in \mathbb{Z}^{>0}$ the $\mathbb{Z}_p$-Thurston norm is a function $||\cdot||_{\mathbb{Z}_p}:H_2(M;\mathbb{Z}_p)\rightarrow \mathbb{Z}^{\geq 0}$ defined for each class $x \in H_2(M;\mathbb{Z}_p)$ by the following expression, allowing $X$ to run over all $\mathbb{Z}_p$-folded surfaces embedded in $M$ that realise $x$:
    
    \[||x||_{\mathbb{Z}_p} := \min_{[X]=x}(\chi_-(X))\]
\end{dfn}

This definition is sufficient for closed 3-manifolds, but there is a natural extension to a function on compact oriented $3$-manifolds with boundary by instead letting $X$ be a \textit{$\mathbb{Z}_p$-folded surface with boundary properly embedded in $M$} - see Definition~\ref{dfn:pfoldedwbdry}. Lemma~\ref{lem:nonorientable} shows that $||\cdot||_{\mathbb{Z}_p}$ generalises the $\mathbb{Z}_2$-Thurston norm of~\cite{jaco2009z2thurstonnormcomplexity3manifolds}.

\begin{lem}\label{lem:nonorientable}
    Definition~\ref{dfn:z2thurstonnorm} and Definition~\ref{dfn:zpthurstonnorm} are equivalent when $p=2$ and $M$ is a closed $3$-manifold.
\end{lem}

\begin{proof}
    We will use $||\cdot||_{\mathbb{Z}_2}$ to refer to the function in Definition~\ref{dfn:z2thurstonnorm} and $||\cdot||_{\mathbb{Z}_2}'$ to refer to the function in Definition~\ref{dfn:zpthurstonnorm}
    
    Suppose $\Sigma$ is an unoriented surface realising $x$ such that 
    \[||x||_{\mathbb{Z}_2} = \sum_{\Sigma_0\in \pi_0(\Sigma)} \max(0, -\chi(\Sigma_0)).\] 
    First we choose a system of disjoint closed curves suggestively named $\Sing(\Sigma)$, embedded in $\Sigma$ such that $\Sigma \setminus \Sing(\Sigma)$ is an orientable $2$-manifold with the same number of connected components as $\Sigma$. This can be achieved by the following induction. We start with $\Sing(\Sigma) = \emptyset$. At each step if $\Sigma \setminus \Sing(\Sigma)$ is non-orientable, it is possible to find a simple closed curve $C$ whose neighbourhood in $\Sigma \setminus \Sing(\Sigma)$ is a Mobius band. Being the core circle of a Mobius band, $\Sigma \setminus \Sing(\Sigma) \setminus C$ must have the same number of components as $\Sigma \setminus \Sing(\Sigma)$. Hence we add $C$ to $\Sing(\Sigma)$ and repeat. Since $2 |\pi_0(\Sigma)| - |\Sing(\Sigma)|$ gives an upper bound on $\chi(\Sigma)$ it follows that this induction must terminate with $\Sing(\Sigma)$ having the desired properties. Now choose some arbitrary orientation of $\Sigma \setminus \Sing(\Sigma)$ and view it as a $\mathbb{Z}_2$-folded surface. The components of $\Sigma \setminus \Sing(\Sigma)$ respectively have the same Euler characteristic as the components of $\Sigma$ that contain them. Hence the following is true:
    
    \[||x||_{\mathbb{Z}_2} = \sum_{\Sigma_0\in \pi_0(\Sigma)} \max(0, -\chi(\Sigma_0)) = \sum_{\Sigma_0' \in \pi_0(\Sigma \setminus \Sing(\Sigma))} \max(0 , -\chi(\Sigma_0')) \geq ||x||_{\mathbb{Z}_2}'.\]

    Since the class $[\Sigma] \in H_2(M;\mathbb{Z}_2)$ does not depend on whether $\Sigma$ is viewed as a $\mathbb{Z}_2$-folded surface or an unoriented surface it follows that $||x||_{\mathbb{Z}_2} \geq ||x||_{\mathbb{Z}_2}'$. 
    
    For the other inequality, let $X$ be a $\mathbb{Z}_2$-folded surface embedded in a closed orientable 3-manifold such that $\chi_-(X) = ||x||_{\mathbb{Z}_2}$. If any component $S \in \Sing(X)$ has a neighbourhood in $X$ homeomorphic to an annulus and bounds a disk, then we can reverse the orientation of the disk and remove $S$ from $\Sing(X)$ to obtain a new $\mathbb{Z}_2$-folded surface $X'$ such that $[X'] = [X]$ and $\chi_-(X') \leq \chi_-(X)$. Reassigning $X := X'$ and repeating this procedure we may assume that $X$ is a $\mathbb{Z}_2$-folded surface in which the closure of any disk component of $\Int(X)$ is homeomorphic to $\mathbb{R}\mathrm{P}^2$ and we still have $\chi_-(X) = ||x||_{\mathbb{Z}_2}$. The only other components of $\Int(X)$ with positive Euler characteristic are $2$-spheres. We can view $X$ as an unoriented surface simply by forgetting about $\Sing(X)$. Let $X_0$ be a connected component of the unoriented surface $X$. It follows that:

    \[\max(0,-\chi(X_0)) = \sum_{X_0' \in \pi_0(X_0 \setminus \Sing(X_0))} \max(0,-\chi(X_0')).\]

    Globally, this implies the other inequality $||x||_{\mathbb{Z}_2} \leq ||x||_{\mathbb{Z}_2}'$, completing the proof.

\end{proof}

The function $||\cdot||_{\mathbb{Z}_p}$ is formally very similar to the Thurston norm on $H_2(M,\partial(M);\mathbb{Z})$ defined in Section~\ref{sec:Zthurstonnorm}. We will now make the relationship between these functions explicit in Lemma~\ref{lem:thurstonnorm}.

Let $x \in H_2(M;\mathbb{Z}_p)$ and take any oriented link $L$ such that $[L] = d_{\mathbb{Z}_p}(x) \in H_1(M;\mathbb{Z})$. Now let $N(L)$ denote an open regular neighbourhood of $L$ such that $M\setminus N(L)$ is a compact orientable $3$-manifold. Let $e:H_2(M,L;\mathbb{Z}) \rightarrow H_2(M \setminus N(L),\partial(N(L));\mathbb{Z})$ be the excision isomorphism. We define the subset $F(x,L) \subset H_2(M,L;\mathbb{Z})$ in Equation~\ref{eq:defofF}.

\begin{equation}\label{eq:defofF}
    F(x,L) := \{y\in H_2(M,L;\mathbb{Z}) : j_{\mathbb{Z}_p}(y) = j_{(M,L)}(x), d_{(M,L)}(y) = p[L]\}
\end{equation}

The proof of Lemma~\ref{lem:existence} shows that the set $F(x,L)$ is nonempty for all $x,L$ as above. When $H_2(M;\mathbb{Z})=0$, $F(x,L)$ contains a single class (i.e. denoted $x'$ in Lemma~\ref{lem:existence}). 

\begin{lem}\label{lem:thurstonnorm}
    For all $x\in H_2(M;\mathbb{Z}_p)$ where $M$ is a $3$-manifold Equation~\eqref{eq:zpthurstonzthurston} holds, where $L$ denotes an oriented link embedded in $M$ and $F$ is defined in Equation~\ref{eq:defofF}.
    
    \begin{equation}\label{eq:zpthurstonzthurston}
        ||x||_{\mathbb{Z}_p} = \min_{[L]=d_{\mathbb{Z}_p}(x)} \left(\min_{y\in F(x,L)} ||e(y)||_{\mathbb{Z}} \right)
    \end{equation}
    
\end{lem}

\begin{proof}
    Let $X$ be a $\mathbb{Z}_p$-folded surface embedded in $M$ via a map $f:X\hookrightarrow M$. To avoid confusion in this proof we avoid all abuses of notation by using $[X]$ and to denote the fundamental class of $X$ and $f_*([X])$ to denote the class in $H_2(M;\mathbb{Z}_p)$ it represents. Suppose $f_*([X]) = x$ and $\chi_{-}(X) = ||x||_{\mathbb{Z}_p}$. 
    
    First let $L := f(\Sing X)$. Equation~\eqref{eq:pSingXfromdef} shows that $[L] = d_{\mathbb{Z}_p}(x)$ using the naturality of the coefficient long exact sequence.

    \begin{equation}
        \label{eq:pSingXfromdef}
        \begin{aligned}
            d_{\mathbb{Z}_p}(x) &= d_{\mathbb{Z}_p} \circ f_*([X]) \\
                &= f_* \circ d_{\mathbb{Z}_p}([X]) \\
                &= f_*([\Sing X]) \\
                &= [f(\Sing X)]\\
                &= [L]
        \end{aligned}
    \end{equation}

    Now let $y' := f_*\circ q_*([X_{cut}]) \in H_2(M,L;\mathbb{Z})$. $e(y')$ is represented by the properly embedded oriented $2$-manifold $\Sigma := f\circ q(X_{cut}) \setminus N(L)$. $\Sigma$ is homeomorphic to $X_{cut}$ and $\chi_-(X_{cut}) = \chi_-(X)$ by definition. Hence $||e(y')|| \leq \chi_-(\Sigma) = \chi_-(X) = ||x||_{\mathbb{Z}_p}$. 
    
    We now claim that $y'\in F(x,L)$. It follows from the definition in Equation~\eqref{eq:defofF} that it is sufficient to show that $y'$ satisfies Claims~\eqref{cl:claim1} and~\eqref{cl:claim2}.

    \begin{equation}\label{cl:claim1}
        j_{\mathbb{Z}_p}(y') = j_{(M,L)}(x)
    \end{equation}

    \begin{equation}\label{cl:claim2}
        d_{(M,L)}(y') = p[L]
    \end{equation}

    First note that $j_{\mathbb{Z}_p}(y') = j_{\mathbb{Z}_p}\circ f_*\circ q_*([X_{cut}]) = f_* \circ j_{\mathbb{Z}_p} \circ q_*([X_{cut}])$ by naturality of the coefficient long exact sequence. It follows from the definition of $[X]$ in Section~\ref{sec:realisation} that $j_{(X,\Sing X)}([X]) = j_{\mathbb{Z}_p} \circ q_*([X_{cut}])$. Hence we obtain $j_{\mathbb{Z}_p}(y') = f_*\circ j_{(X,\Sing X)}([X])$. It follows from the naturality of the long exact sequences of the pairs under the map $(X,\Sing X)\xrightarrow{f} (M,L)$ implies that $f_*\circ j_{(X,\Sing X)} = j_{(M,L)} \circ f_*$ and so we have $j_{\mathbb{Z}_p}(y') =j_{(M,L)} \circ f_*([X]) = j_{(M,L)}(x)$. This proves Claim~\eqref{cl:claim1}.

    The same relation $f_*\circ j_{(X,\Sing X)} = j_{(M,L)} \circ f_*$ also implies that $d_{(M,L)}(y') = d_{(M,L)} \circ f_*\circ q_*([X_{cut}]) = f_* \circ d_{(X,\Sing X)} \circ q_*([X_{cut}])$. The naturality of the long exact sequences of the pairs under the map $(X_{cut},\partial X_{cut})\xrightarrow{q} (X,\Sing X)$ implies that $d_{(X,\Sing X)} \circ q_* = q_* \circ d_{(X_{cut},\partial X_{cut})}$. Hence we have $d_{(M,L)}(y') = f_* \circ q_* \circ d_{(X_{cut},\partial X_{cut})}([X_{cut}])$. $[X_{cut}]$ is the fundamental class of a compact orientable $2$-manifold and so if $\partial X_{cut}$ is endowed with the induced orientation it follows that $d_{(X_{cut},\partial X_{cut})}([X_{cut}]) = [\partial X_{cut}]$. Moreover, it follows from Definition~\ref{dfn:pfoldedsurface} that $q_*([\partial X_{cut}]) = p[\Sing X]$ and so $d_{(M,L)}(y') = f_*(p[\Sing X]) = p[f(\Sing X)] = p[L]$. This proves Claim~\eqref{cl:claim2}.

    Since we have shown that $[L] = d_{\mathbb{Z}_p}(x)$, $y'\in F(x,L)$ and $||x||_{\mathbb{Z}_p} \geq ||e(y')||$. Combining these yields Equation~\eqref{eq:combiningfacts}.
    
    \begin{equation}\label{eq:combiningfacts}
        ||x||_{\mathbb{Z}_p} \geq ||e(y')|| \geq \min_{[L]=d_{\mathbb{Z}_p}(x)} \left(\min_{y\in F(x,L)} ||e(y)||_{\mathbb{Z}} \right)
    \end{equation}

    We will now prove the opposite inequality. Let $L$ be an oriented link and $y\in F(x,L)$ such that $||y||$ minimises the right hand side of the statement in Lemma~\ref{lem:thurstonnorm}. Let $S$ be a properly embedded oriented surface in $M \setminus N(L)$ such that $[S] = y$ and $\chi_-(S) = ||y||$. The two surgeries described in the proof of Lemma~\ref{lem:existence} required to make $S$ homeomorphic to the interior of a $\mathbb{Z}_p$-folded surface can be described simply as gluing disks and annuli to the boundary components of $S$ in an orientation preserving way. If we cap off a boundary component of a connected component $S_i$ of $S$ with a disk to form a new component $S_i'$ of $S'$ then $\chi(S') = \chi(S)+1$ with all other components the same so $\chi_{-}(S')\leq \chi_{-}(S)$ and by the minimality of $\chi_{-}(S)$ we have $\chi_{-}(S') = \chi_{-}(S)$. If we glue an annulus to two boundary components of the same component $S_i$ of $S$ then $\chi(S_i') = \chi(S_i)$ and so $\chi_{-}(S) = \chi_{-}(S')$. Finally if we glue an annulus joining boundary components of distinct components $S_i$ and $S_j$ of $S$ to form $S_{ij}'$ and $S'$ then $\chi(S_{ij}') = \chi(S_i)+\chi(S_j)$. If $\chi(S_i) < 0 < \chi(S_j)$ or vice-versa then $\chi_{-}(S')< \chi_{-}(S)$ which contradicts the minimality of $\chi_{-}(S)$, so this cannot happen. The remaining cases are that both components are disks, both have negative Euler characteristic or one of them is an annulus. In all three cases it follows that $\chi_{-}(S_{ij}') = \chi_{-}(S_i)+ \chi_{-}(S_j)$ and so $\chi_{-}(S') = \chi_{-}(S)$. Combining this observation with the rest of the proof of Lemma~\ref{lem:existence} shows that we can construct a $\mathbb{Z}_p$-folded surface $X$ realising $x$ such that $||x||_{\mathbb{Z}_p} \leq \chi_{-}(X) = \chi_{-}(S) = ||y||$. This yields Equation~\eqref{eq:combiningfacts2}.

    \begin{equation}\label{eq:combiningfacts2}
        ||x||_{\mathbb{Z}_p} \leq \min_{[L]=d_{\mathbb{Z}_p}(x)} \left(\min_{y\in F(x,L)} ||e(y)||_{\mathbb{Z}} \right)
    \end{equation}

    Combining Equations~\eqref{eq:combiningfacts} and~\eqref{eq:combiningfacts} completes the proof.

\end{proof}

\subsection{Weighted folded surfaces}\label{sec:weightedfoldedsurfs}

In~\cite[Chapter~IV.3.3]{Turaev2002torsions} Turaev used the notion of a \textit{weighted surface} as part of the development of a lower bound for the Thurston norm. We adapt this to give  characterisations of a \textit{weighted oriented folded surface} and a \textit{weighted $\mathbb{Z}_p$-folded surface} in Definition~\ref{dfn:weightedfolded}. 

\begin{dfn}[Weighted Folded Surface]\label{dfn:weightedfolded}
    A \textit{weighted oriented folded surface} $X$ is an oriented folded surface $\underline{X}$ together with a positive integer $w_I$ associated to each component $I$ of $\Int(\underline{X})$. If $X$ is embedded in an orientable $3$-manifold then there exists a related unweighted oriented folded surface $X^u \subset M$ which can be constructed by the following procedure for each $I\in \pi_0(\Int(\underline{X}))$. Since both $M$ and $I$ are oriented, the 1-dimensional normal bundle of $I$ in $M$ is trivial. This means that $I$ can be pushed in a normal direction, fixing the complement of $I$ in $X$, to obtain $w_I$ parallel copies with the same singular boundary $\partial(I)$. These are completely disjoint if $\partial(I) = \emptyset$. For each $I\in \pi_0(\Int(\underline{X}))$ replace $I$ with these parallel copies to obtain $X^u$. If $X^u$ is a $\mathbb{Z}_p$-folded surface for some choice of orientation of $\Sing(X^u) = \Sing(\underline{X})$ then we say that $X$ is a \textit{weighted $\mathbb{Z}_p$-folded surface}. To avoid confusion we will write $\underline{X}$ when we are referring to the oriented folded surface obtained from $X$ by forgetting the weights. This is not to be confused with the $\mathbb{Z}_p$-folded surface $X^u$ in general but note that $\underline{X} = X^u$ in the special case when every weight is equal to $1$.
\end{dfn}

Lemma~\ref{lem:weightedfoldedalternatedef} is an analogue of Lemma~\ref{lem:alternatedefinition} for weighted $\mathbb{Z}_p$-folded surfaces and may be taken as an alternate definition of a weighted $\mathbb{Z}_p$-folded surface.

\begin{lem}\label{lem:weightedfoldedalternatedef}
    A weighted folded surface $X$ is a weighted $\mathbb{Z}_p$-folded surface if and only if there exists an orientation of $\Sing(X)$ with the following properties.
    \begin{enumerate}
        \item $\deg(B \xrightarrow{q_{\underline{X}}} q_{\underline{X}}(B))>0$ for all $B\in \pi_0(\partial \underline{X}_{cut})$.
        \item Let $X_{B}$ denote the unique component of $\Int X$ such that the lift of $X_B$ to $X_{cut}$ cobounds $B$. Then $\sum_{B \in \pi_0(q^{-1}(K))} w_{X_B} \cdot \deg(B \xrightarrow{q_{\underline{X}}} K) = p$ for all $K\in \pi_0(\Sing(X))$.
    \end{enumerate} 
\end{lem}

Lemma~\ref{lem:weightedpfoldedsurfaces} is a generalisation of~\cite[Lemma IV.3.4]{Turaev2002torsions}.

\begin{lem}\label{lem:weightedpfoldedsurfaces}
    Let $X$ be a weighted $\mathbb{Z}_p$-folded surface embedded in a connected orientable 3-manifold $M$. There exists another weighted $\mathbb{Z}_p$-folded surface $Y$ with the following properties.
    \begin{enumerate}
        \item $[Y] = [X]$
        \item $\chi_-(Y^u)\leq \chi_-(X^u)$
        \item $|\Sing(\underline{Y})| = |\Sing(\underline{X})|$
        \item $M\setminus \underline{Y}$ is connected
    \end{enumerate}

    Moreover, for any $x \in H_2(M;\mathbb{Z}_p)$ there exists a weighted $\mathbb{Z}_p$-folded surface $X$ such that $[X]=x$, $||x||_{\mathbb{Z}_p} = \chi_{-}(X^u)$ and $M\setminus \underline{X}$ is connected.
\end{lem}

\begin{proof}

    This proof adapted from \cite[Lemma IV.3.4]{Turaev2002torsions}, which is the analogue of Lemma~\ref{lem:weightedpfoldedsurfaces} for \textit{weighted surfaces}. Let $X$ be a weighted $\mathbb{Z}_p$-folded surface. Let $N$ be a component of the complement of $\underline{X}$ in $M$. If $M \setminus \underline{X}$ is not connected, then $\partial(\overline{N}) \neq \emptyset$. Let $\alpha_+$ (resp. $\alpha_-$) be the subset of $\pi_0(\Int(\underline{X}))$ given by the components on the boundary of $\overline{N}$ such that their orientation in $\Int(\underline{X})$ is the same as (resp. opposite to) the orientation of $\partial(\overline{N})$ induced by the orientation of $M$.  Now if $\alpha_{-} = \emptyset$ or $\sum_{I\in \alpha_{+}} \chi_-(I) \geq \sum_{J\in \alpha_{-} } \chi_-(J)$ then decrease each $w_{I\in \alpha_+}$ and increase each $w_{J\in \alpha_{-}}$ by $\text{min}\{w_I |I\in \alpha_+\}$. Otherwise, in the case where $\alpha_- \neq \emptyset$ and $\sum_{I\in \alpha_{+}} \chi_-(I) < \sum_{J\in \alpha_{-} } \chi_-(J)$, then increase each $w_{I\in \alpha_+}$ and decrease each $w_{J\in \alpha_-}$ by $\text{min}\{w_J |J\in \alpha_-\}$. Remove any components of $\Int(\underline{X})$ whose weights become $0$ from $\underline{X}$ and denote the resulting weighted oriented folded surface $Y$. The embedding of $(\underline{Y},\Sing(\underline{Y}))$ in $M$ factors through the embedding of $(\underline{X},\Sing(\underline{X}))$ in $M$. We can therefore abuse notation to denote by $[X],[Y]$ the fundamental classes of $X$ and $Y$ in $H_2(\underline{X};\mathbb{Z}_p)$. It follows from the construction of $Y$ that $m_*([X] - [Y]) \in H_2(\underline{X},\Sing(\underline{X});\mathbb{Z}_p)$ is an integer multiple of $\sum_{I\in \alpha_+}[I] - \sum_{J\in \alpha_-}[J]$ and hence $[X] - [Y]$ is an integer multiple of $[\partial \overline{N}]$, $H_2(X;\mathbb{Z}_p)$ class of the $2$-cycle boundary of $N$. It follows from definition~\ref{dfn:realisation} that $X$ and $Y$ represent equivalent classes in $H_2(M;\mathbb{Z}_p)$ since $[\partial \overline{N}]$ is annihilated by the map induced by $X\hookrightarrow M$.
    
    To see that $Y$ is a weighted $\mathbb{Z}_p$-folded surface and that $\Sing(X) = \Sing(Y)$ simply note that all the components of $Y^u$ still induce consistent orientations on the singular locus and that the long exact sequence of the pair $(\underline{X},\Sing(\underline{X}))$ implies that $\sum_{I\in \alpha_{+}}[\partial(I)] = \sum_{J\in \alpha_{-} }[\partial(J)]  \in H_1(\Sing(\underline{X}))$ and so there is no change in the total number of sheets of $Y^u$ identified along each singular arc. The method by which weights are changed implies that $\chi_-(Y^u)\leq \chi_-(X^u)$. It also follows from the construction that $M\setminus \underline{Y}$ has strictly fewer components than $M\setminus \underline{X}$.
    
    Iterating this process we eventually obtain a weighted $\mathbb{Z}_p$-folded surface $Y$ such that $M\setminus \underline{Y}$ is connected, satisfying the required properties. 

    The proof is completed by noting that for any $x\in H_2(M;\mathbb{Z}_p)$ and $\underline{X}$ a $\mathbb{Z}_p$-folded surface that is optimal in the sense that $||x||_{\mathbb{Z}_p} = \chi_-(\underline{X})$ we can assign weights $w_I = 1$ for all $I \in \pi_0(\Int(\underline{X}))$ to obtain a weighted $\mathbb{Z}_p$-folded surface $X$ such that $X^u = \underline{X}$ and hence $||x||_{\mathbb{Z}_p} = \chi_-(X^u)$. By following the above induction the resulting weighted $\mathbb{Z}_p$-folded surface $Y$ must also have the property that $||x||_{\mathbb{Z}_p} = \chi_-(Y^u)$ (this follows from the definition of $||x||_{\mathbb{Z}_p}$) and the additional property that $M\setminus \underline{Y}$ is connected.

\end{proof}

We now consider the more general situation of 2-polyhedra $Y$ embedded in a $3$-manifold $M$. Note that the interior $\Int(Y)$ of a general $2$-polyhedron is defined to be the set of points with a neighbourhood homeomorphic to an open disk and the singular locus is the $1$-polyhedron $\Sing(Y):= Y \setminus \Int(Y)$. If $M\setminus Y$ is connected then we call it \textit{non-separating}. We call a weighted folded surface $X$ non-separating if $\underline{X}$ is non-separating (even if $X^u$ may not be). Lemma~\ref{lem:injectivitynonsep} proves some useful properties of non-separating 2-polyhedra which can be translated into properties of non-separating weighted $\mathbb{Z}_p$-folded surfaces by setting $Y:=\underline{X}$.

\begin{lem}\label{lem:injectivitynonsep}
    Let $Y$ be a 2-polyhedron embedded in a closed orientable connected 3-manifold $M$. The following are equivalent:
    \begin{enumerate}
        \item $Y$ is non-separating in $M$. \label{enum:nonseparating}
        \item The induced homomorphism $f_{\mathbb{Z}}:H_2(Y;\mathbb{Z})\rightarrow H_2(M;\mathbb{Z})$ is injective. \label{enum:Zinjective}
        \item The induced homomorphism $f_{\mathbb{Z}_p}:H_2(Y;\mathbb{Z}_p)\rightarrow H_2(M;\mathbb{Z}_p)$ is injective for all $p\in \mathbb{Z}^{>1}$. \label{enum:Zkinjective}
    \end{enumerate}
\end{lem}

\begin{proof}
    The implication $\text{\ref{enum:nonseparating}} \Leftrightarrow \text{\ref{enum:Zinjective}}$ can be derived from the long exact sequence of the pair $(M,Y)$ as follows.

    The number of components of $M\setminus Y$ is equal to the rank of $H_3(M,Y;\mathbb{Z})$ by taking each component to be a free generator. $H_3(Y;\mathbb{Z})=0$ by a dimension argument and so $j_{(M,Y)}:H_3(M;\mathbb{Z})\rightarrow H_3(M,Y;\mathbb{Z})$ is injective and so $\ker(H_3(M,Y;\mathbb{Z})\rightarrow H_2(Y;\mathbb{Z})) = H_3(M;\mathbb{Z}) = \mathbb{Z}$. $H_2(Y;\mathbb{Z})$ is torsion free so we can deduce that the kernel of $H_3(M,Y;\mathbb{Z})\rightarrow H_2(Y;\mathbb{Z})$ is generated by a primitive element of $H_3(M,Y;\mathbb{Z})$. It follows from the first isomorphism theorem that $\Ima(H_3(M,Y;\mathbb{Z})\rightarrow H_2(Y;\mathbb{Z}))$ is a free abelian group of rank one less than the rank of $H_3(M,Y;\mathbb{Z})$. It follows immediately that $\ker(H_2(Y;\mathbb{Z})\rightarrow H_2(M;\mathbb{Z})) = 0$ if and only if $M\setminus Y$ has a single connected component.

    We can now prove the implication $\text{\ref{enum:Zinjective}} \Leftrightarrow \text{\ref{enum:Zkinjective}}$ by demonstrating both contrapositives.

    Suppose $\ker(H_2(Y;\mathbb{Z})\rightarrow H_2(M;\mathbb{Z}))$ is non-trivial, then the above arguments imply that it contains a primitive element $x\in H_2(Y;\mathbb{Z})$. This primitive element cannot be in the image of $i_{\mathbb{Z}_p}$ which consists only of non-primitive elements in the subgroup $p \cdot H_2(Y;\mathbb{Z})$. Hence $j_{\mathbb{Z}_p}(x)\neq 0$. But since 
    \[0 = j_{\mathbb{Z}_p}\circ i_{(M,Y)}(x) =  i_{(M,Y)}\circ j_{\mathbb{Z}_p}(x)\] 
    it follows that 
    \[j_{\mathbb{Z}_p}(x) \in \ker(H_2(Y;\mathbb{Z}_p)\rightarrow H_2(M;\mathbb{Z}_p)).\] 

    Now suppose that $x$ is a non-trivial element of $\ker(H_2(Y;\mathbb{Z}_p)\rightarrow H_2(M;\mathbb{Z}_p))$. Lift $x$ to $x'\in H_3(M,Y;\mathbb{Z}_p)$ such that $d_{(M,Y)}(x') = x$. It follows from excising an open neighbourhood $N(Y)$ of $Y$ that $H_2(M,Y;\mathbb{Z})$ is isomorphic to $H_2(M\setminus N(Y),\partial N(Y);\mathbb{Z})$. By Poincar\'e-Lefschetz duality 
    \[H_2(M\setminus N(Y),\partial N(Y);\mathbb{Z}) \simeq H^1(M\setminus N(Y);\mathbb{Z}).\] 
    The universal coefficient theorem implies that 
    \[H^1(M\setminus N(Y);\mathbb{Z}) \simeq \text{Hom}(H_1(M\setminus N(Y);\mathbb{Z}), \mathbb{Z})\] 
    which must be torsion free. We therefore deduce that $H_2(M,Y;\mathbb{Z})$ is torsion free. 
    
    Hence $d_{\mathbb{Z}_p}:H_3(M,Y;\mathbb{Z}_p) \rightarrow H_2(M,Y;\mathbb{Z})$ annihilates all of $H_3(M,Y;\mathbb{Z}_p)$. It follows that $j_{\mathbb{Z}_p}:H_3(M,Y;\mathbb{Z})\rightarrow H_3(M,Y;\mathbb{Z}_p)$ is a surjection and so $x' = j_{\mathbb{Z}_p}(x'')$ for some $x'' \in H_3(M,Y;\mathbb{Z})$. Since $x = d_{(M,Y)} \circ j_{\mathbb{Z}_p} (x'') =  j_{\mathbb{Z}_p} \circ d_{(M,Y)} (x'')$ it follows that $d_{(M,Y)} (x'') \neq 0$. The exactness of the sequence of the pair $(M,Y)$ implies that $d_{(M,Y)} (x'') \in \ker(H_2(Y;\mathbb{Z})\rightarrow H_2(M;\mathbb{Z}))$. This completes the proof.
    
\end{proof}

When $p$ is prime, the second 
\textit{$\mathbb{Z}_p$-Betti number} of a $3$-manifold $M$ is defined to be the dimension of $H_2(M;\mathbb{Z}_p)$ as a $\mathbb{Z}_p$-vector space. Lemma~\ref{lem:components} leverages this property to determine an upper bound on the number of components of a weighted non-separating $\mathbb{Z}_p$-folded surface. 

\begin{lem}\label{lem:components} %
    If $X$ is a weighted non-separating $\mathbb{Z}_p$-folded surface embedded in $M$ such that every component of $X$ has a non-trivial fundamental class (equivalently $w_I < p$ for all $I\in \pi_0(\Int X)$) then $|\pi_0(\underline{X})|\leq \beta_2(M;\mathbb{Z}_p)$.
\end{lem}

\begin{proof}
    Suppose $X$ is a weighted non-separating $\mathbb{Z}_p$-folded surface in $M$ such that every weight is less than $p$. Let $\pi_0(X)$ denote the set of weighted $\mathbb{Z}_p$-folded surfaces $Y$ such that $\underline{Y} \in \pi_0(\underline{X})$ and the weight function on $\pi_0(\Int Y)$ is the restriction of the weight function of $\pi_0(\Int X) \supseteq \pi_0(\Int Y)$. For each $Y \in \pi_0(X)$ we deduce that the fundamental class $[Y] \neq 0$ since $j_{(Y,\Sing Y)}([Y]) \neq 0$ if and only if $w_I \not \equiv 0 \mod p$ for all $I\in \Int X$. Hence $\angb{i_{(X,Y)}([Y])}$ is a linear subspace of $H_2(X;\mathbb{Z}_p)$. Since $\pi_0(\underline{X})$ is a pairwise disjoint set of subspaces of $X$ we can also deduce that $\{i_{(X,Y)}([Y])\ |\ Y \in \pi_0(X) \}$ is a linearly independent set. Hence $|\pi_0(\underline{X})| \leq \beta_2(X;\mathbb{Z}_p) = \dim H_2(X;\mathbb{Z}_p)$. Lemma~\ref{lem:injectivitynonsep} implies that $i_{(M,X)}:H_2(X;\mathbb{Z}_p)\rightarrow H_2(M;\mathbb{Z}_p)$ is injective and hence $\beta_2(X;\mathbb{Z}_p) \leq \beta_2(M;\mathbb{Z}_p)$. Putting together these two inequalities completes the proof. 
\end{proof}

Definition~\ref{dfn:removablecomponents} characterises \textit{removable components} of the interior of a folded surface that are so called because they give rise to certain surgeries that allow us to simplify $\mathbb{Z}_p$-folded surfaces in Lemmas~\ref{lem:noremovabledisks} and~\ref{lem:removenotdisks}. 

\begin{dfn}\label{dfn:removablecomponents}
    Let $X$ be a weighted $\mathbb{Z}_p$-folded surface and $q_{\underline{X}}:\underline{X}_{cut}\rightarrow X$ be the quotient map defined in Section~\ref{sec:theta}. Let $C \in \pi_0(\Int(\underline{X}))$. Denote $C_{cut}$ the component of $X_{cut}$ that is the closure of $q_{\underline{X}}^{-1}(C)$. We call $C$ \textit{removable} if the restriction $C_{cut}\xrightarrow{q_{\underline{X}}}\underline{X}$ is an embedding.
\end{dfn}

\begin{lem} \label{lem:weightp/2}
    If $X$ is a weighted non-separating $\mathbb{Z}_p$-folded surface and $C\in \pi_0(\Int(\underline{X}))$ is a component with weight $w_C>p/2$ then $C$ is removable.
\end{lem}

\begin{proof}

    Let $C$ be any non-compact component of $\Int(\underline{X})$ such that $w_C>p/2$. Let $C_{cut}$ be the component of $\underline{X}_{cut}$ that is the closure of $q_{\underline{X}}^{-1}(C)$. For any $K\in \pi_0(\Sing(\underline{X}))$ let $B(C,K)$ be the set of boundary components of $C_{cut}$ mapped to $K$ by $q_{\underline{X}}$. The construction of $X^u$ induces homeomorphisms from $C$ to its $w_C$ parallel copies $C^i$ which extend uniquely to homeomorphisms of $C_{cut}\rightarrow C^i_{cut}$. This allows us to describe disjoint sets $B^u\subset \pi_0(\partial{X^u_{cut}})$ of the $w_C$ boundary components identified with $B \in B(C,K)$ by these homeomorphisms. The construction also implies that $q_{X^u}(\mathcal{B}) = q_{\underline{X}}(B) = K$ and that $\deg(\mathcal{B}\xrightarrow{q_{X^u}} K) = \deg(B\xrightarrow{q_{\underline{X}}}K)$ for all $\mathcal{B}\in B^u$. Summing the degrees of the corresponding maps to $K$, Lemma~\ref{lem:alternatedefinition} implies that $p = \sum_{\mathcal{B} \in \pi_0(q_{X^u}^{-1}(K))} \deg(\mathcal{B}\xrightarrow{q_{X^u}}K) \geq \sum_{\mathcal{B}\in B^u}\sum_{B\in B(C,K)}\deg(B\xrightarrow{q_{\underline{X}}} K) = \sum_{B\in B(C,K)}\deg(B\xrightarrow{q_{\underline{X}}} K)\cdot w_C$. If $w_C>p/2$ then this inequality implies that $\deg(B\xrightarrow{q_{\underline{X}}} K) = 1$ and $|B(C,K)| = 1$ for all $K\in \pi_0(q_{\underline{X}}(\partial C_{cut}))$. This shows that each boundary component of $C_{cut}$ is mapped homeomorphically onto a distinct component of $\Sing(\underline{X})$ and so $\partial{C_{cut}}\xrightarrow{q_{\underline{X}}} \Sing(\underline{X})$ is an embedding. It is already known that $\Int(\underline{C}_{cut})\xrightarrow{q_{\underline{X}}} C \subset \Int(\underline{X})$ is a homeomorphism. Together this implies that $C_{cut}\xrightarrow{q_{\underline{X}}} \underline{X}$ is an embedding, proving the claim.
    
\end{proof}

\begin{lem}\label{lem:removabledisksorclosed}
    If $C \in \pi_0(\Int(X))$ is homeomorphic to an open disk then either $C$ is a removable disk or the closure of $C$ in $X$ is a component of $X$ homeomorphic to the $\mathbb{Z}_p$-crosscap in Figure~\ref{fig:pcrosscaps}.
\end{lem}

\begin{proof}
    Let $B := \partial{C_{cut}}$. Lemma~\ref{lem:classificationin3mflds} implies that $\deg(B\xrightarrow{q_{X}} q_{X}(B)) \in \{1,p\}$. If $\deg(B\xrightarrow{q_{X}} q_{X}(B)) = p$ then $q_{X}(C_{cut}) = \overline{C}$ is the $\mathbb{Z}_p$-folded surface in Figure~\ref{fig:pcrosscaps}. If $\deg(B\xrightarrow{q_{X}} p(B)) = 1$ then $C_{cut}$ is mapped homeomorphically onto its image and so $C$ is a removable disk by definition.
\end{proof}

We can use these lemmas to restrict the class of weighted $\mathbb{Z}_p$-folded surfaces realising the function $||\cdot||_{\mathbb{Z}_p}$ a little further.

\begin{lem}\label{lem:noremovabledisks}

    Let $X$ be a weighted non-separating $\mathbb{Z}_p$-folded surface with a removable disk $D$ component of $\Int(X)$. There exists another weighted non-separating $\mathbb{Z}_p$-folded surface $Y$ such that:
    
    \begin{enumerate}
        \item $[X] = [Y]$
        \item $\chi_-(Y^u)\leq \chi_-(X^u)$
        \item $|\pi_0(\Sing(Y^u))| = |\pi_0(\Sing(X^u))|-1$
        \item $\chi(Y^u) = \chi(X^u) + p - 2w_D$
        \item $\max_{C\in \pi_0(\Int(\underline{Y}))} w_C \leq \max_{C\in \pi_0(\Int(\underline{X}))} w_C$
    \end{enumerate}
\end{lem}

\begin{proof}
    If $w_D = p$ then the closure of $D$ in $X$ is a connected component of $X$ which we denote $X_D$. $X_D$ represents a trivial homology class in $M$ since it embeds in a 3-ball regular neighbourhood of $D$. Let $Y$ be the weighted $\mathbb{Z}_p$-folded surface obtained by removing this $X_D$ from $X$. Since $\chi(X_{D}^u) = p$ it follows that $\chi(Y^u) = \chi(X^u) - p = \chi(X^u) +p - 2w_D$ as required. The rest of the requirements of the proof are similarly easy to prove.

    Hence, let $w_D <p$ and let $K = \partial{D}$ be the singular arc along its boundary. There is at least one other component $C$ of $\Int(X)$ such that $K\subseteq \partial{C}$. Lemma~\ref{lem:classificationin3mflds} and Lemma~\ref{lem:alternatedefinition} together imply that $\deg(B\xrightarrow{q_{\underline{X}}}K) = 1$ for all $B\in \pi_0(q_{\underline{X}}^{-1}(K))$. Let $k$ be the number of boundary components of $\underline{X}_{cut}$ in the preimage of $K$ under $q$. A similar argument to the proof of Lemma~\ref{lem:regularneighbourhoods} implies that the embedding $\underline{X}\cap N \hookrightarrow N$ where $N$ is a regular neighbourhood of $K$ is homeomorphic to the sub-bundle embedding of $K_{k,1}\times S^1 \hookrightarrow \mathbb{D}^2 \times S^1$ that restricts on sections to an embedding $K_{k,1}\hookrightarrow \mathbb{D}^2$ where the image in the unit disk is given by the $k\leq p$ intervals from the origin to $k$ equally spaced points on the unit circle. The interior of exactly one of these intervals is in $D$.
    
    Now push $D$ in a normal direction relative to $\partial{D}$ to obtain $k-1$ parallel copies $\{D_i\}_{i=1}^{k-1}$ and let $X' = (\underline{X}\setminus D)\cup \bigsqcup_{i=1}^{k-1}D_i$. $X'$ is a folded surface with $2k-2$ sheets meeting along $K$. Now the embedding $X'\cap N\hookrightarrow N$ is homeomorphic to the sub-bundle embedding of $K_{2k-2,1}\times S^1 \hookrightarrow \mathbb{D}^2 \times S^1$ that restricts on sections to an embedding $K_{2k-2,1}\hookrightarrow \mathbb{D}^2$. The image of $K_{2k-2,1}\hookrightarrow \mathbb{D}^2$ in the unit disk is given by the $2k-2$ unit intervals $\{[0,x_j]\}_{j=1}^{2k-2}$ where $x_j := e^{\frac{2i\pi j}{2k-2}}$ and the points $x_i \in D_i$ for all $1\leq i  \leq k-1$ and $x_j\in \Int(\underline{X})\setminus D$ for all $k - 1<j\leq 2k-2$. Replace $X'\cap N$ with the sub-bundle $\left(\bigsqcup_{i=1}^{k-1} I_i \right)\times S^1 \hookrightarrow \mathbb{D}^2\times S^1$ which is described on sections as the proper embedding of $k-1$ intervals $I_i := [x_i,x_{2k-1-i}]$ in $\mathbb{D}^2$. Orient and weight the components of the resulting folded surface to be consistent with the orientation and weight of the components of $\Int(\underline{X})$ where they overlap in $M$. Let the resulting weighted folded surface be $Y$. See Figure~\ref{fig:crosssection} for an illustration of the cross-sections of the intersection of $\underline{X},\underline{X'}$ and $\underline{Y}$ with $N$. The complement $M\setminus \underline{Y}$ is connected since it can be obtained from $M \setminus \underline{X}$ by gluing $(k-2)$ 2-handles that lie in the space between $D_i$ and $D_{i+1}$ for $1\leq i\leq k-2$. See Figure~\ref{fig:diskremoval} for a visualisation of an example of this disk removal procedure where $p=11$ and $k=4$.

    $Y$ is a weighted $\mathbb{Z}_p$-folded surface since the weights and orientation of sheets in neighbourhoods of every singular arc other than $K$ are unchanged. It can also be seen from the construction that $\Sing(Y) = \Sing(X)\setminus K$. Since $Y$ differs from $X$ only in the interior of a $3$-ball regular neighbourhood $H$ of $D$ it follows that $[X]-[Y] \in \text{Im}(H_2(H;\mathbb{Z}_p) \rightarrow H_2(M;\mathbb{Z}_p))$. Since $H_2(H;\mathbb{Z}_p) = 0$ it follows that $[X] = [Y]$. 
    
    $Y^u_{cut}$ is homeomorphic to the $2$-manifold obtained from $X^u_{cut}$ by removing $w_D$ copies of $D_{cut}$ and gluing disks to $p-w_D$ boundary components of $X^u_{cut} \setminus D_{cut}$. From this observation Equations~\eqref{eq:chiminusequation} and~\eqref{eq:chiequation} follow from elementary properties of the Euler characteristic.

    \begin{equation}\label{eq:chiminusequation}
        \chi_-(Y^u) = \chi_-(Y^u_{cut}) \leq \chi_-(X^u_{cut}) = \chi_-(X^u)
    \end{equation}

    \begin{equation}\label{eq:chiequation}
        \chi(Y^u) = \chi(Y^u_{cut}) = \chi(X^u)_{cut} - w_D + p-w_D = \chi(X^u_{cut})+ p - 2w_D
    \end{equation}
    
    The weights of $Y$ are all present among the weights of $X$ by construction. In particular, the maximum weight of $Y$ is bounded by the maximum weight of $X$ from above.

\end{proof}

\begin{figure}
    \centering
    \includegraphics[width=0.6\textwidth]{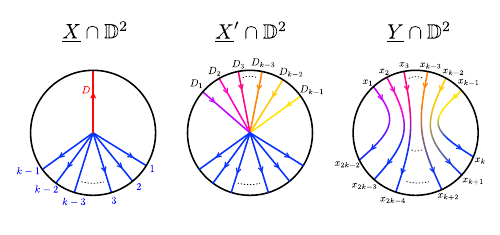}
    \caption{Intersections with meridional disk of $N(K)$ with $\underline{X}$, $\underline{X'}$ and $\underline{Y}$.}\label{fig:crosssection}
\end{figure}

\begin{figure}
    \centering
    \includegraphics[width=0.6\textwidth]{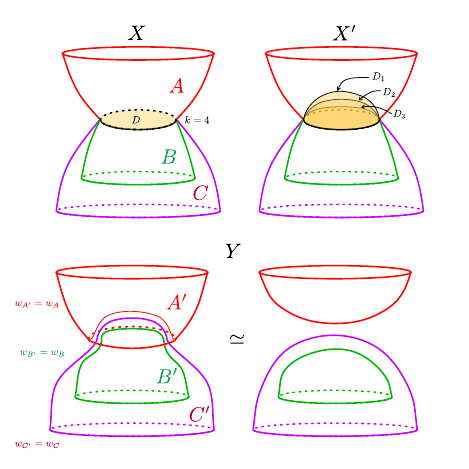}
    \caption{An example of the disk removal procedure when $p=11$ and $k=4$.}\label{fig:diskremoval}
\end{figure}

\begin{lem}\label{lem:removenotdisks}
    Let $X$ be a weighted non-separating $\mathbb{Z}_p$-folded surface with no removable disks, but a removable component $C$ of $\Int(\underline{X})$ such that $w_C>p/2$ with $\partial C \neq \emptyset$. There exists another weighted non-separating $\mathbb{Z}_p$-folded surface $Y$ with the following properties.

    \begin{enumerate}
        \item $[X] = [Y]$
        \item $\chi_-(Y^u)\leq \chi_-(X^u)$
        \item $|\pi_0(\Sing(Y^u))| = |\pi_0(\Sing(X^u))|- |\pi_0(\partial{C})|$ \label{property:sing}
    \end{enumerate}
\end{lem}

\begin{proof}

    This proof is similar to the proof of Lemma~\ref{lem:noremovabledisks}, but we need to be a bit more careful with orientation in the construction and the $\chi_-$ of $X$ and $Y$. It will be sufficient to find a $\mathbb{Z}_p$-folded surface $Y'$ such that $[X] = [Y']$, $\chi_-(Y') \leq \chi_-(X^u)$ and $|\pi_0(\Sing(Y'))| = |\pi_0(\Sing(X^u))|- |\pi_0(\partial{C})|$ since we can use Lemma~\ref{lem:weightedpfoldedsurfaces} to collapse $Y'$ onto a non-separating weighted $\mathbb{Z}_p$-folded surface $Y$ such that $\Sing(Y) = \Sing(Y')$ and $\chi_-(Y^u)\leq \chi_-(Y')$. Additionally, if $w_C = p$ the proof is trivial by simply deleting $\overline{C}$ from the components of $X$ to obtain $Y$ so we assume that $p/2<w_C<p$ in the following.
    
    Let $\{C_i\}_{i=1}^{w_C}$ be the components of $\Int(X^u)$ parallel to $C$ as per the construction of $X^u$ from $\underline{X}$. Remove the components $\{C_i\}_{i=p-w_C+1}^{w_C}$ and reverse the orientation of the remaining components $\{C_i\}_{i=1}^{p-w_C}$ to create a new oriented folded surface $X'$. Let $K$ be one of the singular arcs on the boundary of $C$ and let $N(K)$ be a regular neighbourhood of $K$ in $M$. The embedding $N(K)\cap X'\rightarrow N(K)$ is homeomorphic to the embedding of trivial sub-bundles $K_{2p-2w_C,1}\times S^1 \hookrightarrow \mathbb{D}^2\times S^1$ where the embedding restricts on sections to the embedding $K_{2p-2w_C,1}\hookrightarrow \mathbb{D}^2$ whose image in the unit disk is $2p-2w_C$ unit intervals $\{[0,x_j]\}_{j=1}^{2p-2w_C}$ where $x_j:= e^{\frac{2i\pi j}{2p-2w_C}}$ and $x_i \in C_i$ for all $1\leq i\leq p-w_C$ and $x_j \in (X\setminus C)^u$ for all $p-w_C< j\leq 2p-2w_C$. Fix an orientation of $\mathbb{D}^2$ in the splitting $\mathbb{D}^2\times S^1$. The orientations of $\Int(X)$ and $\mathbb{D}^2$ induce an orientation on the edges of the graph $X'\cap \mathbb{D}^2 \simeq K_{2p-2w_C,1}$ such that the intervals $[0,x_j]$ where $1\leq j \leq p-w_C$ induce the opposite orientation on the endpoint $0$ to the intervals where $p-w_C < j \leq 2p-2w_C$. Replace $X'\cap N(K)$ with the proper embedding of annuli given by the sub-bundle $\left(\bigsqcup_{i=1}^{p-w_C} I_i \right)\times S^1 \hookrightarrow \mathbb{D}^2\times S^1$ which is described on sections as the proper embedding intervals $I_i := [x_i,x_{2p-2w_C-i+1}]$ in $\mathbb{D}^2$ where $1\leq i \leq p-w_C$. Orient $I_i$ such that it induces the same orientation on its endpoints $x_i,x_{2p-2w_C-i+1}$ as the sum of the edges $[0,x_i]$ and $[0,x_{2p-2w_C-i+1}]$ in the embedding $K_{2p-2w_C,1}\hookrightarrow \mathbb{D}^2$. This induces an orientation of the annuli $\left(\bigsqcup_{i=1}^{p-w_C} I_i \right)\times S^1$ such that the orientation induced on their boundary is counter to the orientation induced by $\Int(X')\cap {M\setminus N(K)}$ hence preserving orientability of the resulting folded surface. Perform these replacements to every singular arc on the boundary of $C$ and call the resulting folded surface $Y'$. Orient $\Int(Y')$ such that its orientation is consistent with $\Int(X')$ wherever the two overlap.
    
    From the construction above we have that 
    \[|\pi_0(\Sing(Y'))| = |\pi_0(\Sing(X^u))|- |\pi_0(\partial{C})|.\] 
    Additionally, $Y'$ differs from $X^u$ only in a regular neighbourhood of $C$ homeomorphic to a handlebody $H$ of genus $2g(C) + n(C) -1$ where $g(C)$ and $n(C)$ are the genus and number of boundary components of $\overline{C}$. A singular homology argument therefore implies that $[Y']- [X^u] \in \text{Im}(H_2(H;\mathbb{Z}_p)\rightarrow H_2(M;\mathbb{Z}_p)) = \{0\}$ since $H_2(H;\mathbb{Z}_p) = 0$ and so $[Y'] = [X^u] = [X]$ as required. 

    By assumption there are no removable disks in $X$ and so Lemma~\ref{lem:removabledisksorclosed} implies that $\chi(C_{cut})\leq 0$. $X'_{cut}$ is obtained from $X^u_{cut}$ by deleting $2w_C - p$ components of Euler characteristic $\chi(C)$ and so $\chi_-(X') = \chi_-(X^u) - (2w_C -p)\chi_-(C)\leq \chi_-(X^u)$. $Y'_{cut}$ is obtained from $X'_{cut}$ by attaching $p-w_C$ annuli at each singular circle in $\partial C$ between pairs of boundary components $B_1,B_2$ belonging to components $S_j\in \pi_0(X'_{cut})$ for $j \in \{1,2\}$. The component of $X^u$ containing the $C_i$ has no removable disks and is not a $\mathbb{Z}_p$-crosscap and so Lemma~\ref{lem:removabledisksorclosed} implies that $\chi(S_j)\leq 0$. Let $S$ be the compact 2-manifold obtained by identifying the boundary of an annulus to $B_1$ and $B_2$ of disjoint $S_1$ and $S_2$. An elementary Euler characteristic computation shows that $\chi(S) = \chi(S_1)+\chi(S_2)\leq 0$. It follows that $\chi_-(S) = \chi_-(S_1)+\chi_-(S_2)$. If $S_1 = S_2$ then $\chi(S) = \chi(S_1)\leq 0$ which implies that $\chi_-(S) = \chi_-(S_1)$. Inductively applying these statements to every annulus in the construction this shows that $\chi_-(Y') = \chi_-(X')\leq \chi_-(X^u)$. Applying Lemma~\ref{lem:weightedpfoldedsurfaces} to $Y'$ now completes the proof.
    
\end{proof}

\begin{lem}\label{lem:surfacecomponents}
    If $X$ is a weighted non-separating $\mathbb{Z}_p$-folded surface and $C$ is a component of $\Int(X)$ such that $\partial{C} = \emptyset$ and $w_C>p/2$ then there exists another weighted non-separating $\mathbb{Z}_p$-folded surface $Y$ such that:

    \begin{enumerate}
        \item $[X] = [Y]$
        \item $\chi_-(Y^u)\leq \chi_-(X^u)$
        \item $|\pi_0(Y^u)| = |\pi_0(X^u)| - (2w_C - p)$
        \item $X\setminus C = Y\setminus C$ as weighted $\mathbb{Z}_p$-folded surfaces.
    \end{enumerate}

\end{lem}

\begin{proof}
    
    If $w_C<p$ then let $Y$ be the weighted $\mathbb{Z}_p$-folded surface obtained from $X$ by reversing the orientation of $C$ and changing its weight via $w_C\mapsto p-w_C$. If $w_C = p$ then delete $C$ from the components of $\underline{X}$ entirely to obtain $Y$. If we ignore orientation, then $\underline{Y} \subseteq \underline{X}$ and so $Y$ must still be non-separating. If we let $[C]$ be the image of the fundamental class of $C$ in $H_2(M;\mathbb{Z}_p)$ then $[X]-[Y] = w_C[C] - (p-w_C)[-C] = p[C] = 0$. $Y^u$ can be obtained from $X^u$ by deleting $2w_C - p>0$ of the components of $X^u$ parallel to $C$ and reversing the orientation of the remaining $p-w_C$ components. This means that $\chi_-(Y^u) = \chi_-(X^u) - (2w_C-p)\chi_-(C) \leq \chi_-(X^u)$ and $|\pi_0(Y^u)| = |\pi_0(X^u)| - (2w_C - p)$ as required. 
\end{proof}

\subsection{Reduced $\mathbb{Z}_p$-folded surfaces}\label{sec:reducedpfolded}

We now concern ourselves with a subset of weighted $\mathbb{Z}_p$-folded surfaces characterised in Definition~\ref{dfn:reducedpfolds}.

\begin{dfn}\label{dfn:reducedpfolds}
    A \textit{reduced} $\mathbb{Z}_p$-folded surface is a weighted non-separating $\mathbb{Z}_p$-folded surface such that every weight is less than or equal to $p/2$. 
\end{dfn}

When $p\in \{2,3\}$ then every weight of a reduced $\mathbb{Z}_p$-folded surface is $1$ and so a reduced $\mathbb{Z}_p$-folded surface can be thought of as a non-separating $\mathbb{Z}_p$-folded surface in both cases. Lemma~\ref{lem:reduced} proves the existence of these reduced $\mathbb{Z}_p$-folded surfaces and that they are among the $\mathbb{Z}_p$-Thurston norm realising representatives of any given $H_2(M;\mathbb{Z}_p)$ class:

\begin{lem}\label{lem:reduced}
    For all $x \in H_2(M;\mathbb{Z}_p)$ there exists a reduced $\mathbb{Z}_p$-folded surface $X$ with no removable disks such that $||x||_{\mathbb{Z}_p} = \chi_{-}(X^u)$ and $w_C \leq p/2$ for every component $C \in \pi_0(\Int(\underline{X}))$.
\end{lem}

\begin{proof}

    This is essentially a direct consequence of Lemma~\ref{lem:noremovabledisks}, Lemma~\ref{lem:removenotdisks} and Lemma~\ref{lem:surfacecomponents}. We start with a weighted non-separating $\mathbb{Z}_p$-folded surface $X$ such that $[X]=x$ and $\chi_-(X^u) = ||x||_{\mathbb{Z}_p}$ which exists due to Lemma~\ref{lem:weightedpfoldedsurfaces}. The programme is as follows:

    \begin{enumerate}
        \item If $X$ has any removable disks then apply Lemma~\ref{lem:noremovabledisks}. Since 
        \[|\pi_0(\Sing(\underline{X}))|<\infty,\] 
        it follows that repeatedly applying Lemma~\ref{lem:noremovabledisks} will eventually terminate in a non-separating $\mathbb{Z}_p$-folded surface with no removable disks.\label{step:disks}
        \item If $X$ has any component $C\in \pi_0(\Int(X))$ such that $w_C>p/2$ and $\partial{C}\neq \emptyset$ then Lemma~\ref{lem:weightp/2} implies that $C$ is removable. Apply Lemma~\ref{lem:removenotdisks}. If the resulting weighted non-separating $\mathbb{Z}_p$-folded surface has any removable disks then apply step~\ref{step:disks} again. \label{step:Cp/2}
        \item Repeatedly apply Step~\ref{step:Cp/2}. It follows from Lemma~\ref{lem:noremovabledisks} and Lemma~\ref{lem:removenotdisks} that each time this is done the number of singular arcs of $X$ is strictly reduced. Eventually this terminates because $|\pi_0(\Sing(\underline{X}))|<\infty$.
        \item The result of previous steps is a weighted non-separating $\mathbb{Z}_p$-folded surface with no removable disks in which the only components $C$ of $\Int(\underline{X})$ for which $w_C>p/2$ are closed surfaces in $\underline{X}$. Repeatedly apply Lemma~\ref{lem:surfacecomponents}. This does not affect the number of removable disks or non-surface components with weights greater than $p/2$. $|\pi_0(\underline{X})|<\infty$ and so this will eventually terminate in a weighted non-separating $\mathbb{Z}_p$-folded surface with no removable disks and such that every weight is less than $p/2$.
    \end{enumerate}
    
    Due to the minimality in the definition of $||\cdot||_{\mathbb{Z}_p}$ we have that the resulting weighted non-separating $\mathbb{Z}_p$-folded surface has the required properties.

\end{proof}

Consider the function $\chi_{p}: H_2(M;\mathbb{Z}_p) \rightarrow \mathbb{Z}$ defined by letting $X$ run over all reduced $\mathbb{Z}_p$-folded surfaces realising $x$.

\[\chi_{p}(x) := \max_{[X] = x} \chi(X^u)\]

After proving some preliminary lemmas we show that $\chi_p$ is well-defined by determining an upper bound in Lemma~\ref{lem:eulercharbounds}.

\begin{lem}\label{lem:noremovablediskschi}
    Suppose $X$ is a reduced $\mathbb{Z}_p$-folded surface in $M$. If $p>2$ then there exists a reduced $\mathbb{Z}_p$-folded surface $Y$ with no removable disks such that $[Y] = [X]$ and $\chi(Y^u)\geq \chi(X^u)$ with strict inequality if $X$ has a removable disk. If $p=2$ then there exists a reduced $\mathbb{Z}_2$-folded surface $Y$ with no removable disks such that $[Y]=[X]$ and $\chi(Y) =\chi(X)$.
\end{lem}

\begin{proof}
    Suppose $X$ is a reduced $\mathbb{Z}_p$-folded surface with a removable disk $D$. If $p>2$ since $p$ is an odd prime we have that $w_D<p/2$ so the removal of $D$ as per Lemma~\ref{lem:noremovabledisks} results in a $\mathbb{Z}_p$-folded surface $Y$ such that $[Y] = [X]$, $\chi(Y^u) = \chi(X^u) + p-2w_D > \chi(X^u)$ and $Y$ has fewer singular arcs than $X$. The number of singular arcs in $X$ is finite so iterating we obtain a reduced $\mathbb{Z}_p$-folded surface $Y$ with no removable disks and such that $\chi(Y^u)>\chi(X^u)$.

    If $p=2$ then the removal of $D$ in Lemma~\ref{lem:noremovabledisks} strictly reduces the number of singular arcs and has no effect on the number of components of $M\setminus X$, the maximum weight of components of $\Int(X)$, or the Euler characteristic. Repeatedly removing disks thus will terminate in a $\mathbb{Z}_2$-folded surface $Y$ with no removable disks such that $\chi(Y) = \chi(X)$.

\end{proof}

\begin{lem}\label{lem:eulercharbounds}
    Let $M$ be a closed orientable 3-manifold. $\chi_p(x)\leq p\cdot \beta_2(M;\mathbb{Z}_p)$ for all $x\in H_2(M;\mathbb{Z}_p)$.

\end{lem}

\begin{proof}

    Let $X$ be a reduced $\mathbb{Z}_p$-folded surface embedded in $M$ such that $[X] = x$. Lemma~\ref{lem:noremovablediskschi} implies that there exists another reduced $\mathbb{Z}_p$-folded surface $X_0$ realising the same homology class, with no removable disks and such that $\chi(X_0^u)\geq \chi(X^u)$. Let $Y$ be any connected component of $\underline{X_0}$ inheriting the weights of $X_0$. 
    
    Suppose $\Sing(Y) = \emptyset$. Then $\underline{Y}$ is a closed connected oriented surface so $\chi(\underline{Y}) \leq 2$. $Y^u$ is constructed by pushing $\underline{Y}$ in a normal direction to generate $w_{\underline{Y}}\leq p/2$ parallel copies so $\chi(Y^u) = w_{\underline{Y}}\cdot \chi(\underline{Y}) \leq w_{\underline{Y}} \cdot 2 \leq p$.

    Suppose $\Sing(Y) \neq \emptyset$. Since we assumed that there are no removable disks in $Y$, it follows from Lemma~\ref{lem:removabledisksorclosed} that if any component of $Y_{cut}$ is homeomorphic to a disk then $Y^u$ is a $\mathbb{Z}_p$-crosscap and $\Int(Y)$ consists of a single component of weight $1$ and $\chi(Y^u)=1$. Conversely if no component of $Y_{cut}$ is a disk then every component has non-positive Euler characteristic and so $\chi(Y^u)\leq 0$. 
    
    Hence in general for any $Y \in \pi_0(X_0)$ it follows that $\chi(Y^u)\leq p$. Lemma~\ref{lem:components} implies that the number of such connected components of $\underline{X_0}$ is bounded by the dimension of $H_2(M;\mathbb{Z}_p)$ as a $\mathbb{Z}_p$-vector space. Hence 
    \[\chi(X^u)\leq \chi(X_0^u) = \sum_{Y\in \pi_0(X)}\chi(Y^u) \leq p\cdot \beta_2(M;\mathbb{Z}_p).\] 
    This completes the proof.

\end{proof}

Lemma~\ref{lem:equatingfunctions} shows that the original $||\cdot||_{\mathbb{Z}_p}$ function can be recovered from $\chi_p$ in spaces where $H_2(M;\mathbb{Z}_p) = \mathbb{Z}_p$.

\begin{lem}\label{lem:equatingfunctions}
    If $H_2(M;\mathbb{Z}_p) = \mathbb{Z}_p$, then $||x||_{\mathbb{Z}_p} = \max(-\chi_p(x),0)$ for all $x\in H_2(M;\mathbb{Z}_p)$.
\end{lem}

\begin{proof}

    Let $X$ be a reduced $\mathbb{Z}_p$-folded surface in $M$ with no removable disks such that $[X] = x$ (see Lemma~\ref{lem:reduced}). If $x=0$ then $\chi_p(0) = ||0||_{\mathbb{Z}_p}=0$ satisfies the equality trivially. 
    
    Now assume that $x\neq 0$. Lemma~\ref{lem:components} implies that $\underline{X}$ must be connected since it is non-separating and non-empty. The proof of Lemma~\ref{lem:eulercharbounds} shows that $\chi(X^u)>0$ in precisely two cases \- when $X^u$ is a $\mathbb{Z}_p$-crosscap and when $X^u$ is a disjoint union of $w_{\underline{X}} \leq p/2$ parallel $2$-spheres. In both cases $\chi_-(X^u) = 0 = \max(-\chi(X^u),0)$. 

    If $\chi(X^u)\leq 0$ then $\underline{X}$ cannot be a $2$-sphere and no component of $\Int(X)$ can be a disk due to Lemma~\ref{lem:removabledisksorclosed}. This implies that $\chi(C)\leq 0$ for all $C \in \pi_0(\Int(X^u_{cut}))$. By definition $\chi_-(X^u) = \chi_-(X^u_{cut}) = -\chi(X^u_{cut}) = -\chi(X^u) = \max(-\chi(X^u),0)$. 

    Now take the minimum of both sides of the equation 
    \[\chi_-(X^u) = \max(-\chi(X^u),0)\] 
    where $X$ runs over all reduced $\mathbb{Z}_p$-folded surfaces realising $x$ with no removable disks. Lemma~\ref{lem:reduced} implies that the left hand side equates to $||x||_{\mathbb{Z}_p}$ and Lemma~\ref{lem:noremovablediskschi} implies that the right hand side equates to $\max(-\chi_p(x),0)$, completing the proof.

\end{proof}

When $p=2$ or $3$ we have a special case in which every weight of a reduced weighted $\mathbb{Z}_p$-folded surface must be equal to $1$ and so the weights can be dropped and $||x||_{\mathbb{Z}_p}$ is determined by the maximum Euler characteristic of a non-separating $\mathbb{Z}_p$-folded surface $X$ such that $[X]=x$. The main result of a paper by Bredon and Wood on non-orientable surfaces in lens spaces~\cite{bredonwood1969} determines $\chi_2:H_2(L(2n,q);\mathbb{Z}_2)\rightarrow \mathbb{Z}$. We recall the special case of Bredon and Wood's result for $L(2n,1)$ for comparison with Theorem~\ref{thm:l3n1norm} and Conjecture~\ref{conj:lpn1}.

\begin{thm}[Non-orientable surfaces in orientable 3-manifolds~\cite{bredonwood1969}]\label{thm:l2nqnorm}
    Let $M = L(2n,1)$ and $x$ be the non-trivial element of $H_2(M; \mathbb{Z}_2) = \mathbb{Z}_2$. Then $\chi_2(x) = 2-n$. It follows immediately from Lemma~\ref{lem:equatingfunctions} that $||x||_{\mathbb{Z}_2} = \max(n-2,0)$.
\end{thm}

The main result of this paper is Theorem~\ref{thm:l3n1norm} which determines the $\mathbb{Z}_3$-Thurston norm and $\chi_3$ for the family of lens spaces $L(3n,1)$ where $n \in \mathbb{Z}^{>0}$.

\begin{thm}\label{thm:l3n1norm}
    Let $M = L(3n,1)$ and $x \in H_2(L(3n,1);\mathbb{Z}_3)$ such that $x\neq 0$. 
    \[\chi_3(x) = 2-n\] 
    and 
    \[||x||_{\mathbb{Z}_3} = \max(n-2,0).\]
\end{thm}

Interactive 3D graphics of the maximal Euler characteristic non-orientable surfaces in $L(2n,1)$ and the maximal Euler characteristic non-separating $\mathbb{Z}_3$-folded surfaces in $L(3n,1)$ for $n\in \{1,2,3\}$ are available at \href{https://willhobkirk.com/}{willhobkirk.com}.

We also posit Conjecture~\ref{conj:lpn1} based on data gathered from the construction of \textit{canonical $\mathbb{Z}_p$-folded surfaces} using a large census of one vertex triangulations via a procedure described in Section~\ref{sec:canonicalpfoldedsurfaces}.

\begin{conj}\label{conj:lpn1}
    Let $M = L(pn,1)$. Then 
    \[\max \left\{\chi_p(x) : x \in H_2(M;\mathbb{Z}_p)\setminus \{0\}\right\} = 2-n.\]
\end{conj}

\section{Canonical $\mathbb{Z}_p$-folded surfaces}\label{sec:canonicalpfoldedsurfaces}

\subsection{Construction}\label{sec:normalpfoldedsurfacesconstruction}

See Section~\ref{sec:triangulations} for the definition of a triangulation. 
Throughout Sections~\ref{sec:normalpfoldedsurfacesconstruction} and~\ref{sec:normalisreduced} let $\mathcal{T}$ be a 1-vertex triangulation of a closed oriented $3$-manifold $M$ and let $\varphi \in H^1(M;\mathbb{Z}_p)$. In~\cite{Jaco_2009} for the case $p=2$ the authors give a construction of a canonical normal surface $S_{\varphi}(\mathcal{T})$ realising a class in $H_2(M;\mathbb{Z}_2)$ dual to $\varphi$ and combine this with the result of Bredon and Wood~\cite{bredonwood1969} to determine the complexity of the family of lens spaces $L(2n,1)$ for $n\in \mathbb{Z}^{>0}$. For odd numbers $p>2$ we construct a canonical $\mathbb{Z}_p$-folded surface $S_{\varphi}(\mathcal{T})$ dual to any $\varphi \in H^1(M;\mathbb{Z}_p)$. In Section~\ref{sec:naivecomplexity} we combine this with Theorem~\ref{thm:l3n1norm} to determine a new lower bound on the complexity of the lens spaces $L(3n,1)$. This connection with $3$-manifold complexity was the original motivation for this research and we provide the following construction in full generality for any odd $p>2$ in the hope that it can be used for similar results in the future. 

\begin{rem}
    This construction generalises easily to compact oriented $3$-manifolds. In this context it yields proper embeddings of canonical $\mathbb{Z}_p$-folded surfaces with boundary that realise any $H_2(M,\partial M;\mathbb{Z}_p)$ class. 
\end{rem}

Each oriented edge $e$ of $\mathcal{T}$ determines a class $[e]\in H_1(M;\mathbb{Z}_p)$ since $\mathcal{T}$ has a single vertex. Let $\pi_p:\mathbb{Z}\rightarrow \mathbb{Z}_p$ be the standard quotient. Since $p$ is odd we can define a unique simplicial $1$-cochain $\tilde{\varphi}$ such that $-p/2<\tilde{\varphi}(e)<p/2$ for every oriented edge $e$ of $\mathcal{T}$ and such that $\pi_p\circ\tilde{\varphi}$ is a $1$-cocycle representing $\varphi$. We will now construct a canonical $\mathbb{Z}_p$-folded surface $S_{\varphi}$ that is in general position relative to the $1$-skeleton of the triangulation such that the oriented $2$-manifold $\text{Int}(S_{\varphi})$ intersects each oriented edge $e \in \mathcal{T}^1$ in a signed total of $\tilde{\varphi}(e)$ points. Counting transverse intersections modulo $p$ shows that $[S_{\varphi}] \in H_2(M;\mathbb{Z}_p)$ is the Poincar\'e dual of $\varphi \in H^1(M;\mathbb{Z}_p)$. The construction depends only on the triangulation and the class of $\varphi$ and so we call it the canonical $\mathbb{Z}_p$-folded surface dual to $\varphi$ in $\mathcal{T}$. 

We will describe $S_{\varphi}$ by describing its intersection with each edge, triangle and tetrahedron of $\mathcal{T}$. 

Let $e$ be an edge of $\mathcal{T}$. Simple modular arithmetic verifies that the value of $|\tilde{\varphi}(e)|<p/2$ is independent of a choice of orientation of $e$. Fix $|\tilde{\varphi}(e)|$ equally spaced points in the interior of each edge $e\in \mathcal{T}^1$ to comprise the vertices of $\text{Int}(S_{\varphi}) \cap \mathcal{T}^1$.

Let $t$ be an oriented triangle of $\mathcal{T}$ and let $e_1,e_2,e_3$ be the edges of its boundary with induced orientations such that $\partial(t) = \sum_{i=1}^3 e_i$. Clearly the bounds on $\tilde{\varphi}$ imply the inequalities $-\frac{3}{2}p<d\tilde{\varphi}(t) = \tilde{\varphi}(\partial(t)) <\frac{3}{2}p$. $\pi_p\circ\tilde{\varphi}$ is a $1$-cocycle so $\pi_p\circ d\tilde{\varphi} = 0$ and hence $d\tilde{\varphi}(t) \in \{-p,0,p\}$ for all oriented triangles $t$. Fix one point at the barycenter of each triangle $t \in \mathcal{T}^2$ with the property that $d\tilde{\varphi}(t) = \pm p$ to comprise the vertices of $\text{Sing}(S_{\varphi})\cap \mathcal{T}^2$. 

Note that it is not strictly necessary to rely on the geometry of the standard $3$-simplex for the notions of \textit{equally spaced} and \textit{barycenter} but doing so allows us to avoid a tedious proof that the construction is unique up to piecewise-linear normal isotopy if any set of points is allowed.

Let $t$ be an oriented triangle of $\mathcal{T}$. If $d\tilde{\varphi}(t) = 0$ then $\sum_{i=1}^3 \tilde{\varphi}(e_i) = 0$ implies that without loss of generality we have $|\tilde{\varphi}(e_1)| = |\tilde{\varphi}(e_2)| + |\tilde{\varphi}(e_3)|$. There is a unique non-crossing pairing of the vertices we have fixed in the interior of $e_1$ with the vertices in the interiors of $e_2$ and $e_3$. Linear spans of these vertex pairs form a proper embedding of $\frac{1}{2}\sum_{i=1}^3 |\tilde{\varphi(e_i)}|$ intervals in $t$. The orientation of $t$ induces an orientation of these intervals such that the induced orientation of the vertices in the interior of $e_1$ has the same sign as $\tilde{\varphi}(e_1)$. If $d\tilde{\varphi}(t) = p $ (resp. $-p$) then the fact that $\tilde{\varphi}(e_{i\in\{1,2,3\}})<p/2$ (resp. $>-p/2$) implies that $\tilde{\varphi}(e_{i\in\{1,2,3\}})>0$ (resp. $<0$). This implies that $p = |d\tilde{\varphi}(t)| = \sum_{i=1}^3 |\tilde{\varphi}(e_i)|$ and so we have fixed $p$ vertices on the boundary of $t$ and one in the interior. Take the intervals spanned by the interior vertex with each boundary vertex to form an embedding of $K_{p,1}$ in $t$. The orientation of $t$ induces an orientation of these intervals such that the induced orientation on each boundary vertex has the same sign as $d\tilde{\varphi}(t)$. When the orientation of $t$ is reversed these linearly embedded $1$-complexes are fixed with reversed orientations. When this construction is performed for each triangle of $\mathcal{T}$ the result is a 1-complex embedded in $\mathcal{T}^2$ in general position relative to $\mathcal{T}^1$ which comprises the intersection $S_{\varphi} \cap \mathcal{T}^2$. When $p=3$ there are $3$ different ways that $S_{\varphi}$ may intersect a triangle up to rotational symmetry and reversal of orientation. When $p=5$ there are $5$. These are pictured in Figure~\ref{fig:p35triangleboundary}.

\begin{figure}
    \centering
    \includegraphics[width=0.6\textwidth]{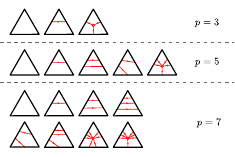}
    \caption{Possible boundary patterns of triangles $p=3,5$ up to symmetry and orientation reversal.}\label{fig:p35triangleboundary}
\end{figure}

Now let $\tilde{\mathfrak{t}}$ be the preimage of tetrahedron $\mathfrak{t}$ of $\mathcal{T}$ oriented in the same way as its image in $M$. Let $\partial(\tilde{\mathfrak{t}}) = \sum_{i=1}^4 \tilde{t_i}$ where the $\tilde{t_i}$ are the four faces of $\tilde{\mathfrak{t}}$ with induced orientations and $t_i$ are their images in $|\mathcal{T}|$. The orientation of each $t_i$ induces an orientation of the 1-complex $S_{\varphi} \cap t_i$ in the way described above which lift together to an oriented 1-complex on the boundary of the preimage tetrahedron $\partial(\
tilde{\mathfrak{t}})$. The way the orientations of the intervals of this boundary pattern were determined ensures that components of the boundary pattern are equivalent up to orientation preserving homeomorphism to a union of oriented simple closed curves or a bipartite $p$-valent graph $(A,B,E)$ with all edges oriented from a vertex in $A$ to a vertex in $B$. When $p=3$ there are precisely $4$ possible boundary patterns of a tetrahedron up to symmetry and orientation reversal and these are pictured in Figure~\ref{fig:p3boundarypatterns}.

\begin{figure}
    \centering
    \includegraphics[width=0.6\textwidth]{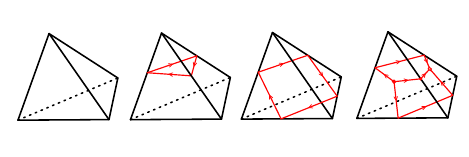}
    \caption{Possible boundary patterns of tetrahedron preimages when $p=3$ up to symmetries of the tetrahedron.}\label{fig:p3boundarypatterns}
\end{figure}

$0 = d^2 \tilde{\varphi}(\mathfrak{t}) = \sum_{i=1}^4 d\tilde{\varphi}(t_i)$ implies without loss of generality that either:

\begin{enumerate}
    \item $d\tilde{\varphi}(t_{1,2,3,4})=0$ \label{case:nosingularlocus}
    \item $d\tilde{\varphi}(t_{1})= -d\tilde{\varphi}(t_{2}) = p$ and $d\tilde{\varphi}(t_{3,4}) = 0$ \label{case:onesingularlocus}
    \item $d\tilde{\varphi}(t_{1,2})= -d\tilde{\varphi}(t_{3,4}) = p$ \label{case:twosingularloci}
\end{enumerate}

On closer inspection case~\ref{case:twosingularloci} is not actually possible. This is because if we let $\partial(t_1) = e_1+e_2+e_3$ then as we noted above $\tilde{\varphi}(e_{1,2,3})>0$. The edges $e_1,e_2,e_3$ also appear as part of the boundary of $t_2,t_3,t_4$ with the opposite induced orientation. In particular this implies that $d\tilde{\varphi}(t_2)<p$ which is a contradiction. This leaves us with case~\ref{case:nosingularlocus} in which the boundary pattern on $\partial(\tilde{\mathfrak{t}})$ is an oriented compact 1-manifold and case~\ref{case:onesingularlocus} where one component of the boundary pattern is the bipartite multigraph $G_p$ with two vertices and $p$ edges.

We now fill in the interior of $\tilde{\mathfrak{t}}$ with oriented linear $2$-simplices that are the convex hulls of the vertices of the boundary pattern such that the resulting space is a $\mathbb{Z}_p$-folded surface in the interior of the tetrahedron inducing the required orientation on its boundary. To do this in a well-defined way we need to introduce some orderings of the vertices of the boundary pattern. 

Let $C$ be a simple closed curve component of the boundary pattern. It is well-known from the theory of normal surfaces and easily verifiable that $C$ intersects either $3$ or $4$ edges of $\tilde{\mathfrak{t}}$ once each. In the former case simply take the linear $2$-simplex spanned by the three intersection points and orient it such that it induces the orientation of $C$ on its boundary. In the latter case let $\{v_1, v_2,v_3,v_4\}$ be the vertices labelled in the cyclic order induced by the orientation of $C$ such that $v_1$ is either of the two vertices coloured green in Figure~\ref{fig:normaldisks} (a different choice will yield a PL-isotopic $\mathbb{Z}_p$-folded surface so it does not really matter too much). Take the linear $2$-simplices spanned by $\{v_1,v_2,v_3\}$ and $\{v_1,v_3,v_4\}$ oriented such that they induce the orientation of $C$ on the boundary. We will call the resulting oriented triangulated disk a \textit{normal disk}.

\begin{figure}
    \centering
    \includegraphics[width=0.3 \textwidth]{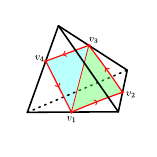}
    \caption{Triangulation of normal disks and vertex labelling.}\label{fig:normaldisks}
\end{figure}

Now let $C$ be one of the $p$ oriented paths of a component of the boundary pattern homeomorphic to $G_p$ described above and let $v_-$ be the source vertex and $v_+$ be the sink vertex in the interiors of two faces of $\tilde{\mathfrak{t}}$. $C$ intersects the edges of $\tilde{\mathfrak{t}}$ transversally in $n \in \{1,2,3\}$ points which we label $\{v_i\}_{i=1}^n$ - see Figure~\ref{fig:normalpfolds}. If $n=1$ take the linear $2$-simplex spanned by $\{v_-,v_1,v_+\}$. If $n=2$ take the two linear $2$-simplices spanned by $\{v_-,v_1,v_2\}$ and $\{v_-,v_2,v_+\}$. If $n=3$ take the three linear $2$-simplices spanned by $\{v_-,v_1,v_2\}, \{v_-,v_2,v_3\}$ and $\{v_-,v_3,v_+\}$. Orient these $2$-simplices such that they induce the orientation of $C$ on their boundary. When this is performed for each path from $v_-$ to $v_+$ in the embedding of $G_p\hookrightarrow \partial \tilde{\mathfrak{t}}$ then the result is $p$ oriented triangulated disks identified along a singular interval spanned by $\{v_+,v_-\}$. The orientation of the disks induces a consistent orientation on the singular edge from $v_+$ to $v_-$ and induces the orientation of $G_p$ on the remaining boundary where they meet $\partial \tilde{\mathfrak{t}}$. We call the union of these disks a \textit{normal p-fold}. 

\begin{figure}
    \centering
    \includegraphics[width=0.6\textwidth]{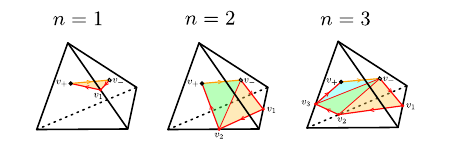}
    \caption{Triangulation of normal $p$-folds and vertex labelling.}\label{fig:normalpfolds}
\end{figure}

Normal $p$-folds and normal disks do not coexist in the same tetrahedron since they imply contradictory signs of $\tilde{\varphi}$ on adjacent pairs of edges (this is an easy case checking exercise). It is also easily checked that the $2$-simplices used in the construction of normal $p$-folds and normal disks intersect in the interior of $\tilde{\mathfrak{t}}$ only if they share an entire edge and so the resulting oriented $2$-complex is a $\mathbb{Z}_p$-folded surface in the interior of $\tilde{\mathfrak{t}}$ which induces the oriented boundary pattern on $\partial \tilde{\mathfrak{t}}$. $M$ is oriented so the identifications between faces of the tetrahedra are orientation reversing so when the normal pieces in each $\tilde{\mathfrak{t}}$ are identified along their boundaries in the image $\mathfrak{t}$ the orientation is preserved by the gluing to form a $\mathbb{Z}_p$-folded surface $S_{\varphi}(\mathcal{T})$. $\text{Sing}(S_{\varphi}(\mathcal{T}))$ is the oriented $1$-manifold given by the image of the oriented intervals from $v_+$ to $v_-$ in each tetrahedron where Case~\ref{case:onesingularlocus} applies. It follows from the definition of Poincar\'e duality in terms of the cap product and by counting signed transverse intersections of $S_{\varphi}$ with the oriented edges of the $1$-skeleton of $\mathcal{T}$ that $[S_{\varphi}(\mathcal{T})]$ is the Poincar\'e dual of $\varphi$. 

\subsection{Relation to reduced $\mathbb{Z}_p$-folded surfaces}\label{sec:normalisreduced}

Lemma~\ref{lem:normalpfoldsarereduced} associates a reduced $\mathbb{Z}_p$-folded surface to each canonical $\mathbb{Z}_p$-folded surface that has the same Euler characteristic. Corollary~\ref{cor:chi3bound} therefore bounds the Euler characteristic of a canonical surface by $\chi_p$ of its dual class. This bound can then be used to derive some information about the triangulations of a particular $3$-manifold that exist, and we demonstrate this in the proof of Corollary~\ref{cor:naivecomplexitybound} in Section~\ref{sec:l3n1normproof}. Via Lemma~\ref{lem:equatingfunctions} one can easily translate this into an analogous result relating the Euler characteristic of canonical $\mathbb{Z}_p$-folded surfaces to the $\mathbb{Z}_p$-Thurston norm when $H_2(M;\mathbb{Z}_p) = \mathbb{Z}_p$.

\begin{lem}\label{lem:normalpfoldsarereduced}
    For each canonical $\mathbb{Z}_p$-folded surface $S_{\varphi}(\mathcal{T})$ there exists a reduced $\mathbb{Z}_p$-folded surface $X$ such that $[X] = [S_{\varphi}(\mathcal{T})]$ and $\chi(X^u) = \chi(S_{\varphi}(\mathcal{T}))$.

\end{lem}

\begin{proof}
    Set $\underline{X} = S_{\varphi}$ and $w_C = 1$ for every $C \in \pi_0(\text{Int}(\underline{X}))$. If $\underline{X}$ is non-separating in $M$ then we are done.

    Otherwise, there must exist a component $N$ of $M \setminus \underline{X}$ that does not contain the single vertex of $\mathcal{T}$. Let $\mathfrak{t}$ be the image of a tetrahedron of $\mathcal{T}$ and let $N'$ be a connected component of $N\cap\mathfrak{t}$. Recall that the intersection $X \cap \mathfrak{t}$ consists of normal $p$-folds or normal disks but not both (see Section~\ref{sec:normalpfoldedsurfacesconstruction}) and so up to isotopy and reversal of orientation of $\partial(N')\subset S_{\varphi} \cap \mathfrak{t}$ there are precisely $5$ possibilities for $N',\partial(N') \hookrightarrow \mathfrak{t}$ pictured in Figure~\ref{fig:N'intersectionwitht}.
    
    Let $N_+$ (resp. $N_-$) be the subsets of $\text{Int}(S_{\varphi}) \cap \partial N$ on the boundary of $N$ with consistent (resp. inconsistent) orientation with the orientation induced by $N\subset M$. Following the proof of Lemma~\ref{lem:weightedpfoldedsurfaces} we can obtain a new weighted $\mathbb{Z}_p$-folded surface $Y$ by decreasing the weights of components of $N_+$ by one and increasing the weights of components of $N_-$ by one. The five cases in Figure~\ref{fig:N'intersectionwitht} clearly show that there is a cell decomposition of $\overline{N_+}$ and $\overline{N_-}$ with the same number of $2$-cells and $1$-cells. For each edge $e$ of $\mathcal{T}$ the degree of the vertices in $e\cap N_+$ and $e \cap N_-$ in the cell decompositions is equal to the degree of $e$ in $\mathcal{T}$ and so the number of $0$-cells is also the same. Hence $\chi(N_+) = \chi(N_-)$ and so $\chi(Y^u) = \chi(X^u) + \chi(N_-) - \chi(N_+) = \chi(X^u)$. Set $X=Y$. Repeat this process fixing $N$ until the number of components of $M \setminus \underline{X}$ decreases, then set $N$ to be another component of $M \setminus \underline{X}$ that does not contain the vertex of $\mathcal{T}$. Iterate until $M \setminus \underline{X}$ is connected. 
    
    Once this is complete $X$ is clearly a non-separating weighted surface with the property that $\chi(X^u) = \chi(S_{\varphi}(\mathcal{T}))$. Singular homology arguments from the proof of Lemma~\ref{lem:weightedpfoldedsurfaces} imply that $[X] = [S_{\varphi}(\mathcal{T})]$. For every edge $e$ of $\mathcal{T}$ oriented such that $\tilde{\varphi}(e)\geq 0$ it is also clear from the proof that the sum of the weights of the components of $\text{Int}(\underline{X})$ intersecting $e$ is $\tilde{\varphi}(e)$ at every step. In particular since $\tilde{\varphi}(e)<p/2$ this bounds every weight and shows that $X$ is also a reduced $\mathbb{Z}_p$-folded surface as required.
    
\end{proof}

\begin{figure}
    \centering
    \includegraphics[width=0.6\textwidth]{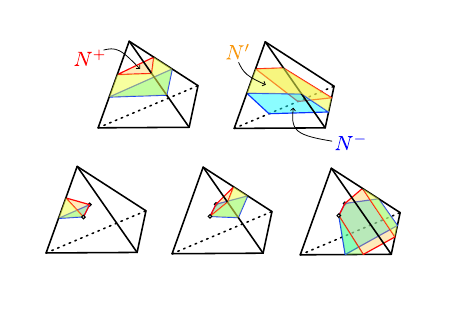}
    \caption{$N'$ a connected component of $\mathfrak{t}\cap N$ and $\partial N'$ with the decomposition into oriented cells in $N_+$ and $N_-$ induced by the orientation of $\underline{X}$. All cases up to symmetries of $\mathfrak{t}$ and reversals of orientation of the folded surface.}\label{fig:N'intersectionwitht}
\end{figure}

Corollary~\ref{cor:chi3bound} follows from Lemma~\ref{lem:normalpfoldsarereduced}

\begin{cor}\label{cor:chi3bound}
    Suppose $S_{\varphi}(\mathcal{T})$ is a canonical $\mathbb{Z}_p$-folded surface dual to a class $\varphi \in H^1(\mathbb{Z}_p)$ where $\mathcal{T}$ is a one-vertex triangulation of a closed orientable irreducible $3$-manifold $M$. It follows that $\chi(S_{\varphi}(\mathcal{T})) \leq \chi_p(\varphi^*)$. Moreover if $\mathcal{C}$ is a collection of $1$-vertex triangulations of $M$ then $\max_{\mathcal{T}\in \mathcal{C}}\chi(S_{\varphi}(\mathcal{T})) \leq \chi_p(\varphi^*)$.
\end{cor}

\begin{proof}
    This follows immediately from Lemma~\ref{lem:normalpfoldsarereduced} and the definition of $\chi_p$ since $S_{\varphi}$ is a $\mathbb{Z}_p$-folded surface realising the dual class $\varphi^*$.
\end{proof}

We posit Conjecture~\ref{conj:maxcanonicalischi_p} which states that the maximum Euler characteristic of a reduced $\mathbb{Z}_p$-folded surface is obtained by a canonical $\mathbb{Z}_p$-folded surface representing the same class for some one-vertex triangulation.

\begin{conj}\label{conj:maxcanonicalischi_p}
    Let $\mathcal{C}$ be the set of all one-vertex triangulations of a closed orientable irreducible $3$-manifold $M$ and let $\varphi \in H^1(M;\mathbb{Z}_p)$. The following holds:
    \[\max_{\mathcal{T}\in \mathcal{C}}\chi(S_{\varphi}(\mathcal{T})) = \chi_p(\varphi^*)\]
\end{conj}

This conjecture holds for the lens space families $L(2n,q)$ and $L(3n,1)$ for all $n\in \mathbb{Z}^{>0}$ and integer $q$ coprime to $2n$ with equality given by the minimal layered triangulations. We have not been able to come up with a neat topological construction that shows it holds in general.

\section{$\mathbb{Z}_3$-folded surface surgeries}\label{sec:3foldedsurfs}

\subsection{Generalised compression surgeries}\label{sec:3foldsurgeries}

Let $X$ be a $\mathbb{Z}_3$-folded surface embedded in an orientable closed $3$-manifold $M$. We now define a family of surgeries called \textit{2k-compressions} through which a new embedded $\mathbb{Z}_3$-folded surface can be obtained realising the same homology class. This technique is inspired by the theory of compressible surfaces developed in~\cite{haken1961}. The surgeries will all depend on an open disk $D$ embedded in $M$ such that $\overline{D}\cap X = \partial(D)$ and $\partial(D)\cap \text{Sing}(X)$ is a finite set of $2k$ isolated points on the boundary. We call $D$ a \textit{$2k$-compression disk}. We will start by describing the \textit{0-compression surgery} which is essentially the same as the usual disk compression encountered in normal surface theory applied to $\Int(X)$. The \textit{$2k$-compression surgery} is a generalisation of this defined up to a choice of orientation. For the purpose of proving Theorem~\ref{thm:l3n1norm} the only required surgeries are those for $k\in\{0,1,2\}$ for reasons which will become clear. It is crucial to note that even if $X$ was non-separating, $X'$ will not necessarily be. We will deal with this problem in Section~\ref{sec:surfaceelimination}.

Let $N$ be a closed regular neighbourhood of $D$ in $M$ (starting with a triangulation for which $X$ and $D$ are subcomplexes). Let $D_1,D_2$ be the two open disk components of $\partial(N) \setminus X$ belonging to the same component of $N \setminus X$ as $D$. Let $Y$ be the $2$-polyhedron $X\sqcup D_1 \sqcup D_2$ (in particular $Y$ is a simple polyhedron - see~\cite{matveev:2003}). The inclusion $X\subset Y$ induces an orientation on $\Int(Y) \setminus (D_1\cup D_2)$.

If $k=0$ then there is an open annulus $A \in \pi_0(\Int(Y))$ such that $\partial(D)\in A$ and $\partial(A) = \partial(D_1)\sqcup \partial(D_2)$. Let $X' := Y \setminus A$ and note that orientations of $D_1\sqcup D_2 \subset X'$ can clearly be chosen such that $X'$ is a $\mathbb{Z}_3$-folded surface. By fixing a cell complex structure for $Y$ a standard Euler characteristic computation can be used to show that $\chi(X') = \chi(X)+2$. We say that $X'$ is the result of 0-compression of $X$ along $D$. See Figure~\ref{fig:interiorcompression} for a visualisation.

\begin{figure}
    \centering
    \begin{tabular}{c@{}c@{}c}
        $X$ & $Y$ & $X'$\\
        \begin{annotate}{\includegraphics[width=0.25\textwidth]{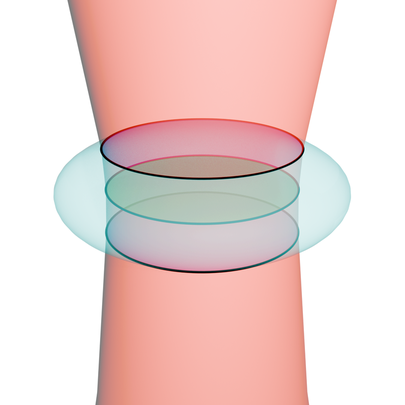}}{0.25}

            \callout{-6,6}{$N$}{-5,2};
            \callout{6,6}{$D_1$}{0,2};
            \callout{6,3}{$D$}{0,0};
            \callout{6,-6}{$D_2$}{0,-2};
            \draw[thick,rounded corners = 10pt, \arcol] (-6,0) -- (-4,2.5) -- (0,3.5) -- (4,2.5) -- (6,0) -- (4,-2.5) -- (0,-3) -- (-4,-2.5) -- cycle;
        \end{annotate}
        & 
        \begin{annotate}{\includegraphics[width=0.25\textwidth]{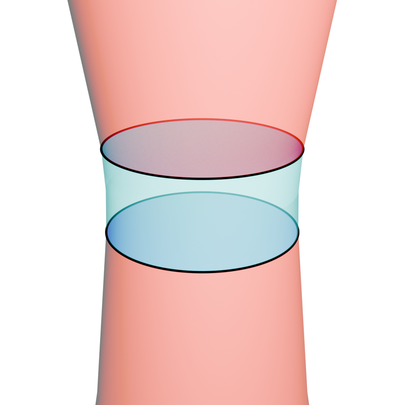}}{0.25}

            \callout{-6,6}{$N'$}{-4,2.5}
            \callout{6,-3}{$A$}{0,-1}
            \draw[thick,rounded corners = 10pt, \arcol] (-4,-2) -- (-4,2.5) -- (0,3.5) -- (4,2.5) -- (4,-2) -- (0,-3) -- cycle;
        \end{annotate}
        &
        \begin{annotate}{\includegraphics[width=0.25\textwidth]{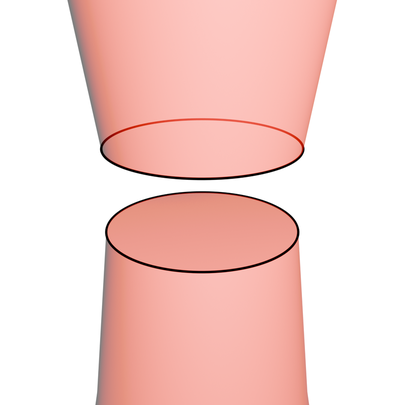}}{0.25}

        \end{annotate}
    \end{tabular}

    \caption{The  $0$-compression surgery \- $X, Y$ and $X'$. The compression disk $D$, its regular neighbourhood $N$, the disks $D_1,D_2\subset \partial N$, the region $N'$ of $N\setminus X$ containing the parallel disks $D,D_1,D_2$ and the annulus $A$ which is a component of $\Int(Y)$.}
    \label{fig:interiorcompression}
\end{figure}

Now consider the case where $k>0$. Again, let $N'$ be the component of $N\setminus X$ that contains $D$ and such that $D_1\sqcup D_2 \subset \partial N'$. The orientation of $N'\subset M$ induces an orientation on each component of $\Int(Y)$ contained in its boundary. Other than $D_1,D_2$, there are $2k$ components of $\Int(Y)$ contained in $\partial N'$, all homeomorphic to disks. Let $\{A^+_i\}_{i=1}^{k}$ be the subset of these components such that the orientation induced on $A^+_i$ by $N' \subset M$ is the same as the orientation that $A^+_i$ inherits from the component of $\Int(X)$ which contains it. Similarly let $\{A^-_i\}_{i=1}^{k}$ be those components such that the two orientations are opposite one another. Let $X^{\pm} := Y \setminus \bigsqcup_{i=1}^k A_i^{\pm}$ and orient $\Int(X^{\pm})$ such that $X^{\pm}\setminus N$ is identical to $X \setminus N$. 

It follows from checking that the local conditions from Definition~\ref{dfn:pfoldedsurface} are satisfied everywhere inside $N$ that $X^{\pm}$ is a $\mathbb{Z}_3$-folded surface. We say that $X^{\pm}$ is the result of a $2k$-gon compression (of sign $\pm$) of $X$ along $D$. Note that in general $X^+$ is isotopic to $X^-$ for $k=1$ but not for $k=2$. Fixing a cell decomposition for $Y$ in which $D_1,D_2$ and $A_i^{\pm}$ are $2$-cells, a standard Euler characteristic computation shows that $\chi(X^{\pm}) = \chi(X) + 2-k$. See Figure~\ref{fig:2goncompression} for a visualisation of $2$-compressions and Figure~\ref{fig:4goncompression} for a visualisation of $4$-compressions for each of $X^{+}$ and $X^{-}$.

\begin{figure}
    \centering
    \begin{tabular}{c@{}c@{}c@{}c}
        $X$ & $Y$ & $X^+$ & $X^-$ \\

        \begin{annotate}{\includegraphics[width = 0.175\linewidth]{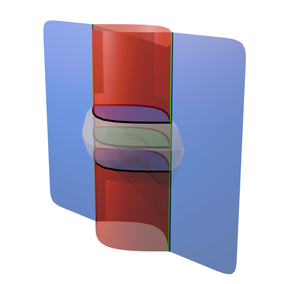}}{0.175} 
            \callout{-6,3}{\tiny $D_1$}{-0.5,1.25};
            \callout{-6,0}{\tiny $D$}{-0.5,0.25};
            \callout{-6,-3}{\tiny $D_2$}{-0.5,-0.75};
        \end{annotate}
        &
        \begin{annotate}{\includegraphics[width = 0.175\linewidth]{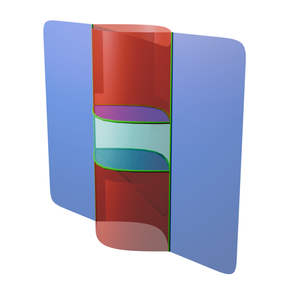}}{0.175} 
            \draw[thick,\arcol, rounded corners=2pt] (1.5,-1.75) -- (-1.5,-1.75) -- (-2.5,-1.25) -- (-2.5,2) -- (1,2) -- (1.5,1.5) -- cycle;
            \callout{-6,-6}{\tiny $N'$}{-1.5,-2};
            \callout{6,-6}{\tiny $A^+$ }{-1,0};
        \end{annotate}
        
        & 
        \begin{annotate}{\includegraphics[width = 0.175\linewidth]{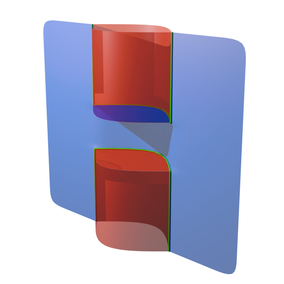}}{0.175}
            
        \end{annotate}
        & 
        \begin{annotate}{\includegraphics[width = 0.175\linewidth]{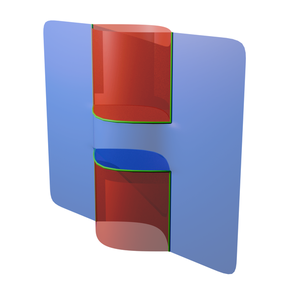}}{0.175}
            
        \end{annotate}
        \\
        \begin{annotate}{\includegraphics[width = 0.175\linewidth]{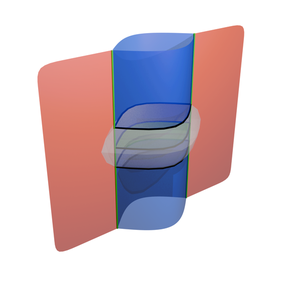}}{0.175} 
            
            \callout{6,6}{\tiny $N$}{3 ,2};
            \draw[thick, \arcol, rounded corners](-2.25 ,-1.75) -- (-2.25,1) -- (-1 ,2.1) -- (3 ,2.1) -- (3 ,-1) -- (1,-2) -- cycle;
        \end{annotate}
        &
        \begin{annotate}{\includegraphics[width = 0.175\linewidth]{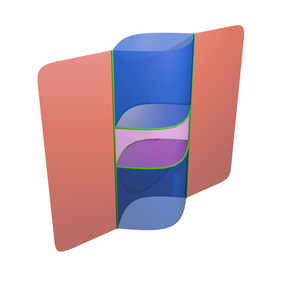}}{0.175}
            \callout{6,-6}{\tiny $A^-$ (front)}{1,0};
            \callout{-6,6}{\tiny $A^+$ (back)}{0,1};
        \end{annotate}
        & 
        \begin{annotate}{\includegraphics[width = 0.175\linewidth]{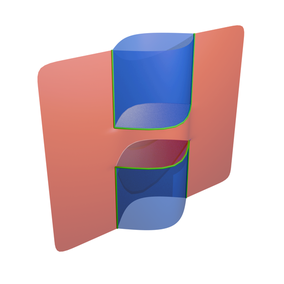}}{0.175}
            
        \end{annotate}
        & 
        \begin{annotate}{\includegraphics[width = 0.175\linewidth]{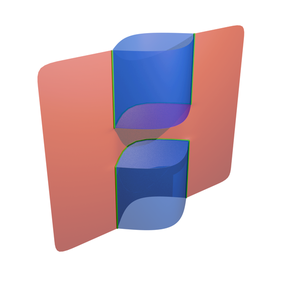}}{0.175}
            
        \end{annotate}
        \\
        
    \end{tabular}
    \caption{The 2-gon compression surgery with all components referenced in proofs labelled. Note that $N$ is a closed $3$-ball regular neighbourhood of the $2$-compression disk $D$ and that $N'$ is the component of $N \setminus Y$ that contains~$D$.}
    \label{fig:2goncompression}
\end{figure}

\begin{figure}
    \centering
    \begin{tabular}{c@{}c@{}c@{}c}
        $X$ & $Y$ & $X^+$ & $X^-$ \\

        \begin{annotate}{\includegraphics[width = 0.175\linewidth]{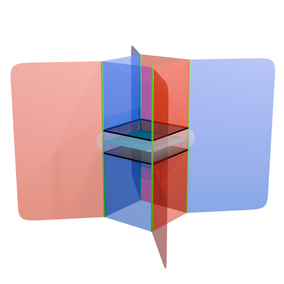}}{0.175} 
            \callout{6,6}{\tiny $N$}{2.5,1.5};
            \draw[thick,\arcol, rounded corners] (-2.5,-1) -- (-2.5,1) -- (0,1.75) -- (2.75,1.25) -- (2.75,-0.75) -- (0.5,-1.75) -- cycle;
        \end{annotate}
        &
        \begin{annotate}{\includegraphics[width = 0.175\linewidth]{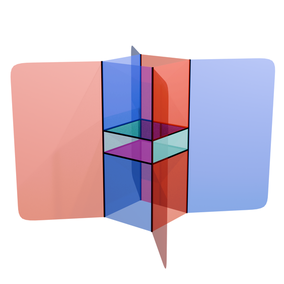}}{0.175} 
            \draw[thick,\arcol, rounded corners=2pt] (-2.1,-0.75) -- (-2.1,0.75) -- (0,1.4) -- (2.5,0.9) -- (2.5,-0.6) -- (0.5,-1.25) -- cycle;
            \callout{-6,-6}{\tiny $N'$}{-1,-1.5};
        \end{annotate}
        
        & 
        \begin{annotate}{\includegraphics[width = 0.175\linewidth]{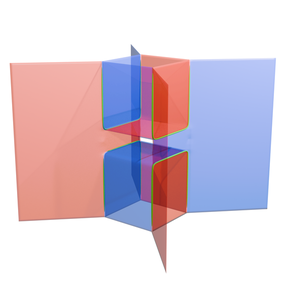}}{0.175}
            
        \end{annotate}
        & 
        \begin{annotate}{\includegraphics[width = 0.175\linewidth]{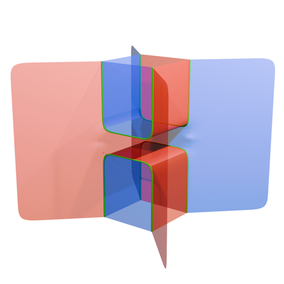}}{0.175}
            
        \end{annotate}
        \\
        \begin{annotate}{\includegraphics[width = 0.175\linewidth]{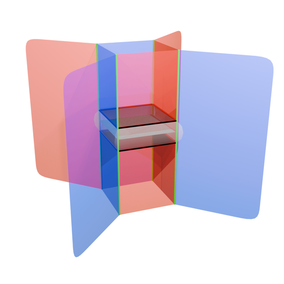}}{0.175} 
            
        \end{annotate}
        &
        \begin{annotate}{\includegraphics[width = 0.175\linewidth]{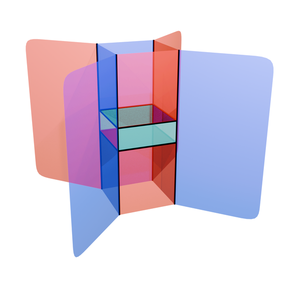}}{0.175}

        \end{annotate}
        & 
        \begin{annotate}{\includegraphics[width = 0.175\linewidth]{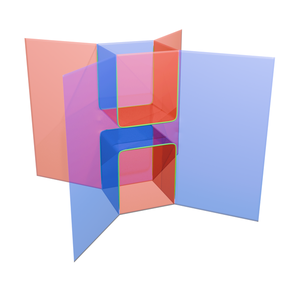}}{0.175}
            
        \end{annotate}
        & 
        \begin{annotate}{\includegraphics[width = 0.175\linewidth]{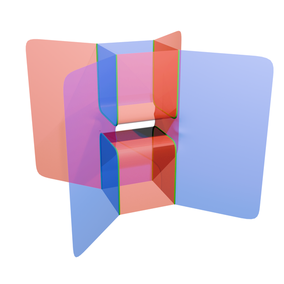}}{0.175}
            
        \end{annotate}
        \\
        \begin{annotate}{\includegraphics[width = 0.175\linewidth]{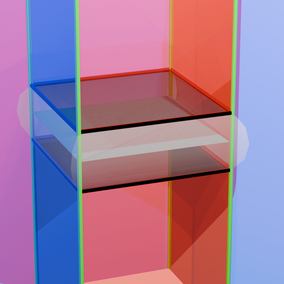}}{0.175} 
            \callout{6,3}{\tiny $D_1$}{-1,2};
            \callout{6,0}{\tiny $D$}{-1,0};
            \callout{6,-3}{\tiny $D_2$}{-1,-1.5};
        \end{annotate}
        &
        \begin{annotate}{\includegraphics[width = 0.175\linewidth]{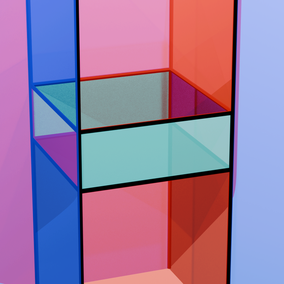}}{0.175}

            \callout{3,-6}{\tiny $A^+_1$ (front)}{0,0};
            \callout{-3,6}{\tiny $A^+_2$ (back)}{-2,2};
            \callout{6,5}{\tiny $A^-_1$ (side)}{2.5,2};
            \callout{-6,-5}{\tiny $A^-_2$ (side)}{-4,0};
        \end{annotate}
        & 
        \begin{annotate}{\includegraphics[width = 0.175\linewidth]{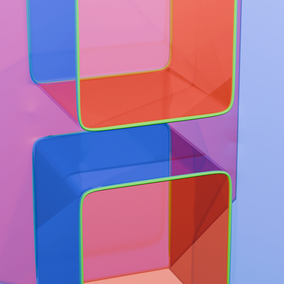}}{0.175}
            
        \end{annotate}
        & 
        \begin{annotate}{\includegraphics[width = 0.175\linewidth]{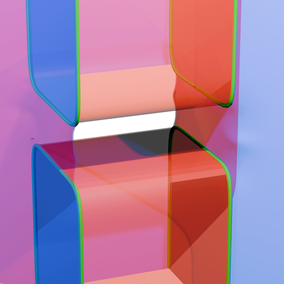}}{0.175}
            
        \end{annotate}
        \\
        
    \end{tabular}
    \caption{The 4-gon compression surgery with all components referenced in proofs labelled. Note that $N$ is a closed $3$-ball regular neighbourhood of the $4$-compression disk $D$ and that $N'$ is the component of $N\setminus Y$ that contains~$D$.}
    \label{fig:4goncompression}
\end{figure}

There are several ways to show that $[X'] = [X]$ for the above surgeries, but it is sufficient to fix $2$-cycles representing the fundamental classes of the two $\mathbb{Z}_3$-folded surfaces that differ only inside $N\subset M$ and note that $H_2(N;\mathbb{Z}_3) = 0$ since $N\simeq \mathbb{D}^3$. 

\subsection{Surface Elimination}\label{sec:surfaceelimination}

Suppose $X$ is a $\mathbb{Z}_3$-folded surface embedded in $M$ and $x \in H_2(X;\mathbb{Z}_2)$. Recall the quotient map $q:X_{cut}\rightarrow X$ from Definition~\ref{dfn:pfoldedsurface}. For each $C\in \pi_0(\Int(X))$ we abuse notation to consider $C\in H_2(X,\text{Sing}(X);\mathbb{Z}_2)$ to be the class given by the image of the fundamental class of the component $\tilde{C}:= \overline{q^{-1}(C)}$ of $X_{cut}$ under the isomorphism $H_2(X_{cut},\partial X_{cut};\mathbb{Z}_2)\rightarrow H_2(X,\text{Sing}(X);\mathbb{Z}_2)$. The map $H_2(X;\mathbb{Z}_2)\rightarrow H_2(X,\text{Sing}(X);\mathbb{Z}_2)$ is an injection by a dimension argument and so there is a unique set $I(x) \subset \pi_0(\Int(X))$ such that $x\mapsto \sum_{C\in I(x)} C$. Lemma~\ref{lem:surfacerealisation} describes the relationship between closed surfaces embedded in $X$ and classes in $H_2(X;\mathbb{Z}_2)$.

\begin{lem}\label{lem:surfacerealisation}
    Let $X$ be a $\mathbb{Z}_3$-folded surface embedded in a $3$-manifold $M$ and $x\in H_2(X;\mathbb{Z}_2)$. Let $S(x) := \bigcup_{C\in I(x)} \overline{C}$. The following are true:
    \begin{enumerate}
        \item The set $S(x)\subset X$ is homeomorphic to a closed surface.
        \item The class $x\in H_2(X;\mathbb{Z}_2)$ is the image of the fundamental class of $S(x)$ under the induced map $H_2(S(x);\mathbb{Z}_2)\rightarrow H_2(X;\mathbb{Z}_2)$. 
        \item If $x\in \ker(H_2(X;\mathbb{Z}_2)\rightarrow H_2(M;\mathbb{Z}_2))$ then $S(x)$ is orientable and every component of $I(x)$ is removable.
    \end{enumerate}
\end{lem}

\begin{proof}

    Let $I(x)\subset \pi_0(\Int(X))$ be the set of components such that 
    \[H_2(X;\mathbb{Z}_2)\rightarrow H_2(X,\text{Sing}(X);\mathbb{Z}_2)=\mathbb{Z}_2[\pi_0(\Int(X))]\] 
    maps $x \mapsto \sum_{C \in I(x)} C$. Let $S(x)$ be the sub-polyhedron given by $\overline{\bigcup I(x)}$. The number of sheets of $I(x)$ cobounding each arc of $S(x)\cap \text{Sing}(X)$ must be even since $\sum_{C \in I(x)} C$ is in the kernel of the connecting homomorphism $H_2(X,\text{Sing}(X);\mathbb{Z}_2)\rightarrow H_1(\text{Sing}(X);\mathbb{Z}_2)$. Since $X$ is a $\mathbb{Z}_3$-folded surface this even number must be $0$ or $2$ and this is sufficient to show that $S(x)$ is a closed $2$-manifold. 

    Now let $f^{S(x)}:S(x)\rightarrow M$, $i^{S(x)}:S(x)\rightarrow X$ and $f^X:X\rightarrow M$ be the three embeddings such that $f^{S(x)} = f^X\circ i^{S(x)}$. Let $[S(x)]\in H_2(S(x);\mathbb{Z}_2)$ be the fundamental class of $S(x)$. The composition of maps $H_2(S(x);\mathbb{Z}_2)\xrightarrow{i^{S(x)}_*} H_2(X;\mathbb{Z}_2) \hookrightarrow H_2(X,\text{Sing}(X);\mathbb{Z}_2)$ clearly sends $[S(x)]\mapsto \sum_{C \in I(x)} C$. The second map is an injection by a dimension argument and so we deduce that $x = i^{S(x)}_*([S])$. 
    
    Now if $x\in \ker(f^X_*:H_2(X;\mathbb{Z}_2)\rightarrow H_2(M;\mathbb{Z}_2))$ then it follows that $f^{S(x)}_*([S]) = f^X_*\circ i^{S(x)}_*([S]) = f^X_*(x) = 0$. It is well known and easily verified that homologically trivial closed surfaces in $3$-manifolds are orientable. It follows immediately that any $C\in I(x)$ must be removable since if any two boundary arcs of $\partial C_{cut}$ are identified with the same singular arc $\alpha$ in $X$ then we can find a simple closed curve embedded in $\overline{C}$ intersecting $\alpha$ transversally once whose regular neighbourhood in $\overline{C}$ is a M\"obius band.
    
\end{proof}

Our aim is to use the surgeries described in Section~\ref{sec:3foldsurgeries} to help us find a non-separating $\mathbb{Z}_3$-folded surface $X$ such that $\chi(X) = \chi_3([X])$ and the boundary pattern of $X$ is as simple as possible. The problem with the $2k$-compression surgeries as they are is that they do not necessarily preserve the property of non-separation. Recall that Lemma~\ref{lem:injectivitynonsep} allows us to measure the extent to which a $\mathbb{Z}_3$-folded surface is non-separating by the rank of 
\[\ker(H_2(X;\mathbb{Z}_2)\rightarrow H_2(M;\mathbb{Z}_2)).\] 
Lemma~\ref{lem:surfaceelimination} provides a method to obtain a $\mathbb{Z}_3$-folded surface $X'$ that is a subcomplex of $X$ such that the rank of 
\[\ker(H_2(X';\mathbb{Z}_2)\rightarrow H_2(M;\mathbb{Z}_2))\] 
is strictly less than the rank of 
\[\ker(H_2(X;\mathbb{Z}_2)\rightarrow H_2(M;\mathbb{Z}_2)).\] 
Induction on Lemma~\ref{lem:surfaceelimination} yields a non-separating $\mathbb{Z}_3$-folded surface realising any class in $H_2(M;\mathbb{Z}_3)$.

\begin{lem}\label{lem:surfaceelimination}
    Let $x \in \ker(H_2(X;\mathbb{Z}_2)\rightarrow H_2(M;\mathbb{Z}_2))$ where $X$ is a $\mathbb{Z}_3$-folded surface embedded in $M$. Let $\mathfrak{o}$ be a choice of orientation of $S(x)$ induced by an orientation of a component of its complement in $M$ and denote the opposite choice $-\mathfrak{o}$. Let $I(x,\mathfrak{o}) \subseteq I(x)$ be the components oriented consistently by the orientation of $\Int(X)$ and $\mathfrak{o}$. Let $X(x,\mathfrak{o})$ be the $\mathbb{Z}_3$-folded surface with underlying set $X \setminus \bigsqcup I(x,\mathfrak{o})$ and the orientation of $\Int(X(x,\mathfrak{o}))$ which is consistent with the orientation of $\Int(X)$ on the set $\Int(X(x,\mathfrak{o})) \setminus \overline{\bigsqcup I(x,-\mathfrak{o})}$ and opposite on the set $\bigsqcup I(x,-\mathfrak{o})$. We say that $X(x,\mathfrak{o})$ is the surface elimination of $x$ with orientation $\mathfrak{o}$. It has the following properties:

    \begin{enumerate}
        \item $[X(x,\mathfrak{o})] = [X]$
        \item If $I(x,\mathfrak{o}) \neq \emptyset$ then $\ker(H_2(X(x,\mathfrak{o});\mathbb{Z}_2)\rightarrow H_2(M;\mathbb{Z}_2))$ is a proper subgroup of $\ker(H_2(X;\mathbb{Z}_2)\rightarrow H_2(M;\mathbb{Z}_2))$ under the inclusion $i_{(X,X(x,\mathfrak{o}))}$
        \item $\chi(X(x,\mathfrak{o})) = \chi(X) - \sum_{C\in I(x,\mathfrak{o})}\chi(C)$
    \end{enumerate}
\end{lem}

\begin{proof}
    The proof of this lemma uses the same tricks as Lemmas~\ref{lem:weightedpfoldedsurfaces} and~\ref{lem:removenotdisks}. First view $X$ as a weighted $\mathbb{Z}_3$-folded surface by assigning weights of $1$ to all components of $\Int(X)$. Let $X_0(x,\mathfrak{o})$ be the result of deleting the components in $I(x,\mathfrak{o})$ and setting weights $w_D = 2$ for all $D\in I(x,-\mathfrak{o})$ then the same reasoning in Lemma~\ref{lem:weightedpfoldedsurfaces} implies that $[X_0(x,\mathfrak{o})] = [X]$ since they differ by the trivial $H_2(M;\mathbb{Z}_3)$ class represented by $S(x)$ with orientation $\mathfrak{o}$. Let $I_0(x)$ be the set of components in $\Int(X)$ cobounding a singular arc in $S(x)$ that are not in $I(x)$. Clearly since $S(x)$ is a closed surface we can deduce that every component of $S(x)\cap \text{Sing}(X)$ is cobounded by precisely one component of $I_0(x)$, one component of $I(x,\mathfrak{o})$ and one component of $I(x,-\mathfrak{o})$, all inducing a consistent orientation on their shared boundary.
    
    This means that if we take $X_0(x,\mathfrak{o})$, reverse the orientation of the components of $I(x,-\mathfrak{o})$ and set their weights equal to $1$ then the two components cobounding each arc of $S(x)\cap \text{Sing}(X)$ now induce opposite orientations along their shared boundary and have the same weight. We can therefore extend the inverted orientation of $\bigsqcup I(x,-\mathfrak{o})$ and the original orientation of $\bigsqcup I^0(x)$ to an orientation of $2$-manifold obtained by gluing them along $S(x)\cap\text{Sing}(X)$. We remove $S(x)\cap\text{Sing}(X)$ from the singular locus and call the resulting oriented folded surface $X(x,\mathfrak{o})$. To see that $X(x,\mathfrak{o})$ is a $\mathbb{Z}_3$-folded surface simply note that the neighbourhoods and orientations of points outside of a regular neighbourhood of $\bigsqcup I(x,-\mathfrak{o})$ are identical to in $X$ and inside this regular neighbourhood $X(x,\mathfrak{o})$ is an oriented $2$-manifold. 
    
    The image of $X(x,\mathfrak{o})$ and its orientation data in $M$ is identical to $(X_0(x,\mathfrak{o}))^u$ outside a regular neighbourhood of each component in $I(x,-\mathfrak{o})$. 
    
    Let $\Sigma$ be a closed $2$-manifold component of $I(x,-\mathfrak{o})$ with the orientation inherited from $X(x,\mathfrak{o})$ and let $[Sigma]\in H_2(M;\mathbb{Z}_3)$ be the class represented by $\Sigma$. The component of the weighted $\mathbb{Z}_3$-folded surface $X_0(x,\mathfrak{o})$ with underlying set $\Sigma$ has the opposite orientation and weight $2$ so it represents the class $-2[\Sigma] \in H_2(M;\mathbb{Z}_3)$. These two classes are equivalent because $1\equiv -2 \mod 3$. It follows that the change made to closed $2$-manifold components does not affect the class represented.
    
    Let $H$ be a regular neighbourhood of $\bigsqcup_{D\in I(x,-\mathfrak{o}), \partial D \neq \emptyset} \overline{D}$. Since each $\overline{D}$ is a compact $2$-manifold with non-empty boundary it has a regular neighbourhood homeomorphic to a handlebody. $H$ is a disjoint union of handlebodies and so $H_2(H;\mathbb{Z}_3) = 0$. It follows from this that the change made to the components of $X_0(x,\mathfrak{o})$ with non-empty boundary also does not affect the class represented.
    
    It follows from these observations that $[X(x,\mathfrak{o})] = [X_0(x,\mathfrak{o})] = [X]$.

    A dimension argument can be used to show that the inclusion $X(x,\mathfrak{o})\subset X$ induces an injection 
    \[H_2(X(x,\mathfrak{o});\mathbb{Z}_2)\rightarrow H_2(X;\mathbb{Z}_2).\] 
    Moreover it is clear that $x$ is not in the image of this homomorphism since $S(x)\not \subset X(x,\mathfrak{o})$. Finally, $X(x,\mathfrak{o})$ has the same image set as $\underline{X_0(x,\mathfrak{o})}$ and so
    \[\chi(X(x,\mathfrak{o})) = \chi(\underline{X_0(x,\mathfrak{o})}) = \chi(\Int(X_0(x,\mathfrak{o}))).\] 
    Recall that the only components of $\Int(X_0(x,\mathfrak{o}))$ with weight not equal to $1$ are those in $I(x,-\mathfrak{o})$ which have weight $2$ and so 
    \[\chi(\Int(X_0(x,\mathfrak{o}))) = \chi(X_0(x,\mathfrak{o})^u) - \sum_{D \in I(x,-\mathfrak{o})}\chi(D).\] 
    Counting the multiplicities of interior components we can obtain 
    \[\chi(X_0(x,\mathfrak{o})^u) = \chi(X)- \sum_{C\in I(x,\mathfrak{o})}\chi(C) + \sum_{D\in I(x,-\mathfrak{o})}\chi(D).\] 
    Combining these statements yields the desired Euler characteristic 
    \[\chi(X(x,\mathfrak{o})) = \chi(X) - \sum_{C\in I(x,\mathfrak{o})}\chi(C).\] 

    The fact that the inclusion 
    \[i_{(X,X(x,\mathfrak{o}))}:H_2(X(x,\mathfrak{o});\mathbb{Z}_2)\rightarrow H_2(X;\mathbb{Z}_2)\] 
    identifies $\ker i_{(M,X(x,\mathfrak{o}))}$ with a proper subgroup of $\ker i_{(M,X)}$ follows from noting the fact that $x$ cannot be in the image of $i_{(X,X(x,\mathfrak{o}))}$ since the class
    \[j_{(X,\Sing X)}(x) \in H_2(X,\Sing X;\mathbb{Z}_2)\] 
    is supported by the classes of components of $I(x,\mathfrak{o})$ that are not in $X(x,\mathfrak{o})$.
    
\end{proof}

\subsection{Separating surfaces in compressed $\mathbb{Z}_3$-folded surfaces}\label{sec:compressionsurfacebirth}

The surgeries in Section~\ref{sec:3foldsurgeries} do not preserve the non-separating property of $\mathbb{Z}_3$-folded surfaces in general. To rectify this using the surface elimination procedure described in Section~\ref{sec:surfaceelimination} we need to understand exactly how a separating $\mathbb{Z}_3$-folded surface can arise from a $2k$-compression surgery of a non-separating $\mathbb{Z}_3$-folded surface $X$. To avoid unnecessary repetition we will use $X'$ to simultaneously denote the following.

\begin{itemize}
    \item The result of a $0$-compression surgery of $X$.
    \item Either of $X^+$ or $X^-$, positive and negative $2k$-compression surgeries of $X$ in the case $k\geq 1$.
\end{itemize} 

Let $Y,N,N',D,D_1$ and $D_2$ be the spaces relevant to the construction of $X'$ from $X$ as defined in Section~\ref{sec:3foldsurgeries}. The embeddings of $X,X'$ and $Y$ in $M$ factor in the following commutative diagram.

\begin{center}
    \begin{tikzcd}[sep = huge]
        X \ar[r,hook,"i_{(Y,X)}"] \ar[rd,hook,swap,"i_{(M,X)}" {sloped}] &
        Y  \ar[d,hook, swap,"i_{(M,Y)}"]&
        X' \ar[l,hook,"i_{(Y,X')}" {sloped}] \ar[ld,hook, swap,"i_{(M,X')}" {sloped}]\\
        & M & \\
    \end{tikzcd}
\end{center}

Recall that due to Lemma~\ref{lem:injectivitynonsep} the induced map $i_{(M,X)}:H_2(X;\mathbb{Z}_2)\rightarrow H_2(M;\mathbb{Z}_2)$ is injective since the embedding of $X$ is non-separating in $M$. Similarly, $X'$ is non-separating if and only if $i_{(M,X')}:H_2(X';\mathbb{Z}_2)\rightarrow H_2(M;\mathbb{Z}_2)$ is injective. The map $i_{(Y,X)}:H_2(X;\mathbb{Z}_2)\rightarrow H_2(Y;\mathbb{Z}_2)$ is injective because the long exact sequence of the pair guarantees that $\ker(i_{(Y,X)})$ is equal to the image of the connecting homomorphism $d_{(Y,X)}:H_3(Y,X;\mathbb{Z}_2)\rightarrow H_2(X;\mathbb{Z}_2)$ and $H_3(Y,X;\mathbb{Z}_2)=0$ by a dimension argument. The map $i_{(Y,X')}:H_2(X';\mathbb{Z}_2)\rightarrow H_2(Y;\mathbb{Z}_2)$ is injective for the same reason. Hence the kernels $\ker(i_{(M,X)}:H_2(X;\mathbb{Z}_2)\rightarrow H_2(M;\mathbb{Z}_2))$ and $\ker(i_{(M,X')}:H_2(X';\mathbb{Z}_2)\rightarrow H_2(M;\mathbb{Z}_2))$ embed in $\ker(i_{(M,Y)}:H_2(Y;\mathbb{Z}_2)\rightarrow H_2(M;\mathbb{Z}_2))$ via the respective inclusions $i_{(Y,X)}$ and $i_{(Y,X')}$.

Another dimension argument using the long exact sequence of the pair $(Y,\Sing Y)$ implies that $j_{(Y,\Sing Y)}:H_2(Y;\mathbb{Z}_2)\rightarrow H_2(Y,\Sing Y;\mathbb{Z}_2)$ is also injective since $H_2(\Sing Y;\mathbb{Z}_2) = 0$. 

Let $\pi_{\{D_1,D_2\}}$ be the homomorphism given by the composition $\ker(i_{(M,Y)}:H_2(Y;\mathbb{Z}_2)\rightarrow H_2(M;\mathbb{Z}_2))\xrightarrow{j_{(Y,\Sing Y)}} H_2(Y,\Sing Y;\mathbb{Z}_2) \rightarrow \mathbb{Z}_2 D_1 \oplus \mathbb{Z}_2 D_2$ where the last map is the projection defined by identifying $H_2(Y,\Sing Y;\mathbb{Z}_2) = \mathbb{Z}_2[\pi_0(\Int Y)]$ as in the proof of Lemma~\ref{lem:surfacerealisation}. The only components of $\Int(Y)$ not in $X \subset Y$ are $D_1$ and $D_2$ and so any class in $\ker \pi_{\{D_1,D_2\}}$ lies in the intersection of the image of $i_{(Y,X)}$ and the kernel of $i_{(M,Y)}$. Moreover since $i_{(Y,X)}:H_2(X;\mathbb{Z}_2)\rightarrow H_2(Y;\mathbb{Z}_2)$ is injective it follows that $\Ima i_{(Y,X)}\cap \ker i_{(M,Y)}\simeq  \ker{i_{(M,X)}} = 0$. Therefore $\pi_{\{D_1,D_2\}}$ is injective. This already implies that $\ker i_{(M,X')}$ contains at most three non-trivial classes due to the sequence of inclusions in Equation~\eqref{eq:sequenceofinclusions}.

\begin{equation}\label{eq:sequenceofinclusions}
    \ker i_{(M,X')}\lhook\joinrel\xrightarrow{i_{(Y,X')}} \ker i_{(M,Y)} \lhook\joinrel\xrightarrow{\pi_{\{D_1,D_2\}}} \mathbb{Z}_2 \oplus \mathbb{Z}_2
\end{equation}

Let $N'$ be the component of $N \setminus X$ containing $D_1,D_2$ and $D$. Let $[\partial(N')] \in H_2(Y;\mathbb{Z}_2)$ be the image of the fundamental class of the boundary of $N'$ under the homomorphism induced by the inclusion $\partial(N')\subset Y$. Clearly $[\partial(N')]$ is annihilated by $i_{(M,Y)}$ since it is represented by the boundary of a region in $M$. $\partial(N')$ contains the components $D_1$ and $D_2$ of $\Int Y$ and so $\pi_{\{D_1,D_2\}}([\partial(N')]) = D_1 + D_2$. Suppose $X'$ is separating in $M$. Let $x$ be a non-trivial element of $\ker(i_{(M,X')}:H_2(X';\mathbb{Z}_2)\rightarrow H_2(M;\mathbb{Z}_2))$. It follows that $i_{(Y,X')}(x)$ is a non-trivial element of $\ker(i_{(M,Y)})$. Hence $\pi_{\{D_1,D_2\}} \circ i_{(Y,X')}(x) \neq 0$. If $\pi_{\{D_1,D_2\}} \circ i_{(Y,X')}(x) = D_1 + D_2$ then $i_{(Y,X')}(x) + [\partial N'] \in \ker \pi_{\{D_1,D_2\}}$ and hence $i_{(Y,X')}(x) + [\partial N'] = 0 \Rightarrow i_{(Y,X')}(x) = [\partial N']$. This is a contradiction since in Section~\ref{sec:3foldsurgeries} $X'$ was constructed from $Y$ by removing a subset of $\partial(N')$ and hence $[\partial N'] \not\in i_{(Y,X')}(H_2(X';\mathbb{Z}_2))$. It follows that $\ker (i_{(M,X')}:H_2(X';\mathbb{Z}_2)\rightarrow H_2(M;\mathbb{Z}_2))$ contains at most one non-trivial class $x$ and that the orientable surface $S(x)$ realising $x$ contains exactly one of $D_1$ or $D_2$. Refer to Section~\ref{sec:surfaceelimination} for more details on realising $\mathbb{Z}_2$ $2$-homology classes of $\mathbb{Z}_3$-folded surfaces as surfaces.

For $0$-compressions this completely describes how the unique separating surface $S(x)$ can intersect a neighbourhood of the compression disk $D$ - that is, either the cup or the cap in Figure~\ref{fig:interiorcompression} is contained in $S(x)$. 

Now let $k>0$ and $X^{\pm}$ be the two possible $2k$-gon compressions about $D$ and assume that both $X^+$ and $X^-$ separate $M$. Let $x_{\pm}$ be the unique non-trivial classes in $\ker(i_{(M,X^{\pm})}:H_2(X^{\pm};\mathbb{Z}_2)\rightarrow H_2(M;\mathbb{Z}_2))$. Define $\pi_{\{A_i^{\pm}\}_{i=1}^k}:\ker(i_{(M,Y)})\rightarrow \mathbb{Z}_2[\{A_i^{\pm}\}_{i=1}^k]$ similarly to $\pi_{\{D_1,D_2\}}$ above. By definition $X^{\pm} = Y \setminus \bigsqcup_{i=1}^k A_i^{\pm}$ and so it follows that $\pi_{\{A_i^{\pm}\}_{i=1}^k}\circ i_{(Y,X^{\pm})}(x_{\pm}) = 0$. Now without loss of generality we suppose that $\pi_{\{D_1,D_2\}} \circ i_{(Y,X^+)}(x_+) = D_1$ (this is valid since the labelling of $D_1,D_2$ depends only on an arbitrary choice of orientation of the compressing disk $D$). We now separate the two cases $\pi_{\{D_1,D_2\}}\circ i_{(Y,X^-)}(x_-) = D_1$ and $\pi_{\{D_1,D_2\}}\circ i_{(Y,X^-)}(x_-) = D_2$.

\begin{enumerate}
    \item Suppose that 
    \[\pi_{\{D_1,D_2\}}\circ i_{(Y,X^-)}(x_-) = D_1.\] 
    It follows that 
    \[\pi_{\{D_1,D_2\}}(i_{(Y,X^+)}(x_+) + i_{(Y,X^-)}(x_-)) = 0\] 
    and hence 
    \[i_{(Y,X^+)}(x_+) + i_{(Y,X^-)}(x_-) = 0 \Rightarrow i_{(Y,X^+)}(x_+) = i_{(Y,X^-)}(x_-) =: x'.\] 
    Hence $\pi_{\{A_i^{+}\}_{i=1}^k}(x') = 0$ and $\pi_{\{A_i^{-}\}_{i=1}^k}(x') =0$.
    \item Suppose $\pi_{\{D_1,D_2\}}\circ i_{(Y,X^-)}(x_-) = D_2$. Then 
    \[\pi_{\{D_1,D_2\}}(i_{(Y,X^+)}(x_+) + i_{(Y,X^-)}(x_-)) = D_1 + D_2\] 
    and 
    \[\pi_{\{D_1,D_2\}}([\partial N']+i_{(Y,X^+)}(x_+) + i_{(Y,X^-)}(x_-)) = 0.\] 
    Hence 
    \[i_{(Y,X^+)}(x_+) + i_{(Y,X^-)}(x_-) = [\partial N'].\] 
    Since $A_i^+$ and $A_i^-$ are contained in $\partial N'$ for all $1\leq i \leq k$ it follows that 
    \[\pi_{\{A_i^{\pm}\}_{i=1}^k}([\partial N']) = \sum_{i=1}^k A_i^{\pm}.\] 
    Therefore we deduce that:
    \begin{align*}
        \pi_{\{A_i^{\pm}\}_{i=1}^k} \circ i_{(Y,X^{\mp})}(x_{\mp})  &= \pi_{\{A_i^{\pm}\}_{i=1}^k}(i_{(Y,X^{+})}(x_{+}) + i_{(Y,X^{-})}(x_{-}))\\
        &= \pi_{\{A_i^{\pm}\}_{i=1}^k}([\partial N']) \\
        &= \sum_{i=1}^k A_i^{\pm}.
    \end{align*}
\end{enumerate} 

The map $H_2(X^\pm,\Sing X^\pm;\mathbb{Z}_2)\rightarrow H_2(Y,\Sing Y;\mathbb{Z}_2)$ induced by the inclusion $(X^{\pm},\Sing X^{\pm})\subset (Y,\Sing Y)$ is injective by another dimension argument. It follows from the naturality of the sequences of relative homology that if the coefficient of a component $C\in \pi_0(\Int Y)$ is non-zero in $j_{(Y,\Sing Y)}\circ i_{(Y,X^\pm)}(x_\pm)$ then there is a unique component $C'$ of $\Int X^\pm$ such that $C\subset C'$ and that the coefficient of $C'$ in $j_{(X^\pm,\Sing X^\pm)}(x^\pm)$ is non-zero. Let $g^{\pm}: \pi_0(\Int Y) \setminus \{A_i^{\pm}\}_{i=1}^k \rightarrow \pi_0(\Int X^\pm)$ be the function defined such that $g^\pm(C) := C'$ where $C'$ is the unique component containing $C$. Let $I(x_{\pm})$ be the set of components of $\Int X^\pm$ contained in the separating surface $S(x_{\pm})\subset X^\pm$ realising $x_\pm$ as in Section~\ref{sec:surfaceelimination}. The above cases are equivalent to the following.

\textit{Case 1}
\begin{itemize}
    \item $g^+(D_1) \in I(x_+)$
    \item $g^+(D_2) \notin I(x_+)$ and $g^+(A_i^-) \notin I(x_+)$ for all $1\leq i \leq k$.
    \item $g^-(D_1) \in I(x_-)$
    \item $g^-(D_2) \notin I(x_-)$ and $g^-(A_i^+)\notin I(x_-)$ for all $1\leq i \leq k$.
\end{itemize}

\textit{Case 2}
\begin{itemize}
    \item $g^+(D_1) \in I(x_+)$ and $g^+(A_i^-) \in I(x_+)$ for all $1\leq i \leq k$.
    \item $g^+(D_2) \notin I(x_+)$
    \item $g^-(D_2) \in I(x_-)$ and $g^-(A_i^+)\in I(x_-)$ for all $1\leq i \leq k$.
    \item $g^-(D_1) \notin I(x_-)$
\end{itemize}

The remaining elements of $I(x_+)$ and $I(x_-)$ which intersect a neighbourhood of the surgery are determined by the requirement that an even number of the sheets meeting in each singular arc of $X^\pm$ must belong to $S(x_{\pm})$. Figure~\ref{fig:2compressionsurfaces} shows the two cases for how $S(x_+) \subset X^+$ can intersect with a neighbourhood of the $2$-compression surgery (we omit $S(x_-) \subset X^-$ since $X^+$ and $X^-$ are ambient isotopic and so the picture looks very similar). Figure~\ref{fig:4compressionsurfaces} shows the two cases for how $S(x_+) \subset X^+$ and $S(x_-) \subset X^-$ can intersect a neighbourhood of a $4$-compression surgery.

\subsection{Reduced $2k$-compression}\label{sec:reducedsurgeries}

We now restrict our focus to closed orientable $3$-manifolds in which every embedding of $S^2$ separates the ambient space. These are precisely the closed orientable $3$-manifolds whose prime decomposition does not admit an $S^1 \times S^2$ summand (see Hatcher's notes~\cite{Hatcher2001NotesOB} for a short proof of this fact). 

Suppose $X$ is a non-separating $\mathbb{Z}_3$-folded surface embedded in such a $3$-manifold $M$ such that $\chi(X) = \chi_3([X])$. We now combine the $2k$-compression surgeries of Section~\ref{sec:3foldsurgeries} and the surface elimination surgery of Section~\ref{sec:surfaceelimination} to define reduced versions of the $2k$-compression surgeries for $k=0,1,2$ such that the resulting $\mathbb{Z}_3$-folded surface is also non-separating and of maximal Euler characteristic. Definition~\ref{dfn:reduced2kcompression} describes the \textit{reduced $2k$-compression} surgeries indirectly. 

\begin{dfn}[Reduced $2k$-compression surgery]\label{dfn:reduced2kcompression}
    Let $X$ be a non-separating $\mathbb{Z}_3$-folded surface embedded in a $3$-manifold $M$ with the property that
    \[\chi(X) = \chi_3([X]).\] 
    Let $X'$ be the result of a $2k$-compression of $X$ about some $2k$-compression disk $D$. Let $X''$ be a $\mathbb{Z}_3$-folded surface with the following properties:

    \begin{itemize}
        \item $X''\subset X'$.
        \item $[X''] = [X]$.
        \item $\chi(X'')=\chi_3([X])$.
        \item $X''$ is non-separating in $M$.
    \end{itemize}

    Then we say that $X''$ is the result of a \textit{reduced $2k$-compression surgery} of $X$ about $D$.

\end{dfn}

Lemma~\ref{lem:reduced2kcompressionsexist} proves that these reduced $2k$-compression surgeries exist for $k\in \{0,1,2\}$ when $M$ satisfies certain requirements.

\begin{lem}\label{lem:reduced2kcompressionsexist}
    Let $X$ be a non-separating $\mathbb{Z}_3$-folded surface embedded in a closed orientable $3$-manifold $M$ with the property that every $2$-sphere embedded in $M$ separates the ambient space and let $D$ be a $2k$-compression disk of $X$ in $M$ for some $k\in \{0,1,2\}$. It follows that there exists a reduced $2k$-compression surgery of $X$ about $D$. 
\end{lem}

\begin{proof}

    The proofs of this statement for $k=0,1,2$ are the constructions in Section~\ref{sec:casek0}, Section~\ref{sec:casek1} and Section~\ref{sec:casek2} respectively. Note that when $k\in\{1,2\}$ the reduced $2k$-compressions arising from a particular $2k$-compression surgery may not be unique.

\end{proof}

We find it convenient to introduce some notation before progressing with the body of the proof. Let $Y$ be the $2$-polyhedron $X\cup D_1 \cup D_2$ defined as in Section~\ref{sec:3foldsurgeries}. Note that for $k=0$ we have the inclusion $\Int(Y)\setminus A \subset \Int X'$. Let $g':\pi_0(\Int Y)\setminus \{A\} \rightarrow \pi_0(\Int X')$ be the function on sets of connected components induced by this inclusion such that $g'(C) = C' \Leftrightarrow C\subset C'$. For $k\geq 1$ we will use the function $g^{\pm}:\pi_0(\Int Y)\setminus \{A_i^{\pm}\}_{i=1}^{k} \rightarrow \pi_0(\Int X^{\pm})$ defined similarly as in Section~\ref{sec:compressionsurfacebirth}.

\subsubsection{Reduced $0$-compressions}\label{sec:casek0}

Let $X'$ be the result of a $0$-compression of $X$. Since $\chi(X') = \chi(X)+2$ (see Section~\ref{sec:3foldsurgeries}) the maximality of the Euler characteristic of $X$ among non-separating $\mathbb{Z}_3$-folded surfaces forces $X'$ to admit a separating surface. In Section~\ref{sec:compressionsurfacebirth} it is shown that this separating surface is unique and contains exactly one of the two disks $D_1,D_2$ parallel to the $0$-compression disk. Let $x\in \ker(H_2(X';\mathbb{Z}_2)\rightarrow H_2(M;\mathbb{Z}_2))$ be the unique non-trivial class. Up to a choice of orientation of $D$ or equivalently choice of labelling of $D_1,D_2$ we can assume $g'(D_1)\in I(x)\subset \pi_0(\Int X')$. Let $\mathfrak{o}$ be the orientation of $S(x)\subset X'$ that is consistent with the orientation of $g'(D_1)$ as a component of $\Int(X')$. Let $X'' := X'(x,\mathfrak{o})$. We say that $X''$ is the result of a reduced $0$-compression of $X$ about $D$. Recall from Lemma~\ref{lem:surfaceelimination} that $\chi(X'') = \chi(X')- \sum_{C\in I(x,\mathfrak{o})}\chi(C)$. Every element of $I(x)$ other than $g'(D_1)$ does not intersect the neighbourhood of $D$ in which the $0$-compression takes place and so has a neighbourhood in $M$ identical to a neighbourhood of a component of $\Int X$. In particular no component of $\Int X$ can be a $2$-sphere since $X$ is non-separating and every $2$-sphere embedded in $M$ separates it. Similarly no component of $\Int X$ can be a removable disk due to Lemma~\ref{lem:noremovablediskschi}. Every component of $I(x)$ is removable due to Lemma~\ref{lem:surfacerealisation} and must have non-positive Euler characteristic with the exception possibly of $g'(D_1)$ and so $\chi(X'') \geq \chi(X') - \chi(g'(D_1))$. Equation~\eqref{eq:eulercharineqk0} follows from the observation that $g'(D_1)$ is a connected compact $2$-manifold and so $\chi(g'(D_1))\leq 2$.

\begin{equation}\label{eq:eulercharineqk0}
    \chi(X'') \geq \chi(X')-\chi(g'(D_1)) \geq \chi(X') -2 = \chi(X)
\end{equation}

Lemma~\ref{lem:surfaceelimination} implies that $\ker(H_2(X'';\mathbb{Z}_2)\rightarrow H_2(M;\mathbb{Z}_2))$ is a proper subgroup of $\ker(H_2(X';\mathbb{Z}_2)\rightarrow H_2(M;\mathbb{Z}_2)) = \mathbb{Z}_2$ and so $\ker(H_2(X'';\mathbb{Z}_2)\rightarrow H_2(M;\mathbb{Z}_2)) = 0$. Hence, Lemma~\ref{lem:injectivitynonsep} implies that $X''$ is non-separating in $M$. The maximality of $\chi(X)$ among non-separating $\mathbb{Z}_3$-folded surfaces realising $[X]$ implies that $\chi(X'') \leq \chi(X)$ and combining this with Equation~\eqref{eq:eulercharineqk0} implies that $\chi(X'') = \chi(X) = \chi_3([X])$ as required. 

In Definition~\ref{dfn:0incompressibility} we generalise the notion of \textit{incompressibility} as a property of surfaces embedded in $3$-manifolds to a property of $\mathbb{Z}_3$-folded surfaces.

\begin{dfn}[$0$-incompressibility]\label{dfn:0incompressibility}
    Let $X$ be a $\mathbb{Z}_3$-folded surface embedded a $3$-manifold $M$. We call $X$ \textit{$0$-incompressible} if any $0$-compressing disk $D$ of $X$ has the property that $\partial(D)$ bounds a disk in $\Int X$.
\end{dfn}

Lemma~\ref{lem:0incompressibility} follows from the fact that $g'(D_1)$ must be a $2$-sphere when $X$ is a non-separating $\mathbb{Z}_3$-folded surface of maximal Euler characteristic, given any $0$-compressing disk $D$. 

\begin{lem}[$0$-incompressibility]\label{lem:0incompressibility}
    Let $X$ be a non-separating $\mathbb{Z}_3$-folded surface embedded in a closed orientable $3$-manifold $M$ with the property that every $2$-sphere embedded in $M$ separates the ambient space. If $\chi(X) = \chi_3([X])$ then $X$ must be $0$-incompressible.
\end{lem}

\begin{proof}    
    
    We adopt the notation already established in this section, allowing $D$ to be any $0$-compressing disk of $X$. The equality $\chi(X'') = \chi(X)$ forces the Equality $\chi(X')-\chi(g'(D_1)) = \chi(X') -2 \Rightarrow \chi(g'(D_1)) = 2$ and so $g'(D_1)\simeq S^2$. Let $A$ be the open annulus in $X$ bounded by $\partial D_1 \sqcup \partial D_2$. $g'(D_1)$ is formed by gluing the disk $D_1$ along $\partial D_1$ to the component $B$ of $\Int(X)\setminus A$ that cobounds $\partial D_1$. It follows from elementary Euler characteristic arguments that $B$ is homeomorphic to a closed disk. 
\end{proof}

\subsubsection{Reduced $2$-compressions}\label{sec:casek1}

Let $X^+$ be the positive $2$-compression of $X$ about a $2$-compression disk $D$. Since $\chi(X^+) = \chi(X)+1$ (see Section~\ref{sec:3foldsurgeries}) we can deduce from the maximality of $\chi(X)$ among all non-separating $\mathbb{Z}_3$-folded surfaces that $X^+$ admits a surface which separates $M$. In Section~\ref{sec:compressionsurfacebirth} we showed that this surface is unique and contains exactly one of $D_1$ and $D_2$. Let $x_+\in \ker(H_2(X^+;\mathbb{Z}_2)\rightarrow H_2(M;\mathbb{Z}_2))$ be the unique non-trivial class such that $S(x_+)$ is this separating surface. Without loss of generality, either by inverting the orientation of $D$ or by relabelling $D_1$ and $D_2$, we can assume that $g^+(D_1)\in I(x_+)\subset \pi_0(\Int X^+)$. Let $\mathfrak{o}$ be the orientation of $S(x_+)$ such that $g^+(D_1) \in I(x_+,\mathfrak{o})$.

Suppose that $g^+(A_1^-) \in I(x_+)$. If this is the case then $g^+(A_1^-)$ is the only element of $I(x_+,-\mathfrak{o})$ that intersects the neighbourhood $N$ of $D$ in which the surgery takes place. This can be seen in Figure~\ref{fig:2compressionsurfaces} which depicts the two subsets $I(x_+,-\mathfrak{o})$ and $I(x_+,\mathfrak{o})$ in different colours to highlight this fact. In particular the same reasoning from Section~\ref{sec:casek0} implies that every element of $I(x_+,-\mathfrak{o})$ other than $g^+(A_1^-)$ must have non-positive Euler characteristic and so $\chi(X^+(x_+,-\mathfrak{o})) \geq \chi(X^+)-\chi(g^+(A_1^-)) = \chi(X)+1-\chi(g^+(A_1^-))$. Similarly, $X^+(x_+,-\mathfrak{o})$ must be non-separating due to Lemma~\ref{lem:surfaceelimination} and so the maximality of $\chi(X)$ among non-separating $\mathbb{Z}_3$-folded surfaces implies that $\chi(X^+(x_+,-\mathfrak{o})) = \chi(X)$ and so $g^+(A_1^-)$ must be a removable disk. $g^+(A_1^-)$ cannot be connected in $\Int X^+$ to any of the other components depicted in Figure~\ref{fig:2goncompression} which intersect the neighbourhood of the surgery since these other components are either not in $S(x_+)$ or they belong to $I(x_+,\mathfrak{o})$. This implies that by removing the closure of $A_1^-$ from $g^+(A_1^-)$ we recover either a component or a union of two components of $\Int X$. An elementary Euler characteristic computation yields that $\chi(g^+(A_1^-) \setminus \overline{A_1^-}) = \chi(g^+(A_1^-)) + 1 = 2$ and hence $g^+(A_1^-) \setminus \overline{A_1^-}$ is a union of two disks (they are not closed surfaces and hence cannot be $S^2$). Since these disks are connected in $X$ they also must be removable due to Lemma~\ref{lem:removabledisksorclosed}. However Lemma~\ref{lem:noremovablediskschi} contradicts this since it implies that no component of $\Int X$ is a removable disk. It follows that our initial assumption is wrong and hence $g^+(A_1^-) \not\in I(x_+)$.

We therefore know that the only element of $I(x_+,\mathfrak{o})$ intersecting the neighbourhood of the surgery $N$ is $g^+(D_1)$ - this is the first case in Figure~\ref{fig:2compressionsurfaces}. Since all other elements of $I(x_+,\mathfrak{o})$ have neighbourhoods identical to components of $\Int X$ we can deduce that none of them are removable disks or $2$-spheres. Moreover $g^+(D_1)$ has non-empty boundary and so $\chi(g^+(D_1))\leq 1$. Applying Lemma~\ref{lem:surfaceelimination} produces a non-separating $\mathbb{Z}_3$-folded surface $X^+(x_+,\mathfrak{o})$ with $\chi(X^+(x_+,\mathfrak{o})) = \chi(X')- \sum_{C\in I(x,\mathfrak{o})}\chi(C)$. The observation that $\chi(C)\leq 0$ for all $C\in I(x,\mathfrak{o})\setminus \{g^+(D_1)\}$ yields Equation~\ref{eq:eulercharineqk1}.

\begin{equation}\label{eq:eulercharineqk1}
    \chi(X^+(x_+,\mathfrak{o})) \geq \chi(X^+) - \chi(g^+(D_1)) \geq \chi(X^+)-1 =\chi(X)
\end{equation}

The maximality of $\chi(X)$ among all non-separating $\mathbb{Z}_3$-folded surfaces in $M$ implies that $\chi(X^+(x_+,\mathfrak{o}))\leq \chi(X)$. Combining this with Equation~\eqref{eq:eulercharineqk1} yields the required equality $\chi(X^+(x_+,\mathfrak{o})) = \chi(X) = \chi_3([X])$.

Alternately if we let $\mathcal{E}^+$ be the unique element of $I(x_+,-\mathfrak{o})$ which intersects $N$ (labelled in Figure~\ref{fig:2compressionsurfaces}) then identical reasoning implies that $X^+(x_+,-\mathfrak{o})$ is non-separating, Equation~\eqref{eq:eulercharineqk1_2} holds and 
\[\chi(X^+(x_+,-\mathfrak{o})) = \chi(X).\]

\begin{equation}\label{eq:eulercharineqk1_2}
    \chi(X^+(x_+,-\mathfrak{o})) \geq \chi(X^+) - \chi(\mathcal{E}^+) \geq \chi(X^+)-1 =\chi(X)
\end{equation}

Hence both $X^+(x_+,\mathfrak{o})$ and $X^+(x_+,-\mathfrak{o})$ satisfy the requirements of Definition~\ref{dfn:reduced2kcompression} as reduced $2$-compressions about $D$. 

Definition~\ref{dfn:2incompressible} gives the analogue of incompressibility for $2$-compressions.

\begin{dfn}[$2$-incompressibility]\label{dfn:2incompressible}
    Let $X$ be a $\mathbb{Z}_3$-folded surface embedded in a $3$-manifold. If for any $2$-compressing disk $D$ of $X$ there is an interval $I \subset \text{Sing}(X)$ with endpoints $\partial I = D\cap \text{Sing}(X)$ such that the theta graph $I\cup \partial D$ is the boundary of two disks in $\Int(X) \setminus \partial D$, we say that $X$ is \textit{$2$-incompressible}.
\end{dfn}

Lemma~\ref{lem:2incompressibility} follows immediately from the preceding discussion.

\begin{lem}\label{lem:2incompressibility}
    Let $X$ be a non-separating $\mathbb{Z}_3$-folded surface embedded in a closed orientable $3$-manifold $M$ with the property that every $2$-sphere embedded in $M$ separates the ambient space. If $\chi(X) = \chi_3([X])$ then $X$ is $2$-incompressible.
\end{lem}

\begin{proof}
    We adopt the notation already established in this section, allowing $D$ to be any $0$-compressing disk of $X$.
    
    Combining the fact that $\chi(X^+(x_+,\mathfrak{o})) = \chi(X)$ with Equation~\eqref{eq:eulercharineqk1} implies that $\chi(g^+(D_1)) = 1$ and so $g^+(D_1)$ is homeomorphic to an open disk. $g^+(D_1)$ is the interior of the boundary connect sum of the closure of $D_1$ and the closure of $g^+(D_1) \setminus \overline{D_1}$ which is a component of $\Int(X) \setminus \partial D_1$. It follows that $\chi(g^+(D_1)) = \chi(D_1) + \chi(B)-1$ and hence $\chi(g^+(D_1) \setminus \overline{D_1})=1$ which implies that $g^+(D_1) \setminus \overline{D_1}$ is an open disk bounded by an interval of $\partial D_1$ and an interval of $\Sing(X)$.

    Combining the fact that $\chi(X^+(x_+,-\mathfrak{o})) = \chi(X)$ with Equation~\eqref{eq:eulercharineqk1_2} implies that $\chi(\mathcal{E}^+) = 1$ and so $\mathcal{E}^+$ is homeomorphic to an open disk. $\mathcal{E}^+$ is a component of $\Int(X) \setminus \partial D_1$ that is bounded by a complementary interval of $\partial D_1$ and the same interval of $\Sing(X)$ (see Figure~\ref{fig:2compressionsurfaces}), forming the required theta graph.

    We assert that this is sufficient to prove Lemma~\ref{lem:2incompressibility} since $D_1$ also runs over all $2$-compressing disks of $X$.

\end{proof}

\begin{figure}
    \centering
    \begin{tabular}{c@{}c}
        \fbox{Case 1: $g^+(A_1^-) \notin I(x_+)$} & \fbox{Case 2: $g^+(A_1^-)\in I(x_+)$}  \\

        \begin{annotate}{\includegraphics[width = 0.3\linewidth]{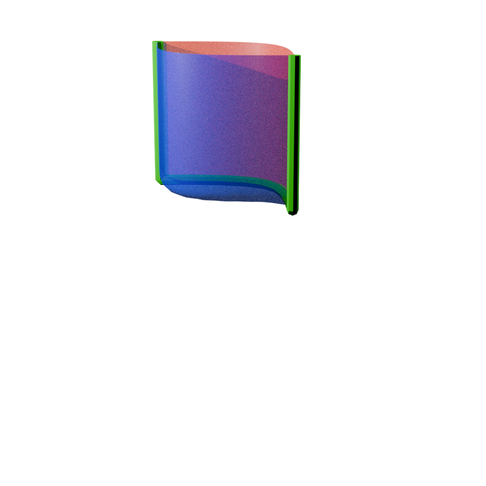}}{0.4}

            \callout{-3,-3}{$g^+(D_1)$}{-1,3}
        \end{annotate}
        &
        \begin{annotate}{\includegraphics[width = 0.3\linewidth]{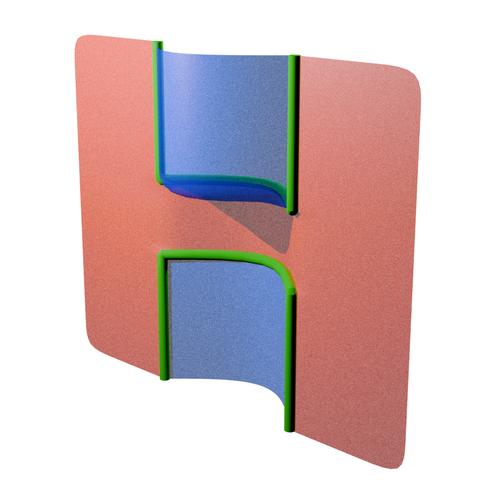}}{0.4} 

        \callout{-3,-3}{$g^+(A_1^-)$}{-1,0.5}
        \end{annotate}

        \\
        \begin{annotate}{\includegraphics[width = 0.3\linewidth]{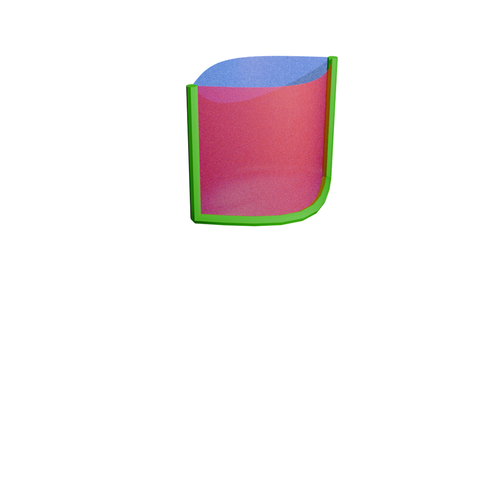}}{0.4} 

            \callout{4,-3}{$\mathcal{E}^+$}{0.5,3}
        \end{annotate}
        &
        \begin{annotate}{\includegraphics[width = 0.3\linewidth]{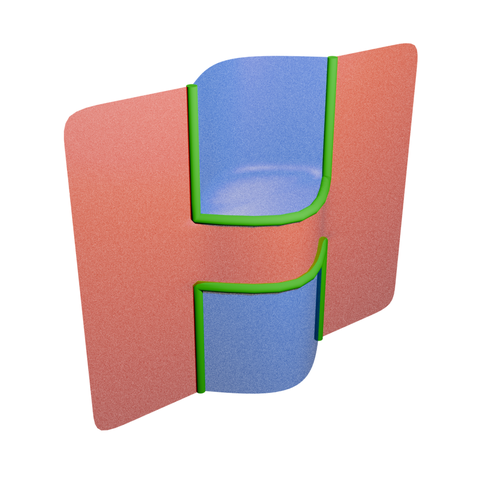}}{0.4}

        \end{annotate}
        \\
        
    \end{tabular}
    \caption{The two cases (up to the obvious symmetries) for how $S(x_+)\subset X^+$ can intersect a neighbourhood of the $2$-gon compression disk by which $X^+$ was created from the non-separating $\mathbb{Z}_3$-folded surface $X$. Note that the class $x_+\in H_2(X^+;\mathbb{Z}_2)$, which $S(x_+)$ realises, is the unique non-trivial element of $\ker i_{(M,X^+)}$. Components of $I(x_+,\mathfrak{-o})$ are \textcolor{red}{red}. Components of $I(x_+,\mathfrak{o})$ are \textcolor{blue}{blue}.}
    \label{fig:2compressionsurfaces}
\end{figure}

\subsubsection{Reduced $4$-compressions}\label{sec:casek2}

Let $X^+$ and $X^-$ be the two $4$-compressions of $X$ about a $4$-compression disk $D$. Because $\chi(X^{\pm}) = \chi(X)$ (see Section~\ref{sec:3foldsurgeries}), it is possible that $X^+$ or $X^-$ is non-separating. If $X^+$ is non-separating we can take $X'':= X^+$ as the result of a reduced $4$-compression of $X$ about $D$. Otherwise if $X^-$ is non-separating then we take $X'':= X^-$. 

We now assume that $X^+$ and $X^-$ both separate $M$. As described in detail in Section~\ref{sec:compressionsurfacebirth} there are unique non-trivial classes $x_+ \in \ker(H_2(X^{+};\mathbb{Z}_2)\rightarrow H_2(M;\mathbb{Z}_2))$ and $x_- \in \ker(H_2(X^{-};\mathbb{Z}_2)\rightarrow H_2(M;\mathbb{Z}_2))$. As in Section~\ref{sec:compressionsurfacebirth} we assume without loss of generality that $g^+(D_1) \in I(x_+)$. Recall the two cases $g^-(D_i)\in I(x_-)$ for some $i\in \{1,2\}$. Let $\mathfrak{o}_+$ be the orientation of $S(x_{+})$ that is consistent with the orientation of $g^{+}(D_1)$ in $\Int X^+$ and let $\mathfrak{o}_-$ be the orientation of $S(x_{-})$ that is consistent with the orientation of $g^-(D_i)$ in $\Int X^-$.

First assume that $g^-(D_2)\in I(x_-)$. It follows that the components $g^{\pm}(A^{\mp}_1)$ and $g^{\pm}(A^{\mp}_2)$ are both in $I(x_{\pm},-\mathfrak{o}_{\pm})$. Since we do not know a priori whether these components are distinct we will use one symbol $\mathcal{A}^{\mp} := g^{\pm}(A^{\mp}_1) \cup g^{\pm}(A^{\mp}_2)$ to denote their union.

$\mathcal{A}^{\mp}$ contains the only elements of $I(x_{\pm},-\mathfrak{o}_{\pm})$ that intersect the neighbourhood of the surgery. It follows that every other component of $I(x_{\pm},-\mathfrak{o}_{\pm})$ has a neighbourhood in $M$ identical to a component of $X$ and so must have non-positive Euler characteristic by similar reasoning as in Section~\ref{sec:casek1}. Hence it follows that $\chi(X^{\pm}(x_\pm,-\mathfrak{o}_{\pm})) \geq \chi(X^{\pm}) - \chi(\mathcal{A}^{\mp}) = \chi(X) - \chi(\mathcal{A}^{\mp})$. Lemma~\ref{lem:surfaceelimination} implies that the $\mathbb{Z}_3$-folded surface $X^{\pm}(x,-\mathfrak{o})$ is non-separating and so to avoid a contradiction with the maximality of $\chi(X)$ it follows that $\chi(\mathcal{A}^{\mp}) \geq 0$. Now note that since the only components of $I(x_{\pm},-\mathfrak{o}_{\pm})$ that intersect a neighbourhood of the surgery are $g^{\pm}(A^{\mp}_1)$ and $g^{\pm}(A^{\mp}_2)$ it follows that we can recover a union of components $C$ of $\Int X$ by removing the closures of $A^{\mp}_1$ and $A^{\mp}_2$ from $\mathcal{A}^{\mp}$. An elementary Euler characteristic computation yields that $\chi(C) = \chi(\mathcal{A}^{\mp})+2 \geq 2$. Since $C$ consists only of non-compact components it follows that two of the components of $C$ are disks that are connected in $X$. Hence it follows from Lemma~\ref{lem:removabledisksorclosed} that these two disks are both removable which is a contradiction of Lemma~\ref{lem:noremovablediskschi} which implies that $X$ has no removable disks. Hence our initial assumption that $g^-(D_2)\in I(x_-)$ was wrong. 

By default it follows that $g^-(D_1)\in I(x_-)$. Section~\ref{sec:compressionsurfacebirth} implies that the only element of $I(x_{\pm},\mathfrak{o}_{\pm})$ intersecting $N$ is $g^{\pm}(D_1)$. Let $\mathcal{E}^{\pm}$ denote the union of the components of $I(x_{\pm},-\mathfrak{o}_{\pm})$ that intersect $N$ (this will either be one or two components, which are labelled in Figure~\ref{fig:4compressionsurfaces}). From Figure~\ref{fig:4compressionsurfaces} we see that $\mathcal{E}^+ = g^-(D_1)\setminus \overline{D_1}$ and $\mathcal{E}^- = g^+(D_1)\setminus \overline{D_1}$. An elementary Euler characteristic argument yields the fact that $\chi(\mathcal{E}^{\pm}) = \chi(g^{\mp}(D_1))+1$. Every component of $I(x_{\pm})$ other than $g^{\pm}(D_1)$ and the components of $\mathcal{E}^{\pm}$ does not intersect the surgery neighbourhood $N$ and so must have non-positive Euler characteristic since components of $\Int(X)$ with identical neighbourhoods exist and no component of $\Int X$ is a removable disk or a $2$-sphere. 

Suppose $g^{+}(D_1)$ is anything other than a removable disk. It follows that $\chi(g^{+}(D_1))\leq 0$ and so $\chi(X^+(x_+,\mathfrak{o}_+)) \geq \chi(X^+) - \chi(g^{+}(D_1)) \geq \chi(X^+) = \chi(X)$. As in Section~\ref{sec:casek1}, Lemma~\ref{lem:surfaceelimination} implies that $H_2(X^+(x_+,\mathfrak{o}_+);\mathbb{Z}_2)\rightarrow H_2(M;\mathbb{Z}_2)$ is injective and Lemma~\ref{lem:injectivitynonsep} implies that $X^+(x_+,\mathfrak{o}_+)$ is a non-separating $\mathbb{Z}_3$-folded surface. Hence we can take $X^+(x_+,\mathfrak{o}_{+})$ as the result of a reduced $4$-compression in this case.

Otherwise, suppose $g^{+}(D_1)$ is a removable disk. It follows from the preceding discussion that we can conclude that $\chi(\mathcal{E}^-) = \chi(g^+(D_1))+1 = 2$ and so $\mathcal{E}^-$ must be a disjoint union of two removable disks (i.e. $\overline{g^+(D_1)}$ is the boundary connect sum of the two components of $\overline{\mathcal{E}^-}$). This forces $g^-(D_1)$ to have at least two boundary components and so $\chi(g^-(D_1))\leq 0$. We now note that $X^-(x_-,\mathfrak{o}_{\pm})$ is non-separating and that $\chi(X^-(x_-,\mathfrak{o}_{\pm}))\geq \chi(X)$ for the same reasons in the former case. In both cases since $\chi(X)$ is maximal among all non-separating $\mathbb{Z}_3$-folded surfaces realising $[X]$ we conclude that $\chi(X^-(x_-,\mathfrak{o}_{\pm})) = \chi(X)$. Thus $X^-(x_-,\mathfrak{o}_{\pm})$ meets the requirements of Definition~\ref{dfn:reduced2kcompression} and is the result of a reduced $4$-compression of $X$ about $D$ in this case.

\begin{figure}
    \centering
    \begin{tabular}{c@{}c@{}|@{}c@{}c}
        \multicolumn{2}{c}{\fbox{Case 1}}  & \multicolumn{2}{c}{\fbox{Case 2}} \\ 
        $X^+$ & $X^-$ & $X^+$ & $X^-$ \\

        \begin{annotate}{\includegraphics[width = 0.19\linewidth]{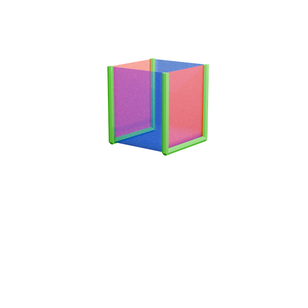}}{0.19} 

            \note{0,-5}{\parbox{1.9cm}{\textcolor{red}{everything red} is $\mathcal{E}^+$}}

        \end{annotate}
        &
        \begin{annotate}{\includegraphics[width = 0.19\linewidth]{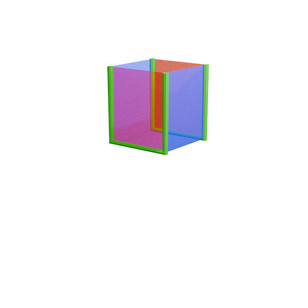}}{0.19} 
            \note{0,-5}{\parbox{1.9cm}{\textcolor{red}{everything red} is $\mathcal{E}^-$}}

        \end{annotate}
        
        & 
        \begin{annotate}{\includegraphics[width = 0.19\linewidth]{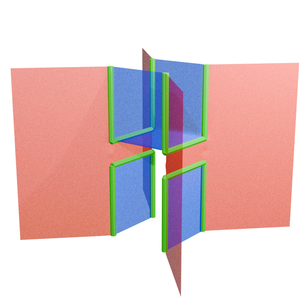}}{0.19}

            \note{0,6}{\textcolor{red}{everything red} is $\mathcal{A}^+$}
            \callout{-6,-6}{$g^+(D_1)$}{0,1}
        \end{annotate}
        & 
        \begin{annotate}{\includegraphics[width = 0.19\linewidth]{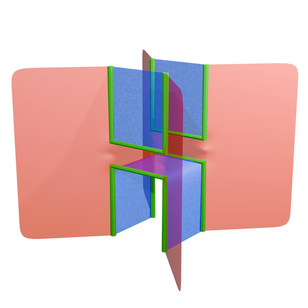}}{0.19}
            
        \end{annotate}
        \\
        \begin{annotate}{\includegraphics[width = 0.19\linewidth]{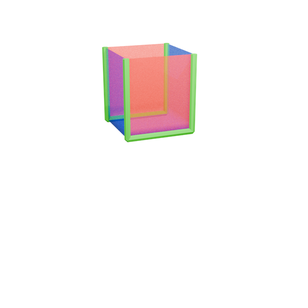}}{0.19} 
            \note{0,-5}{\parbox{2cm}{\textcolor{blue}{everything blue} is $g^+(D_1)$}}
            
        \end{annotate}
        &
        \begin{annotate}{\includegraphics[width = 0.19\linewidth]{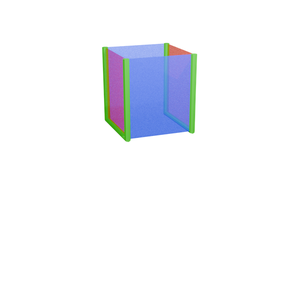}}{0.19}
            \note{0,-5}{\parbox{2cm}{\textcolor{blue}{everything blue} is $g^-(D_1)$}}
        
        \end{annotate}
        & 
        \begin{annotate}{\includegraphics[width = 0.19\linewidth]{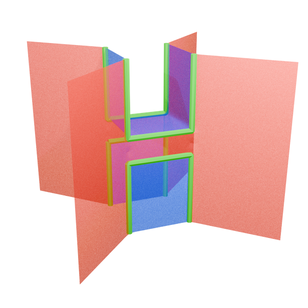}}{0.19}

        \end{annotate}
        & 
        \begin{annotate}{\includegraphics[width = 0.19\linewidth]{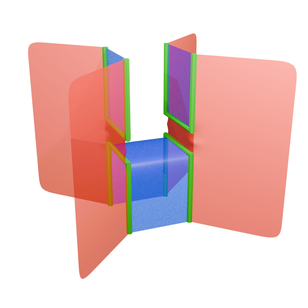}}{0.19}

            \callout{-1,6}{$g^-(D_2)$}{-0.5,0}
            
        \end{annotate}
        \\
        
    \end{tabular}
    \caption{The two cases (up to the obvious symmetries) for how $S(x_+)\subset X^+$ and $S(x_-)\subset X^-$ can intersect a neighbourhood of the $4$-gon compression disk that yields $X^\pm$ from the non-separating $\mathbb{Z}_3$-folded surface $X$. Note that the classes $x_{\pm}\in H_2(X^{\pm};\mathbb{Z}_2)$, which $S(x_{\pm})$ realise, are the unique non-trivial elements of $\ker i_{(M,X^{\pm})}$. Components of $I(x_{\pm},\mathfrak{-o})$ are \textcolor{red}{red}. Components of $I(x_{\pm},\mathfrak{o})$ are \textcolor{blue}{blue}. Recall that $\mathcal{A}^{\pm}$ is the union of $g^{\pm}(A_1^{\mp})$ and $g^{\pm}(A_2^\mp)$ which are respectively the components of $\Int X^+$ that contain $A_1^{\mp}$ and $A_2^{\mp}$.}
    \label{fig:4compressionsurfaces}
\end{figure}

\section{Proof of Theorem~\ref{thm:l3n1norm}}\label{sec:l3n1normproof}

\subsection{Decompositions of lens spaces}\label{sec:decompositions}

In this section everything stated will apply to any lens space $L = L(k,q)$ where $\gcd(k,q)=1$ and $k>q>0$. In later sections we will specialise to $k=3n$ for $n\in\mathbb{Z}^{>0}$ and finally $q=1$ for the proof of Theorem~\ref{thm:l3n1norm}.

Let $S_1$ and $S_2$ be a pair of solid tori. For $i\in \{1,2\}$, let $\mu_i$ be a meridional homology class of $H_1(\partial S_i;\mathbb{Z})$ generating the kernel of $H_1(\partial S_i;\mathbb{Z})\rightarrow H_1(S_i;\mathbb{Z})$ (i.e. realised by a simple closed curve bounding a disk in $S_i$). We use the term \textit{meridional disk} to refer to a disk $D_m$ properly embedded in $S_i$ such that $(\partial S_i,\partial D_m)$ is collared in $(S_i,D_m)$ and $\partial D_m$ is a simple closed curve which for some choice of orientation realises the homology class of $\mu_i$.

Let $\lambda_i$ be a choice of longitudinal class such that its image under the map $H_1(\partial S_i;\mathbb{Z})\rightarrow H_1(S_i;\mathbb{Z})$ generates $H_1(S_i;\mathbb{Z})$. There exists a homeomorphism $\phi:\partial S_1 \rightarrow \partial S_2$ such that $\phi_*(\mu_1) = k \lambda_2 + q\mu_2$ for coprime $k,q$. The quotient space obtained by identifying the boundaries of the solid tori with the gluing map $\phi$ is the lens space $L(k,q)$. We denote the splitting torus $T:= S_1\cap S_2$. Define $\bar{q}$ to be the unique positive integer such that $k>\bar{q}$ and $q \bar{q} \equiv 1 \mod k$. Fix $\lambda_1$ such that $\phi^{-1}_*(\mu_2) = k \lambda_1 + \bar{q} \mu_1$. We call this Heegaard decomposition the \textit{canonical Heegaard decomposition} of $L(k,q)$.

Another Heegaard splitting $L(k,q) = S_1' \cup S_2'$ with splitting torus $T' = S_1'\cap S_2'$ is equivalent to the splitting $S_1\cup S_2$ if there exists a homeomorphism $f:L(k,q)\rightarrow L(k,q)$ that maps one Heegaard decomposition to the other such that $f(S_1) = S_1', f(S_2) = S_2', f(T) = T'$.

Now suppose $X$ is a $\mathbb{Z}_p$-folded surface embedded in $L(k,q)$. Lemma~\ref{lem:changingcorecircle} will enable us to find a genus $1$ Heegaard splitting of $L(k,q)$ that is well adapted to $X$:

\begin{lem}\label{lem:changingcorecircle}

    Let $L$ be a lens space with a splitting torus $T$ and let $C$ be a core circle of one of the solid tori of $L\setminus T$. Let $h_t:L\rightarrow L$ be an ambient isotopy of $L$. Let $T_1$ be the torus given by the boundary of a regular neighbourhood (the star) of the circle $h_1(C)$. Then there exists a homeomorphism $g:L\rightarrow L$ such that $g(T) = T_1$. In other words $T_1$ gives an equivalent Heegaard splitting.
\end{lem}

\begin{proof}
    Let $T_0$ be the torus boundary of a regular neighbourhood of $C$. $T_0$ and $T$ cobound a thickened torus and this can be used to construct an isotopy carrying $T$ to $T_0$. This can be extended to an ambient isotopy $\psi_t: L\rightarrow L$ by the PL isotopy extension theorem (see~\cite{hudson1966}) such that $\psi_1(T) = T_0$. The tori $h_1(T_0)$ and $T_1$ are the boundaries of two regular neighbourhoods of the same circle $h_1(C)$ given by different simplicial structures of $L$. The uniqueness of regular neighbourhoods implies that there exists a homeomorphism $f:L\rightarrow L$ such that $f|_{h_1(C)} = id$ and $f(h_1(T_0)) = T_1$. In particular if we set $g = f\circ h_1 \circ \psi_1$ we have a homeomorphism of $L$ such that $g(T) = T_1$ as required. 

\end{proof}

\subsection{Mocktransversality and meridional intersection patterns}\label{sec:mocktransversality}

We now quote the definition of mock transversality given by McCrory in \cite[p.286]{mcrory1975} which adapts Buonchristiano, Rourke and Sanderson in~\cite{rourkesanderson1976}. 

\begin{dfn}\label{dfn:mocktransverse}
    Suppose we have two sub-polyhedra $P,Q$ in a compact $n$-manifold $M$. We say that $P$ is a \textit{proper subpolyhedron} of $M$ if $(\partial M, P\cap \partial M)$ is collared in $(M,P)$.
    
    Given a PL cell complex $\mathcal{C}$ of $M$ we say that $P$ is \textit{transverse} to $\mathcal{C}$ if for each cell $\alpha$ of $\mathcal{C}$ the intersection $(\partial \alpha, P\cap \partial \alpha)$ collars in $(\alpha, P\cap \alpha)$. 
    
    We say that $P$ is \textit{mocktransverse} to $Q$ if $Q$ is a subcomplex of $\mathcal{C}$ and $P$ is transverse to $\mathcal{C}$.

\end{dfn}

Note that mocktransversality is a symmetric property. Theorem~\ref{thm:relmocktransversality} recalls the statement given in~\cite[Theorem~4.7]{mcrory1975} of the Relative Mocktransversality Theorem due to~\cite{rourkesanderson1976}.

\begin{thm}[Relative Mocktransversality Theorem~\cite{rourkesanderson1976}]\label{thm:relmocktransversality}
    Let $P$ and $Q$ be proper subpolyhedra of a compact $n$-manifold $M$, with $P\cap \partial M$ mocktransverse to $Q\cap \partial M$ in the $n-1$-manifold $\partial M$. There exists an arbitrarily small ambient isotopy $h_t$ of $M$ such that $h_t|_{\partial M} = id$ for all $t$ and $h_1(P)$ is mocktransverse to $Q$.
\end{thm}

We can now use Lemma~\ref{lem:changingcorecircle} and Theorem~\ref{thm:relmocktransversality} to prove Lemma~\ref{lem:meridionaldiskintersection}.

\begin{lem}\label{lem:meridionaldiskintersection}
    For any $\mathbb{Z}_p$-folded surface $X$ embedded in $L(k,q)$ there exists a genus $1$ Heegaard splitting $L(k,q) = S_1\cup S_2$ such that $X\cap S_1$ consists of a finite number $d$ of meridional disks and $(S_1\cap S_2, X\cap S_1\cap S_2)$ is bicollared in $(L(k,q),X)$.
    Moreover if we fix a meridional homology class $\mu_1$ of $S_1$ and let $d_{\pm}$ be the number of disks of $X \cap S_1$ whose boundary (with the induced orientation) realises the class $\pm \mu_1 \in H_1(\partial S_1;\mathbb{Z})$ then the homology class of $[X]$ uniquely determines $d_+ - d_- \mod p$. We will call a genus $1$ Heegaard splitting of this kind \textbf{adapted} to $X$.
\end{lem}

\begin{proof}
    Let $C_1$ be a core circle of the solid torus $S_1$ in the genus $1$ Heegaard splitting $L(k,q)=S_1\cup S_2$. Theorem~\ref{thm:relmocktransversality} implies that there exists an ambient isotopy $h_t$ of $L(k,q)$ such that $h_1(C_1)$ is mocktransverse to $X$. Use Lemma~\ref{lem:changingcorecircle} to replace $S_1$ with a closed regular neighbourhood of $h_1(C_1)$ and $S_2$ with the closure of its complement. It follows from a dimension argument that $h_1(C_1)$ intersects $X$ only at points in $\text{Int}(X)$ and that these intersections are locally transverse (i.e. locally modelled by $\mathbb{R}\times \{0\}\times \{0\}$ and $\{0\}\times \mathbb{R}\times \mathbb{R}$). It follows from considering the intersection of $S_1$ with these local neighbourhoods of intersection points that $X$ intersects the new $S_1$ in a disjoint union of meridional disks and that there is at least a local bicollaring of $(\partial S_1,X\cap \partial S_1)$ in $L(k,q)$. This local bicollaring extends to a global bicollaring since in particular $\partial S_1$ is the boundary of a regular neighbourhood and so must have a trivial normal bundle in $L(k,q)$.
    
    Now orient these disks with the orientation inherited from $\text{Int}(X)$ and fix a meridional homology class $\mu_1 \in H_1(\partial(S_1);\mathbb{Z})$. Let $[\mu_1]$ be the image of $\mu_1$ under the change of coefficients $H_1(\partial S_1;\mathbb{Z})\rightarrow H_1(\partial S_1;\mathbb{Z}/p\mathbb{Z})$. The Mayer-Vietoris sequence on the homology of $L(k,q),S_1,S_2$ with coefficients in $\mathbb{Z}/p\mathbb{Z}$ yields a well defined map $\partial_*:H_2(L(k,q);\mathbb{Z}/p\mathbb{Z})\rightarrow H_1(\partial S_1;\mathbb{Z}/p\mathbb{Z})$ which sends $[X]\mapsto (d_+-d_-)[\mu_1]$. The subgroup generated by $[\mu_1]$ is $\mathbb{Z}/p\mathbb{Z}$ so this determines the $\mod p$ equivalence class of $d_+-d_-$. Moreover since $H_2(S_1;\mathbb{Z}/p\mathbb{Z})\oplus H_2(S_2;\mathbb{Z}/p\mathbb{Z})\cong 0$ the exactness of the Mayer-Vietoris sequence yields that $\ker \partial_*=0$ and so this class is unique.
    
\end{proof}

\begin{lem}\label{lem:transversemeridionaldisk}
    Suppose that $X$ is a $\mathbb{Z}_p$-folded surface embedded in $L(k,q)$ with $S_1\cup S_2$ a genus $1$ Heegaard splitting adapted to $X$ as per Lemma~\ref{lem:meridionaldiskintersection} and $T:= S_1\cap S_2$. There exists an oriented meridional disk $D$ of $S_2$ with the following properties.

    \begin{itemize}
        \item $\partial D$ is mocktransverse to $X\cap T$ in $T$.
        
        \item $\partial D \cap X$ is a set of $kd$ points with a cyclic order induced by the orientation of $\partial D$. 

        \item Let $\sigma:\partial D \cap X\rightarrow \partial D \cap X$ be the permutation which maps each point to its immediate successor in the cyclic order. For all $b\in \partial D \cap X$ let $M$ be the component of $X\cap S_1$ that contains it. The point $\sigma^d(b)$ lies on the boundary of the same meridional disk $M$ and the sign of the intersection of $\partial M$ with $\partial D$ at both points is the same.
    
        \item $(D,D \cap X)$ bicollars in $(S_2,S_2\cap X)$.
    \end{itemize}
\end{lem}

\begin{proof}

    Let $d$ be the number of meridional disks in $S_1\cap X$. Let $T = S_1\cap S_2$ be the splitting torus. Since $\partial S_1$ and $\partial S_2$ are identified with $T$ in $L(k,q)$ we can interpret $\mu_1,\mu_2,\lambda_1,\lambda_2$ as elements of $H_1(T;\mathbb{Z})$. We first describe $(T,X\cap T, \partial D)$. Since mocktransversality and the sign of intersections are preserved by ambient isotopies it is sufficient to describe curves on $T$ that are ambient isotopic to $X\cap T$ and $\partial D$. To this end consider the triangulation $\mathcal{T}$ of $T$ with one vertex, two triangles and three edges $e_{\lambda_2},e_{\mu_2},e_{\lambda_2+\mu_2}$ which realise the homology classes $\lambda_2,\mu_2, \lambda_2+\mu_2$ respectively when viewed as $1$-cycles when oriented as in Figure~\ref{fig:Ttriangulation}. 

\begin{figure}
    \centering
    \includegraphics[width=0.7\linewidth]{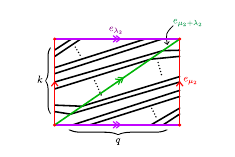}
    \caption{The distinguished triangulation of $\partial S_2$ and the unique normal meridional curve realising $\mu_1$.}
    \label{fig:Ttriangulation}
\end{figure}

    Standard arguments from the theory of normal curves imply that for each simple closed curve $Y$ embedded in $T$ that does not bound a disk there exists an ambient isotopy of $T$ carrying $Y$ to a curve $Y'$ that is normal with respect to the triangulation $\mathcal{T}$ (see Theorem 3.2.2 of~\cite{matveev:2003} for details of the normalisation procedure). This very simple triangulation $\mathcal{T}$ also has the property that such a normal curve is unique up to normal isotopy. If $Y$ is an oriented simple closed curve then the solution to the matching equations of $\mathcal{T}$ that describes $Y'$ is determined by the Poincar\'e dual $PD([Y])\in H^1(T;\mathbb{Z})$ of the homology class $[Y]\in H_1(T;\mathbb{Z})$ realised by $Y$. To be precise the number of points of intersection of $Y'$ with the edge $e_{\alpha}$ is given by $|PD([Y])(\alpha)|$ for $\alpha \in \{\mu_2,\lambda_2,\mu_2+\lambda_2\}$ where we identify $H^1(T;\mathbb{Z}) = \text{Hom}(H_1(T;\mathbb{Z}),\mathbb{Z})$ via the universal coefficient theorem and $|\cdot|:\mathbb{Z}\rightarrow \mathbb{Z}^{\geq0}$ is the usual absolute value function. $PD(a\mu_2 + b\lambda_2)(c\mu_2+d\lambda_2) = \pm(ad-bc)$ with the $\pm$ depending on a particular choice of Poincar\'e duality map $PD$ or equivalently a fundamental class of $T$.
    
    Now suppose we have performed an ambient isotopy to $T$ and put $X\cap T$ into a position such that it is normal with respect to $\mathcal{T}$. Each component $\partial M$ of $X\cap T$ is a simple closed curve that inherits an orientation from a corresponding component $M$ of $X\cap S_1$ which cobounds it and so realises the homology class $\pm \mu_1 = \pm(k\lambda_2 + q\mu_2) \in H_1(T;\mathbb{Z})$. The above discussion shows that $\partial M$ intersects $e_{\mu_2}$ in $k$ points, $e_{\lambda_2}$ in $q$ points and $e_{\mu_2+\lambda_2}$ in $k-q$ points (since we assume $k>q$). This normal curve is shown in Figure~\ref{fig:Ttriangulation}. It follows from the intersection form or by an inductive argument following these normal arcs through the triangulation that every intersection of $\partial M$ with the oriented edge $e_{\mu_2}$ must be of the same sign. The required cyclic ordering comes from the fact that there are $d$ such components of $X\cap S_1$ and so the normal curve $X\cap T$ consists of $d$ parallel copies of this normal curve, intersecting $e_{\mu_2}$ in $kd$ points. 

    We now set $\partial D:= e_{\mu_2}$. It follows from Definition~\ref{dfn:mocktransverse} and the properties of normal curves that $\partial D$ is mocktransverse to $X\cap T$.
    
    $\partial D$ is a meridional curve of $S_2$ and so it cobounds a properly embedded disk $D$ in $S_2$. We now invoke the Relative Mocktransversality Theorem quoted in Theorem~\ref{thm:relmocktransversality} to obtain an ambient isotopy of $S_2$ fixed on $\partial S_2 = T$ that allows us to replace $D$ with a disk that is mocktransverse to $X\cap S_2$ cobounding $\partial D$ as required.

    By definition~\ref{dfn:mocktransverse} there exists a PL cell complex $\mathcal{B}$ of $S_2$ such that $D$ is a subcomplex and to which $X\cap S_2$ is transverse. By dissolving the $3$-cells of this cell complex along the $2$-cells that are not contained in the subcomplex $\partial S_2 \cup D$ we obtain a PL cell complex $\mathcal{C}$ with a single $3$-cell, in which $D$ is a $2$-cell and every other $2$-cell is contained in the boundary torus. The local conditions of transversality are preserved under this cell complex amalgamation so $X$ is transverse to $\mathcal{C}$. It follows from definitions that $(D,D \cap X)$ collars in $(S_2,S_2\cap X)$ on both sides as required.

\end{proof}

We now define the $1$-dimensional analogue of a $\mathbb{Z}_p$-folded surface which will describe the topology of $D \cap X$.

\begin{dfn}\label{dfn:pfolded1mfld}
    A $\mathbb{Z}_p$-folded $1$-manifold with boundary is a $1$-polyhedron $\Gamma$ and two disjoint $0$-dimensional sub-polyhedra $\text{Sing}(\Gamma)$ and $\partial \Gamma$ with a choice of orientation of the $1$-manifold $\text{Int}(\Gamma) := \Gamma \setminus (\text{Sing}(\Gamma)\sqcup \partial(\Gamma))$ such that the following properties are satisfied:

    \begin{itemize}
        \item $\text{Sing}(\Gamma)$ is a set of $p$-valent vertices which we call the \textit{singular vertices}
        \item $\partial \Gamma$ is a set of univalent vertices which we call the \textit{boundary vertices}
        \item We call $\text{Int}(\Gamma)$ the interior of $\Gamma$. Every non-compact component of $\text{Int}(\Gamma)$ is bounded by two distinct vertices in the set $\text{Sing}(\Gamma) \sqcup \partial \Gamma$ and its orientation induces an orientation on each of these vertices. The orientation of $\text{Int}(\Gamma)$ must be such that $\forall v \in \text{Sing}(\Gamma) \sqcup \partial \Gamma$ every component of $\text{Int}(\Gamma)$ cobounding $v$ induces the same orientation on it. Hence we can say that the orientation of $\text{Int}(\Gamma)$ induces an orientation of $\text{Sing}(\Gamma) \sqcup \partial \Gamma$.
    \end{itemize}

    We say that $\Gamma$ is properly embedded in a $2$-manifold $S$ with boundary $\partial S$ if $\Gamma \cap \partial S = \partial \Gamma$.
    
\end{dfn}

\begin{dfn}\label{dfn:pfoldedwbdry}
    $\mathbb{Z}_p$-folded $1$-manifolds are the boundaries of \textit{$\mathbb{Z}_p$-folded surfaces with boundary} which can be defined by broadening the definition of a $\mathbb{Z}_p$-folded surface $X$ to include two new neighbourhood types: $[0,\infty) \times \mathbb{R}$ and $[0,\infty) \times K_{p,1}$. The boundary $\partial X$ is the complement of the points of $X$ that have local neighbourhoods given in the original Definition~\ref{dfn:pfoldedsurface}. If we let $v$ be the $p$-valent vertex of $K_{p,1}$ then we define $\Sing(X)$ as the set of points locally modelled by $(0,v)$ in either $[0,\infty) \times K_{p,1}$ or $\mathbb{R} \times K_{p,1}$ such that $\partial \Sing(X) = Sing(X) \cap \partial X$. We define $\Int(X):= X\setminus \Sing(X)$ as before. A $\mathbb{Z}_p$-folded surface with boundary $X$ is \textit{properly embedded} in a $3$-manifold with boundary $M,\partial M$ when $X\cap \partial M = \partial X$.
\end{dfn}

A large portion of the proof of Theorem~\ref{thm:l3n1norm} is an induction on objects of the kind described in Definition~\ref{dfn:meridionalintersectionpattern}.

\begin{dfn}[Meridional Intersection Patterns]\label{dfn:meridionalintersectionpattern}
    Suppose $X$ is a $\mathbb{Z}_p$-folded surface embedded in the lens space $L(k,q)$ with a genus $1$-Heegaard decomposition $S_1\cup S_2$ adapted to $X$ in the sense of Lemma~\ref{lem:meridionaldiskintersection} and $D$ is a meridional disk of $S_2$ of the kind described in Lemma~\ref{lem:transversemeridionaldisk}. We suggestively adopt the notation that $\Gamma_{X}:= X\cap D$. By combining the facts that $(D,D\cap X)$ bicollars in $(S_2,X)$, $(\partial S_2,X\cap \partial S_2)$ collars in $(S_2,X)$ and the definition of a $\mathbb{Z}_p$-folded surface (see Definition~\ref{dfn:pfoldedsurface}) it follows that $\Gamma_{X}$ must satisfy all topological properties of a $\mathbb{Z}_p$-folded $1$-manifold with no orientation data.

    To obtain an orientation of $\text{Int}(\Gamma_{X})$ we start by fixing an orientation of $L(k,q)$ and of $D$ arbitrarily. Let $B_2$ be the compact $3$-ball obtained by cutting $S_2$ along $D$. The orientation of $S_2\subset L(k,q)$ induces an orientation on $B_2$ and hence an orientation on $\partial B_2$. The subpolyhedron $X\cap S_2$ lifts to a subpolyhedron $f^{-1}(X\cap S_2)$ and the orientation of $\text{Int}(X)\cap S_2$ lifts to an orientation of $f^{-1}(\text{Int}(X)\cap S_2)$. Let $\tilde{D}$ be the disk in $\partial B_2$ such that the map $B_2\xrightarrow{f} S_2$ restricts to an orientation preserving homeomorphism $\tilde{D}\xrightarrow{f} D$. The map $f|_{\tilde{D}}$ identifies $\tilde{\Gamma}_X := \tilde{D}\cap f^{-1}(X\cap S_2)$ with $\Gamma_{X}$. Each component of $\text{Int}(\tilde{\Gamma}_X)$ is part of the boundary of a unique component of $f^{-1}(\text{Int}(X)\cap S_2)$ and so inherits an orientation. The fact that the orientation of each component of $\text{Int}(\tilde{\Gamma}_X)$ induces a consistent orientation on $\text{Sing}(\tilde{\Gamma}_{X})$ is an immediate consequence of the fact that $\text{Sing}(\tilde{\Gamma}_{X}) \subset f^{-1}(\text{Sing}(X))$ and the requirements on the orientation of components of $\text{Int}(X)$ that cobound each singular arc as per Definition~\ref{dfn:pfoldedsurface}. This gives $(\tilde{D},\tilde{\Gamma}_X)$ the structure of a $\mathbb{Z}_p$-folded $1$-manifold properly embedded in a $2$-manifold and hence also for the pair $(D,\Gamma_X)$ via the homeomorphism $(\tilde{D},\tilde{\Gamma}_X)\xrightarrow{f} (D,\Gamma_X)$. Note that a different choice of orientation of $L(k,q),D$ possibly reverses the orientation of $\Gamma_X$ but that this is inconsequential for our proof. We call such an embedding $\Gamma_X\hookrightarrow D$ for some choice of oriented $(S_1\cup S_2,D)$ a \textit{meridional intersection pattern} of $X$. 

\end{dfn}

\begin{lem}\label{lem:boundaryvertexpattern}
    Suppose $(D,\Gamma_X)$ is a meridional intersection pattern of $X\hookrightarrow L(k,q) = S_1\cup S_2$ and that $X\cap S_1$ consists of $d$ meridional disks. Let $\sigma$ be the permutation of $X\cap \partial D = \partial(\Gamma_X)$ described in Lemma~\ref{lem:transversemeridionaldisk}. The orientation of $b\in \partial(\Gamma_X)$ induced by the orientation of $\text{Int}(\Gamma_X)$ is the same as the orientation of $\sigma^d(b)$.
\end{lem}

\begin{proof}

    Let $B_2$ be the $3$-ball described in Definition~\ref{dfn:meridionalintersectionpattern} and let $A$ be the annular preimage of $T = S_1\cap S_2$ under the quotient map $B_2\xrightarrow{f} S_2$. The orientation of $\text{Int}(X)\cap S_2$ induces an orientation on the union of curves $X\cap \partial S_2$ that intersect $\partial D$ mocktransversally as per Lemma~\ref{lem:transversemeridionaldisk}. With this orientation, $f^{-1}(X\cap T)$ is a disjoint union of oriented intervals, each of which has one endpoint in $\partial(\tilde{\Gamma}_X)$. For each $b \in \partial(\Gamma_X)$ let $\tilde{b}\in \partial(\tilde{\Gamma}_X)$ be the point identified to it by $f$. Let $I$ be the interval in $f^{-1}(X\cap T)$ that cobounds $\tilde{b}$ and $J$ be the interval in $\text{Int}(\tilde{\Gamma}_X)$. Since $(T,X\cap T)$ collars in $(S_2,X\cap S_2)$ (see Lemma~\ref{lem:meridionaldiskintersection}) we can deduce that $I$ and $J$ are both intervals of the boundary of the same oriented component of $f^{-1}(\text{Int}(X)\cap S_2)$ and so must induce opposite orientations on the vertex $\tilde{b}$ they both cobound. The orientation that $I$ induces on $\tilde{b} = I\cap \partial \tilde{D}$ determines the sign of the intersection of $X\cap T$ with $\partial D$. The orientation induced on $X\cap T$ by the orientation of $\text{Int}(X)\cap S_2$ is the opposite of the orientation induced by $X\cap S_1$ and it remains true with either choice that the sign of the intersection of $X\cap T$ with $\partial D$ at any point $b\in \partial(\Gamma_X)$ determines the orientation of $b$ induced by the orientation of $\text{Int}(X)$. The proof is completed by applying Lemma~\ref{lem:transversemeridionaldisk}.

\end{proof}

Let $(D,\Gamma_X)$ be a meridional intersection pattern of a $\mathbb{Z}_p$-folded surface embedded in $L(k,q)$. The orientation induced by the orientation of $\text{Int}(\Gamma_X)$ partitions the vertices of $\text{Sing}(\Gamma_X)$ and $\partial \Gamma_X$ into \textit{black} and \textit{white} vertices. We therefore adopt the diagrammatic convention that these are depicted by solid and hollow dots. We refer to vertices belonging to these partitions as \textit{black} and \textit{white} vertices. We refer to components of $\Gamma_X$ other than simple closed curves as \textit{graph components} since they are homeomorphic to a multigraph with vertices given by the singular and boundary vertices. The constraints in Definition~\ref{dfn:pfolded1mfld} imply that such graph components are bipartite with respect to the vertex partition, $p$-valent at all singular vertices and univalent at all boundary vertices. The orientation data of $\Gamma_X$ beyond the partition it induces on the singular and boundary vertices is inconsequential to our arguments and so is omitted. Our diagrams are therefore invariant to the choice of orientation of $D$ and $L(k,q)$ up to inversion of the vertex colouring. Several examples of the meridional intersection pattern diagrams that follow these conventions for different lens spaces, prime $p$ and number of meridional disks $d$ are depicted in Figure~\ref{fig:meridionaldiskdiagrams}. In each example note that the orientation of the boundary vertices follows the cyclic pattern required by Lemma~\ref{lem:boundaryvertexpattern}.

\begin{figure}
    \centering
    \includegraphics[width=0.5\linewidth]{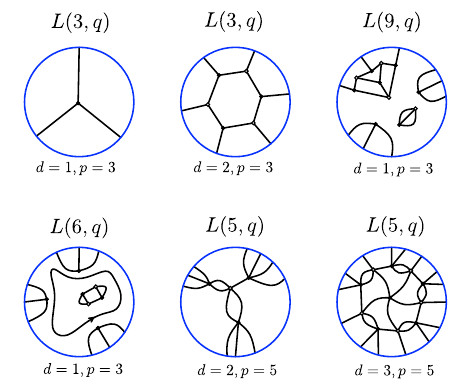}
    \caption{Examples of possible intersection patterns for different lens spaces $L(k,q)$, $\mathbb{Z}_p$-folded surfaces $X$ and $d$ meridional disks in the intersection $X\cap S_1$}
    \label{fig:meridionaldiskdiagrams}
\end{figure}

\subsection{Meridional intersection patterns of $\chi$-optimal $\mathbb{Z}_3$-folded surfaces in lens spaces}\label{sec:meridionalintpats}

In this section we restrict our focus to understanding non-separating, non-empty $\mathbb{Z}_3$-folded surfaces $X$ embedded in a lens space $L$ (other than $S^2\times S^1$) that have maximal Euler characteristic (i.e. $\chi(X) = \chi_3([X]))$. To package these adjectives we use Definition~\ref{dfn:chioptimal}, recalling that in the case $p\in \{2,3\}$ we take \textit{reduced} to mean \textit{non-separating} and $X^u = X$.

\begin{dfn}[$\chi$-optimal]\label{dfn:chioptimal}
    If $X$ is a non-empty reduced $\mathbb{Z}_p$-folded surface embedded in a $3$-manifold $M$ and $\chi(X^u) = \chi_p([X])$ then we say that $X$ is \textit{$\chi$-optimal}. 
\end{dfn}

Note that Lemma~\ref{lem:injectivitynonsep} implies that if $X$ is non-empty then $0 \neq [X] \in H_2(L;\mathbb{Z}_3)$ and so, with the exception of $S^2\times S^1$, such $\mathbb{Z}_3$-folded surfaces only exist in the family of lens spaces $L(3n,q)$ where $n\in \mathbb{Z}^{>0}$ and $q$ is coprime to $3n$ that have non-trivial $H_2(L(3n,q);\mathbb{Z}_3) = \mathbb{Z}_3$. The exclusion of $S^2\times S^1$ is necessary for the reduced $2k$-compression surgeries defined in Section~\ref{sec:reducedsurgeries} to exist. It is also not hard to show using Lemma~\ref{lem:components} that the essential $2$-sphere in $S^2\times S^1$ is a $\chi$-optimal $\mathbb{Z}_3$-folded surface and so $\chi_3(x)=2$ for either non-trivial class $x\in H_2(S^2\times S^1;\mathbb{Z}_3)$.

Lemma~\ref{lem:simplifiedl3nqintersectionpattern} finds a meridional intersection pattern of a $\chi$-optimal $\mathbb{Z}_3$-folded surface in $L(3n,1)$ that has been simplified as far as possible via Lemma~\ref{lem:simplifiedwithsurgieries} which performs a reduced $2k$-compression on the $\mathbb{Z}_3$-folded surface and Lemma~\ref{lem:intervalbetterintersection} which replaces the genus $1$ Heegaard decomposition of $L(3n,1)$.

\begin{lem}\label{lem:simplifiedl3nqintersectionpattern}
    There exists a $\chi$-optimal $\mathbb{Z}_3$-folded surface $X$ embedded in each lens space $L(3n,q)$ with $\gcd(3n,q)=1$ and a genus $1$ Heegaard splitting $S_1\cup S_2 = L(3n,q)$ adapted to $X$ such that $X\cap S_1$ consists of a single meridional disk and there exists a meridional intersection pattern $(D,\Gamma_X)$ with the following properties.

    \begin{itemize}
        \item No component of $D\setminus \Gamma_X$ is a disk bounded by a simple closed curve component of $\Gamma_X$ or a cycle of length $2$ or $4$ in a graph component of $\Gamma_X$.
        \item No component of $\Gamma_X$ is contained in the interior of $D$.
    \end{itemize}
\end{lem}

The proof of this Lemma combines the surgeries described in Section~\ref{sec:reducedsurgeries} with the Heegaard splitting replacement in Lemma~\ref{lem:changingcorecircle}. 

\begin{lem}\label{lem:compressioninterior}
    Let $X$ be a $\chi$-optimal $\mathbb{Z}_3$-folded surface in $L(3n,q) = S_1\cup S_2$ with meridional intersection pattern $(D,\Gamma_X)$. If any component of $D\setminus \Gamma_X$ cobounds a simple closed curve component or a cycle of two or four edges in a graph component of $\Gamma_X$ then there exists another $\chi$-optimal $\mathbb{Z}_3$-folded surface $X'$ embedded in $L(3n,q)$ with the following properties.

    \begin{itemize}
        \item $X'\cap S_1$ is a subset of the meridional disks of $X\cap S_1$.
        \item $(D,D\cap X')$ bicollars in $(S_2,X'\cap S_2)$ and so $\Gamma_{X'}:= D\cap X'$ is a meridional intersection pattern as per definition~\ref{dfn:meridionalintersectionpattern}. 
        \item $|\pi_0(\text{Int}(\Gamma_{X'}))| < |\pi_0(\text{Int}(\Gamma_{X}))|$ and $|\Sing(\Gamma_{X'})| \leq |\Sing(\Gamma_{X})|$.
    \end{itemize} 
\end{lem}

\begin{proof}

    Let $E$ be a component of $D\setminus \Gamma_X$ as described. If $E$ cobounds a simple closed curve in $\Gamma_X$ then $E$ is a $0$-compression disk. Similarly if $E$ cobounds a cycle of length $2$ or $4$ in $\Gamma_X$ then $E$ is a $2$ or $4$-compression disk respectively. Accordingly let $X'$ be the result of performing a reduced $2k$-compression of $X$ about $E$ where $k\in \{0,1,2\}$ as necessary (see Section~\ref{sec:reducedsurgeries} for details). The regular neighbourhood of $E$ used in the surgery can be taken to be a subcomplex of the bicollar $[-1,1]\times D$ such that the two parallel disks are $\{1\}\times E$ and $\{-1\}\times E$. The fact that $(D,D \cap X')$ still bicollars in $(S_2,X'\cap S_2)$ follows from taking a slightly smaller collar $[-1+\varepsilon,1 - \varepsilon]\times D$ and noting that in this collar $\text{Int}(X')$ is a subset of the components of $\text{Int}(X)$. The effect of $2k$-compressions on the meridional intersection pattern is summarised in Figure~\ref{fig:meridionalsurgeries}. When embedded orientable surfaces are removed as part of the reduction of these surgeries to recover a non-separating $\mathbb{Z}_3$-folded surface it is possible that the result will be an even simpler meridional intersection pattern with fewer interior components and singular vertices and some of the meridional disks in the intersection with $S_1$ may also be eliminated. 

\end{proof}

\begin{figure}
    \centering
    \includegraphics[width=0.8\linewidth]{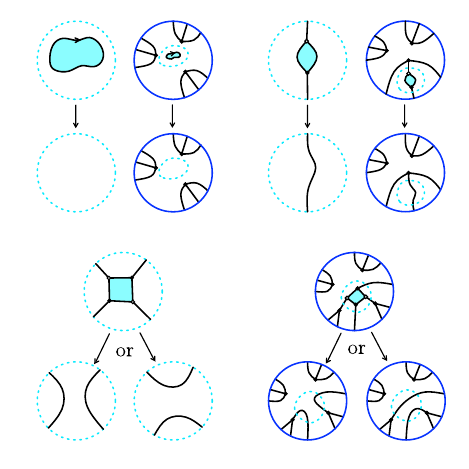}
    \caption{The impact of $2k$-compression surgeries on meridional intersection patterns shown in a neighbourhood of the compressing disks embedded in $D$. Note that due to the possible surface removal step it may be the case that the meridional intersection pattern resulting from a \textit{reduced} $2k$-compression is a subcomplex of one of the examples above. The two options for the $2$-compression correspond to $X^+$ and $X^-$.}
    \label{fig:meridionalsurgeries}
\end{figure}

\begin{lem}\label{lem:floatingcomponents}
    Let $\Gamma$ be a connected finite trivalent bipartite multigraph embedded in the plane. There exists an open disk component of $\mathbb{R}^2\setminus \Gamma$ that is bounded by a cycle of two or four edges in $\Gamma$.
\end{lem}

\begin{proof}
    First consider the embedding $\Gamma \hookrightarrow \mathbb{R}^2\hookrightarrow S^2$ via the one point compactification and let $\mathcal{C}$ be the cell decomposition of $S^2$ induced by $\Gamma$ such that $\mathcal{C}^{(1)}=\Gamma$. All but one of the $2$-cells in $\mathcal{C}^{(2)}$ lift to a disk in $\mathbb{R}^2\setminus \Gamma$. Let $c_{2k}$ be the number of $2$-cells of $\mathcal{C}^{(2)}$ cobounding a $2k$-cycle in $\Gamma$. Each edge in $\Gamma$ occurs exactly twice in the boundaries of these $2$-cells and each vertex occurs exactly three times. From this observation it follows from computing the Euler characteristic by counting cells that:
    
    \begin{align*}
        2  &= \chi(S^2)\\
        &= \sum_{k=1}^{\infty} c_{2k} - \frac{1}{2}\sum_{k=1}^{\infty} 2k c_{2k} + \frac{1}{3}\sum_{k=1}^{\infty} 2k c_{2k} \\
        &= \sum_{k=1}^{\infty} (1-\frac{k}{3})c_{2k}\\
    \end{align*}
    
    This implies that $2c_2 + c_4 \geq 6$ and so $c_2+c_4 \geq 3$. Since there are at least three $2$-cells that are bounded by cycles of length $2$ or $4$ it follows that at least two of these do not contain the point at $\infty$ and therefore lift to open disks in $\mathbb{R}^2 \setminus \Gamma$ that bound cycles of length $2$ or $4$ as required.

\end{proof}

\begin{cor}\label{cor:innermostdisk}
    Let $X$ be a $\chi$-optimal $\mathbb{Z}_3$-folded surface embedded in $L(3n,q)$ with meridional intersection pattern $(D,\Gamma_X)$. If any component of $\Gamma_X$ is contained in the interior of $D$ then there is an open disk component of $D\setminus \Gamma_X$ bounded by a simple closed curve component or a cycle of length two or four in a graph component of $\Gamma_X$.
\end{cor}

\begin{proof}
    This follows from a standard inner-most disk style argument to find a component $G$ of $\Gamma_X$ contained in the interior of $D$ such that every other component of $\Gamma_X$ is contained in a single open annulus component of $D\setminus G$. Then if $G$ is a simple closed curve the remaining component of $D\setminus G$ is an open disk cobounding it. Otherwise $G$ is a trivalent bipartite graph embedded in the plane and we can apply Lemma~\ref{lem:floatingcomponents} to find an open disk in $D\setminus G$ cobounding a cycle of length two or four.
\end{proof}

\begin{lem}\label{lem:simplifiedwithsurgieries}
    For any $d\in \mathbb{Z}^{>0}$ if there exists a $\chi$-optimal $\mathbb{Z}_3$-folded surface $X$ embedded in $L(3n,q) = S_1\cup S_2$ such that $X\cap S_1$ consists of $d$ meridional disks, then there exists a $\chi$-optimal $\mathbb{Z}_3$-folded surface $X'$ embedded in $L(3n,q)$ with a meridional intersection pattern $(D,\Gamma_{X'})$ such that:

    \begin{itemize}
        \item No component of $D\setminus \Gamma_{X'}$ is an open disk bounded by a simple closed curve component of $\Gamma_{X'}$ or a cycle of two or four edges in $\Gamma_{X'}$.
        \item $X'\cap S_1$ consists of at most $d$ meridional disks.
        \item Every component of $\Gamma{X'}$ intersects $\partial D$.

    \end{itemize}
\end{lem}

\begin{proof}
    This follows from induction on the number of components of the interior of the meridional intersection pattern using Lemma~\ref{lem:compressioninterior} and the first two properties are immediate. Corollary~\ref{cor:innermostdisk} implies the third property.

\end{proof}

\begin{lem}\label{lem:intervalbetterintersection}
    Let $X$ be a $\chi$-optimal $\mathbb{Z}_3$-folded surface embedded in $L(3n,q)$ with $S_1\cup S_2$ adapted to $X$ such that $X\cap S_1$ consists of $d>1$ meridional disks and $(D,\Gamma_X)$ is an intersection pattern of $X$ of the kind described in Lemma~\ref{lem:simplifiedwithsurgieries}. There exists an interval $I$ properly embedded in $D$ such that $I$ intersects $\Gamma_X$ locally transversally at points in $\text{Int}(\Gamma_X)$ and such that there is an arc $I'$ in $\partial D$ such that $\partial I' = \partial I$ and $|I\cap \Gamma_X| < |I'\cap \Gamma_X| \leq d$. Moreover there exists an equivalent genus $1$ Heegaard decomposition $S_1'\cup S_2'$ of $L(3n,q)$ adapted to $X$ such that $X\cap S_1'$ consists of fewer than $d$ meridional disks.
\end{lem}

\begin{proof}

    If any component of $\Gamma_X$ is an interval whose boundary vertices are consecutive in the cyclic order of $\partial(\Gamma_X)\subset \partial D$ then it is easy to construct $I$ and $I'$ by taking another embedded interval parallel to this component of $\Gamma_X$ such that $0=|I\cap \Gamma_X| < |I'\cap \Gamma_X|=2 \leq d$. Similarly if any two consecutive vertices of $\partial(\Gamma_X)$ are cobounded by two components of $\text{Int}(\Gamma_X)$ that share a singular vertex on their boundary then it is also not hard to construct $I$ parallel to the union of these interior components such that $1=|I\cap \Gamma_X| < |I'\cap \Gamma_X|=2 \leq d$. These two simple cases are shown in Figure~\ref{fig:intervalbetterintersectionarcs}. We now assume that neither of these occurs. Let $\Gamma_X'$ be a connected component of $\Gamma_X$ with the property that the vertices of $\partial(\Gamma_X')$ form a convex set with respect to the cyclic order induced by the orientation of $\partial D$. The existence of such a $\Gamma_X'$ follows from the fact that the partition of the vertices of $\partial(\Gamma_X)$ induced by the connected components of $\Gamma_X$ is non-crossing. If $\Gamma_X$ has more than one connected component then there are at least two possible choices for $\Gamma_X'$. Since $\Gamma_X'$ is connected, every component of $D\setminus \Gamma_X'$ is an open disk and so there is a cell complex $\mathcal{C}$ of $D$ induced by $\Gamma_X'$ such that $\Gamma_X'\cup \partial D = \mathcal{C}^{(1)}$. If $\Gamma_X$ is not connected then there is a unique $2$-cell $H$ in $\mathcal{C}^{(2)}$ that contains $\Gamma_X\setminus \Gamma_X'$. If $\Gamma_X$ is connected then for the simplicity of the argument let $H$ be any of the $2$-cells whose boundary intersects $\partial D$. The orientation of $\partial D$ induces a cyclic order of the $1$-cells whose image is contained in $\partial D$. Each of these $1$-cells on the boundary is contained in precisely one $2$-cell and so this induces a cyclic order on this family of $2$-cells of $\mathcal{C}$ whose boundary intersects $\partial D$ which we denote $B$. By counting the number of $1$-cells contained in $\partial D$ it follows that $|B| = |\partial(\Gamma_X')|$. Clearly $H \in B$. If we cut the cyclic order of $B$ at $H$ then we obtain a linear order of $B\setminus \{H\} = \{B_i\}_{i=1}^{|\partial(\Gamma_X')|-1}$.

\begin{figure}
    \centering
    \includegraphics[width=0.5\linewidth]{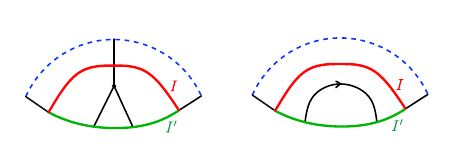}
    \caption{Two simple cases allowing for a Heegaard splitting adapted to $X$ such that $S_1\cap X$ consists of fewer meridional disks.}
    \label{fig:intervalbetterintersectionarcs}
\end{figure}

    Let $f:\mathcal{C}^{(2)}\rightarrow \mathbb{Z}^{>0}$ be the function that assigns to any $2$-cell the number of components of $\text{Int}(\Gamma_X')$ that are contained in its boundary.
    
    Let $b_i = |f^{-1}(i)\cap (B\setminus \{H\})|$, the number of $2$-cells in $B\setminus \{H\}$ that are bounded by $i$ components of $\text{Int}(\Gamma_X')$. Let $c_i = |f^{-1}(i) \setminus B|$ be the number of $2$-cells whose closure is contained in the interior of $D$ which are bounded by $i$ components of $\text{Int}(\Gamma_X')$. Let $h = f(H)$. The orientation data of $\Gamma_{X}'$ as a $\mathbb{Z}_3$-folded $1$-manifold implies that $c_i=0$ if $i$ is odd.
    
    The number of $2$-cells in $\mathcal{C}$ is $1+\sum_{i=1}^{\infty} \left[ b_i + c_{2i}\right]$. Each $1$-cell contained in $\text{Int}(\Gamma_X')$ is part of the boundary of two $2$-cells and so there are $\frac{h}{2}+\sum_{i=1}^{\infty} \left[\frac{2ic_{2i} +ib_i}{2} \right]$ such $1$-cells. Each $1$-cell contained in $\partial D$ is in the boundary of exactly one $2$-cell and so there are $1+\sum_{i=1}^{\infty} b_i$ such $1$-cells. The $0$-cells of $\mathcal{C}$ are either singular or boundary vertices of $\Gamma_X'$. Singular vertices appear in the boundary of three $2$-cells and so there are $\frac{h-1}{3}+\sum_{i=1}^{\infty}\frac{(i-1)b_i + 2i c_{2i}}{3}$ such vertices. Boundary vertices appear in the boundary of two $2$-cells and so there are $1+\sum_{i=1}^{\infty}b_i$ such vertices. $\mathcal{C}$ is a cell complex of this disk $D$ and so it follows that:

    \begin{multline*}
        1+\sum_{i=1}^{\infty} \left[ b_i + c_{2i}\right] - \left(\frac{h}{2}+\sum_{i=1}^{\infty} \left[\frac{2ic_{2i} +ib_i}{2} \right] + 1+\sum_{i=1}^{\infty} b_i\right) \\
         + \frac{h-1}{3}+\sum_{i=1}^{\infty}\frac{(i-1)b_i + 2i c_{2i}}{3} + 1+\sum_{i=1}^{\infty}b_i = \chi(D) = 1 \\
    \end{multline*}

    This simplifies to Equation~\ref{eq:cellcounting}.
    \begin{equation}\label{eq:cellcounting}
        \frac{4-h}{6} + \sum_{i=1}^{\infty}\left[\frac{4-i}{6}b_i +\frac{3-i}{3}c_{2i}\right] = 1
    \end{equation}
    
    By our assumption that neither case in Figure~\ref{fig:intervalbetterintersectionarcs} occurs it follows that $b_1=b_2=0$ and that $h\geq 2$. Since $\Gamma_X$ is of the kind described in Lemma~\ref{lem:simplifiedwithsurgieries} then it also follows that $c_2=c_4=0$. If we input these assumptions to the above equality we obtain inequality~\eqref{eq:prelimineq}

    \begin{equation}\label{eq:prelimineq}
        \sum_{i=1}^{\infty} (4-i)b_i \geq 4
    \end{equation}

    In particular it follows that $b_3 \geq 4+ \sum_{i=5}^{\infty} b_i$. Hence there exist at least three (closed) intervals in the total order of $B\setminus \{H\}$ of the form $[B_i,B_{i+k}]$ for $k\geq 1$ such that $f(B_i) = f(B_{i+k}) = 3$ and $f(B_{i+j})=4$ for all $1 \leq j<k$ if $k>1$. The family of subcomplexes that this describes are depicted in Figure~\ref{fig:b3tob3interval} and we will call this a \textit{special sequence} of length $k$. Also shown in this figure are a pair of embedded intervals $I,I'$ such that $I'\subset \partial D$ and $I$ is properly embedded in $D$ intersecting $\Gamma_X'$ transversally at points in $\text{Int}(\Gamma_X')$. In each case for an interval of this kind $[B_i,B_{i+k}]$ we have $|I'\cap \Gamma_X'| = k+2$ and $|I \cap \Gamma_X'| = k+1$. If we can show that $d\geq k+2$ then we are done. Hence we make the assumption that $d<k+2$ seeking a contradiction.

\begin{figure}
    \centering
    \includegraphics[width=0.7\linewidth]{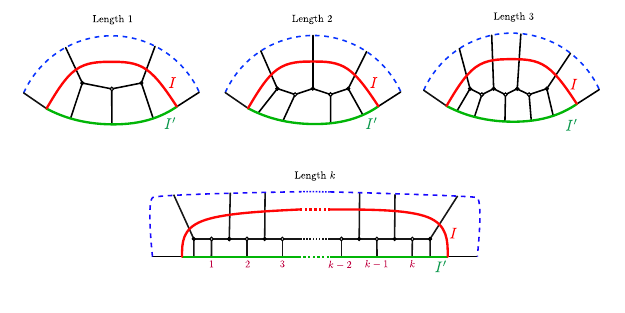}
    \caption{\textit{Special sequences} in meridional intersection patterns allowing for a reduction in the number of meridional disks.}
    \label{fig:b3tob3interval}
\end{figure}

    Because the edges of $\text{Int}(\Gamma_X')$ cobounding each vertex $v\in \text{Sing}(\Gamma_X')$ are oriented such that they induce a consistent orientation on $v$ it follows that a pair of vertices $v_1,v_2$ in $\partial(\Gamma_X')\sqcup \text{Sing}(\Gamma_X')$ have the same induced orientation if and only if any path of edges in $\Gamma_X'$ with endpoints $v_1,v_2$ is of even length. In particular if we label the boundary vertices $\partial(\Gamma_X') = \{v_i\}_{i=1}^{|\partial \Gamma_X'|}$ such that $\partial(B_i)\cap \partial(\Gamma_X') = \{v_i,v_{i+1}\}$ then since $\partial(B_i) \cap \Gamma_X'$ is a path of $f(B_i)$ edges in $\text{Int}(\Gamma_X')$ joining $v_i$ to $v_{i+1}$ we can deduce that $v_i$ has the same orientation as $v_{i+1}$ if and only if $f(B_i)$ is even. 

    In particular consider the vertices $I'\cap \Gamma_X' = \{v_i,\ldots v_{i+k+1}\}$. Without loss of generality if we assume that $v_i$ has \textit{black} orientation then $v_{i+1}$ must have \textit{white} orientation since $f(B_i) = 3$ is odd. Moreover since $f(B_{i+j})=4$ is even for all $1 \leq j<k$ it follows that (if $k>1$) the vertices $\{v_{i+1},\ldots, v_{i+k}\}$ are also white. Finally since $f(B_{i+k}) = 3$ is odd and $v_{i+k}$ is white it follows that $v_{i+k+1}$ is black. Recall from Lemma~\ref{lem:boundaryvertexpattern} that if $\sigma \in Sym(\partial \Gamma_X)$ is the permutation sending each boundary vertex to the next boundary vertex with respect to the cyclic order induced by the orientation of $\partial D$ then $\sigma^d$ preserves the orientation of $\partial \Gamma_X$. The vertices $\{v_{i+1},\ldots,v_{i+k+1}\}\subset \partial \Gamma_X$ are a closed interval in this cyclic order since they are a sub-interval of the closed interval $\{v_1,\ldots ,v_{|\partial\Gamma_X'|}\}$ (which is an interval by definition of $\Gamma_X'$). It follows that $d\geq k+1$ since if $d\leq k$ then $\sigma^d(v_i) = v_{i+d} \in \{v_{i+1},\ldots v_{i+k}\}$ does not preserve the orientation of $v_i$. Since we assumed that $d<k+2$ this implies that $d=k+1$.

    Suppose $k=1$. It follows that $d=k+1 = 2$ and so $X\cap S_1$ consists of two meridional disks of opposite orientations. From the Mayer-Vietoris sequence of $(L(3n,q),S_1,S_2)$ it follows that $[X]$ is in the kernel of $H_2(L(3n,q);\mathbb{Z}_3)\rightarrow H_1(S_1\cap S_2;\mathbb{Z}_3)$ and hence the image of $H_2(S_1;\mathbb{Z}_3)\oplus H_2(S_2;\mathbb{Z}_3)\rightarrow H_2(L(3n,q);\mathbb{Z}_3)$. But $H_2(S_1;\mathbb{Z}_3)=H_2(S_2;\mathbb{Z}_p) = 0$ and so $[X]=0$ which is a contradiction of our initial assumption. This homology argument implies that no special sequence with $k=1$ and $d<k+2$ exists in $B\setminus \{H\}$.

    Note that if $d=k+1$ and a special sequence of length $k>1$ exists then the black vertices in $\Gamma_X'$ are separated by $k$ white vertices in the total order. We call a $2$-cell $B_i\in B\setminus \{H\}$ \textit{odd} if $f(B_i)$ is odd. For $1<i<|\partial \Gamma_X'|-1$ if $B_i$ is odd then one of $\{v_i,v_{i+1}\}$ is black and so respectively one of $\{B_{i-1},B_{i+1}\}$ is odd and the other is even. If we partition the odd cells of $B\setminus \{H\}$ into subsets of adjacent cells then every subset will be a pair of adjacent cells except possibly for two singletons $\{B_1\}$ and $\{B_{|\partial \Gamma_X'|-1}\}$. 
    
    Since we assumed that no special sequence with $k=1$ exists it follows that each pair contains at most one cell in $f^{-1}(3)$. From this it follows that $b_3-2 \leq \sum_{i=2}^{\infty} b_{2i+1}$. A contradiction follows from the fact that $\sum_{i=2}^{\infty} b_{2i+1} \leq \sum_{i=5}^{\infty} b_{i} \leq b_3 -4$. Therefore $d\geq k+2$ and the first part of the Lemma is complete.

    Without loss of generality we can assume that $I$ and $I'$ are as above and that $|I'\cap \partial(\Gamma_X)| = d$. This is because where $|I'\cap \partial(\Gamma_X)| < d$ we can simply extend one end of $I'$ by some interval in $\partial D$ to add the required number of intersection points while extending $I$ parallel to $\partial D$ to add the same number of points in the interior of $D$. 
    
    $I'$ intersects the $d$ parallel meridians of $X\cap \partial S_1$ once each and so its endpoints $\partial I'$ are in the same component of $\partial S_1 \setminus X$. Let $C_1'$ be a simple closed curve on $\partial S_1$ consisting of the union of $I'$ with an interval in $\partial S_1 \setminus X$ connecting its endpoints. Since $C_1'$ intersects each meridian curve in $\partial S_1 \cap X$ in a single point we can deduce from the intersection form on the homology of the torus $\partial S_1$ that with some choice of orientation $[C_1'] = \lambda_1 + a \mu_1$ for some $a\in \mathbb{Z}$. It follows from the theory of normal arcs on a triangulation of the torus with two triangles that the ambient isotopy classes of oriented simple closed curves on the torus are classified by their homology classes. An ambient isotopy of $\partial S_1$ induces isotopies of the embedding $\partial S_1 \hookrightarrow L(3n,q)$ that fix $\partial S_1$ setwise. The PL isotopy extension theorem of~\cite{hudson1966} implies that this isotopy extends to an ambient isotopy of $L(3n,q)$ that carries $C_1'$ to any other curve on $\partial S_1$ realising the same homology class (up to a choice of orientation). For any $a\in \mathbb{Z}$ it can be shown constructively that there exists an annulus $A \subset S_1$ such that one component of $\partial A$ is the core circle $C_1$ and the other component is a curve on $\partial S_1$ which realises the homology class $\lambda_1 + a\mu_1$ with some choice of orientation. The extension of an isotopy over this annulus implies that $C_1$ is ambient isotopic to $C_1'$.
    
    Let $C_1''$ be the curve given by $(C_1' \setminus I')\sqcup I$. Together $I$ and $I'$ bound a disk in $D$ over which an isotopy can be performed to show that $C_1''$ is isotopic to $C_1'$. This isotopy extends to an ambient isotopy with a further application of the isotopy extension theorem.

    Following Lemma~\ref{lem:meridionaldiskintersection} it is sufficient to show that $C_1''$ is mocktransverse to $X$ in order to complete the proof. To see this note that $C_1''$ intersects $X$ at fewer than $d$ points in $\text{Int}(\Gamma_X)$ and that these intersection points are locally transverse in $D$. Since $(D,\Gamma_X)$ bicollars in $(S_2,X\cap S_2)$ by Lemma~\ref{lem:transversemeridionaldisk} it follows that $C_1''$ is locally transverse to the $2$-manifold $\text{Int}(X)$. Now recall the proof of Lemma~\ref{lem:meridionaldiskintersection} and note that $C_1''$ has the required properties of local transversality with $\text{Int}(X)$ and ambient isotopy equivalence with $C_1$ such that we can apply Lemma~\ref{lem:changingcorecircle} to obtain a new Heegaard decomposition $S_1'\cup S_2'$ of $L(3n,q)$ adapted to $X$ such that $S_1'$ is a regular neighbourhood of $C_1''$. Since $|C_1''\cap X|<d$ it follows that $S_1' \cap X$ consists of fewer than $d$ meridional disks as required.

\end{proof}

\begin{proof}[Proof of Lemma~\ref{lem:simplifiedl3nqintersectionpattern}]

    The proof of Lemma~\ref{lem:simplifiedl3nqintersectionpattern} is a straightforward induction on the number of meridional disks. The general structure of the argument is as follows:
    
    We start with any $\chi$-optimal $\mathbb{Z}_3$-folded surface $X$ embedded in a lens space with a genus $1$ Heegaard decomposition $S_1\cup S_2$ and a meridional intersection pattern $(D,\Gamma_X)$. Fixing $S_1,S_2,D$ we can apply Lemma~\ref{lem:simplifiedwithsurgieries} to replace $X$ with another $\chi$-optimal $\mathbb{Z}_3$-folded surface whose corresponding meridional intersection pattern has been maximally simplified with the reduced compression surgeries described in Section~\ref{sec:reducedsurgeries}. If $X\cap S_1$ consists of more than one meridional disk then we fix $X$ and apply Lemma~\ref{lem:intervalbetterintersection} to replace $S_1,S_2$ with an equivalent genus $1$ Heegaard decomposition adapted to $X$ such that now $S_1\cap X$ consists of strictly fewer meridional disks. We then use Lemma~\ref{lem:transversemeridionaldisk} to obtain a new mock-transverse meridional disk $D$ that determines a new meridional intersection pattern as per Definition~\ref{dfn:meridionalintersectionpattern}. 

    Repeatedly applying Lemmas~\ref{lem:simplifiedwithsurgieries}, ~\ref{lem:intervalbetterintersection} and ~\ref{lem:transversemeridionaldisk} results in a $\chi$-optimal $\mathbb{Z}_3$-folded surface, adapted genus $1$ Heegaard splitting and meridional intersection pattern $(D,\Gamma_X)$ with the required properties.

\end{proof}

\subsection{Snake moves}\label{sec:snakesec}

To continue simplifying these meridional intersection patterns we need to consider how $X\cap D$ and $X\cap S_2$ fit together. To this end, as in Definition~\ref{dfn:meridionalintersectionpattern}, let $B_2$ be the $3$-ball obtained by cutting $S_2$ along $D$ and $\tilde{X}$ be the $\mathbb{Z}_3$-folded surface with boundary properly embedded in $B_2$ obtained by lifting $X\cap S_2$ over the projection $f:B_2\rightarrow S_2$. The meridional intersection pattern $\Gamma_X \subset D$ determines $\partial \tilde{X} \subset \partial B_2$ as follows. Let $f^{-1}(D) = D^+\sqcup D^-$ where $D^+\xrightarrow{f}D$ is an orientation preserving homeomorphism and $D^-\xrightarrow{f}D$ is orientation reversing. The meridional intersection pattern $(D,\Gamma_X)$ determines the intersection of $\partial \tilde{X}$ with $D^+ \sqcup D^-$. Label the vertices of $X\cap \partial D = \{x_i\}_{i\in \mathbb{Z}_{3n}}$ such that the canonical cyclic order of $\mathbb{Z}_{3n}$ is equivalent to the cyclic order of the vertices induced by the orientation of $\partial{D}$. Let $x_i^{\pm}$ be the two vertices in $\partial D^{\pm}$ such that $f(x_i^{\pm}) = x_i$. $X\cap \partial S_2$ is a simple closed curve and with choice of sign/orientation it realises the homology class $\mu_1 = q\mu_2 + 3n \lambda_2 \in H_1(\partial S_2;\mathbb{Z})$ since it bounds a meridional disk in $S_1$. This slope implies that $\partial \tilde{X} \cap f^{-1}(\partial S_2)$ consists of $3n$ disjoint intervals properly embedded in the annulus $f^{-1}(\partial S_2)\subset \partial B_2$. With an appropriate initial choice of orientation of $D$ the endpoints of these intervals are given by the pairs $\{(x_i^{-},x_{i+q}^+)\}_{i\in \mathbb{Z}_{3n}}$ (see Figure~\ref{fig:snakemove} for example). We refer the reader to the Bredon and Wood paper~\cite{bredonwood1969} for more details, which uses the same approach to study non-orientable surfaces in lens spaces. Figure~\ref{fig:b_2boundarypattern} shows an example of $\partial{\tilde{X}} \subset \partial B_2$ with a particular choice of meridional intersection pattern the lift of $X\cap (\partial S_2 \cup D) \subset \partial S_2 \cup D$ to $\partial \tilde{X} \subset \partial B_2$ for the lens spaces $L(15,1),L(15,2),L(15,4)$. The diagrams in the Figure should be interpreted as stereographic projections of the embedding $\partial \tilde{X} \subset \partial B_2$ with the images of $D^+,D^-$ and $f^{-1}(\partial S_2)$ labelled accordingly.

\begin{figure}
    \centering
    \includegraphics[width=0.8\linewidth]{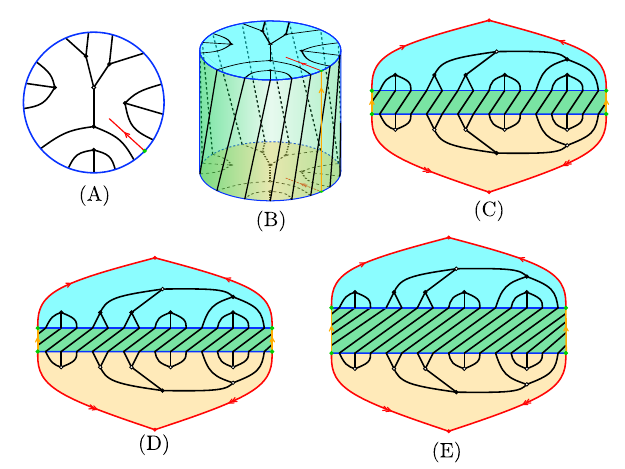}
    \caption{A meridional intersection pattern $(D,\Gamma_X)$ with $|\partial \Gamma_X| = 15$ and the corresponding $\mathbb{Z}_p$-folded $1$-manifold $\partial \tilde{X} \subset \partial B_2$ for the lens spaces $L(15,1), L(15,2),L(15,4)$.}
    \label{fig:b_2boundarypattern}
\end{figure}

We now have the flexibility to describe \textit{snake} moves. These moves are inspired by~\cite{segerman2024}. The effect of these moves on meridional intersection patterns is shown in Figure~\ref{fig:snakemove}, adding two new singular vertices. The snake move arises from an ambient isotopy that pulls an interval of $\text{Sing}(X)$ in the collar of $X\cap D$ on one side of $D$, over the collar of $X\cap \partial S_2$ and through $D$ in the collar on the other side. In Figure~\ref{fig:snakemodel} we show the lift of this isotopy to the embedding $\tilde{X}\subset B_2$ in the case $q \neq \pm 1 \mod 3n$. When $q= \pm1 \mod 3n$ snake moves arise in exactly the same way, differing only in the fact that the `head' of the snake emerges closer to its `tail' after passing through the meridional disk. We encourage the reader to visit \href{https://willhobkirk.com/}{willhobkirk.com} for interactive 3D models of this isotopy.Snake moves arise from components $C$ of $D\setminus \Gamma_X$ whose boundary consists of an interval of $\partial D$ and two intervals of $\text{Int}(\Gamma_X)$. The interval of $\text{Sing}(X)$ can be taken as an interval in the collar of the singular vertex in $\partial C$ in the positive or negative collar of $(D,\Gamma_X)$ and so there are two possible snake moves related to such a component $C$. These are summarised in Figure~\ref{fig:snakemove}.

\begin{figure}
    \centering
    \includegraphics[width=0.8\linewidth]{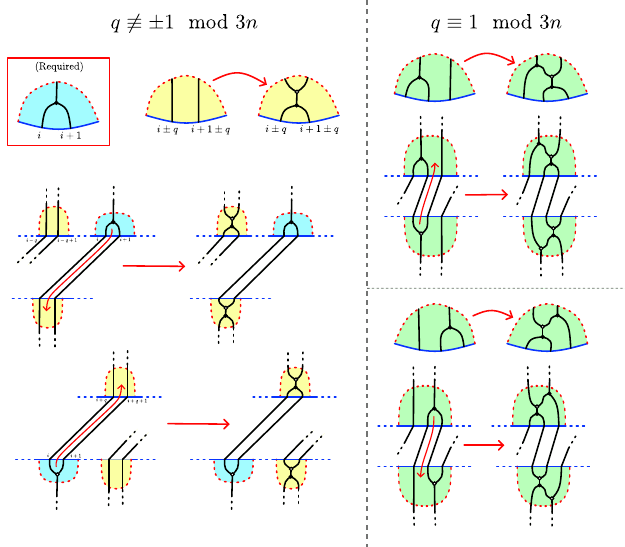}
    \caption{The meridional intersection pattern operations induced by the snake moves.}
    \label{fig:snakemove}
\end{figure}

\begin{figure}
    \centering
    \begin{tabular}{c@{}c}
        Before & After  \\

        \begin{annotate}{\includegraphics[width = 0.4\linewidth]{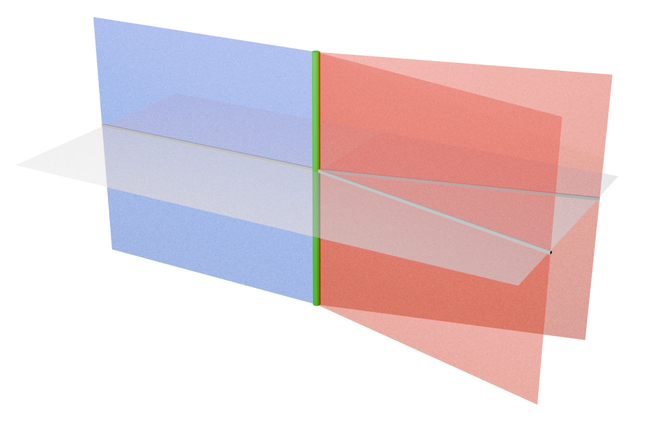}}{0.4}

        \end{annotate}
        &
        \begin{annotate}{\includegraphics[width = 0.4\linewidth]{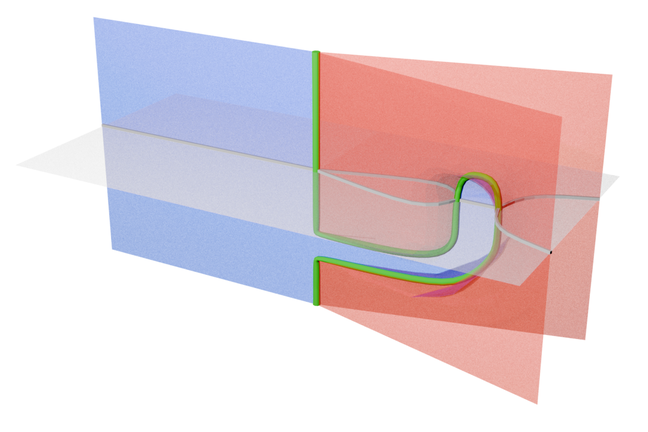}}{0.4} 

        \end{annotate}

        \\
        \begin{annotate}{\includegraphics[width = 0.4\linewidth]{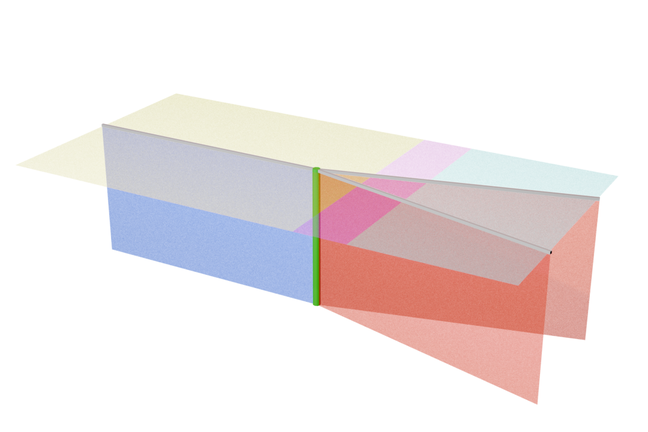}}{0.4} 

            \callout{5,-5}{$D^+$}{4,0}
            \callout{3,4}{$f^{-1}(\partial S_2)$}{1.5,0.5}
            \callout{-5,-5}{$D^-$}{-3,1}
        \end{annotate}
        &
        \begin{annotate}{\includegraphics[width = 0.4\linewidth]{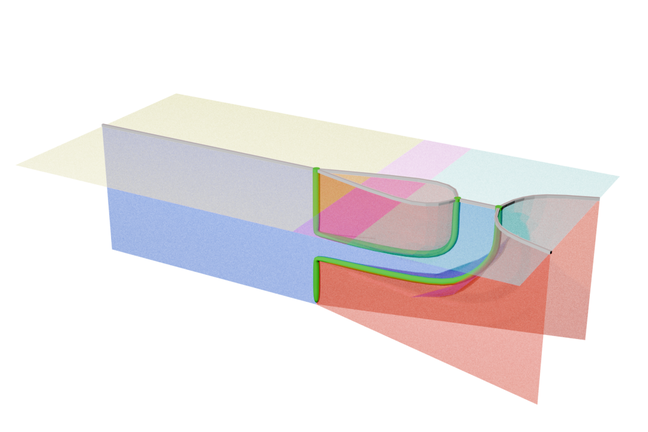}}{0.4}

        \end{annotate}
        \\

        \begin{annotate}{\includegraphics[width = 0.4\linewidth]{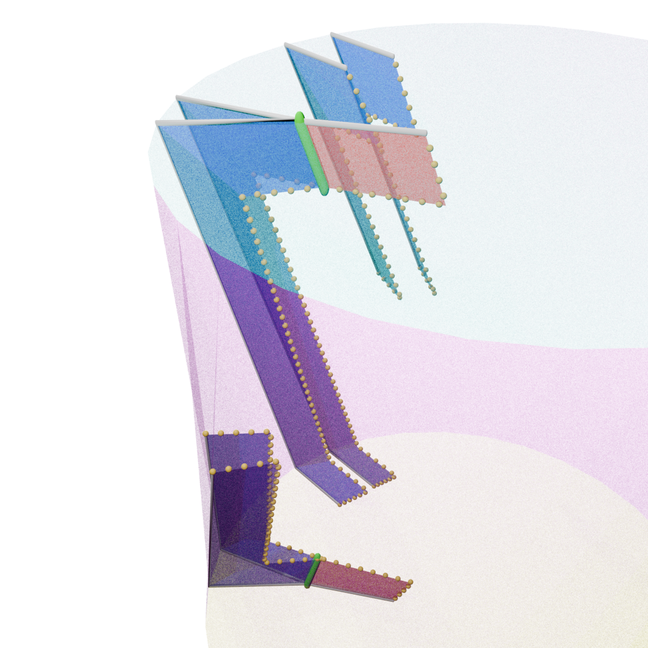}}{0.4} 

            \callout{-6,6}{$D^+$}{5,5}
            \callout{-6,0}{$f^{-1}(\partial S_2)$}{3,-1}
            \callout{-6,-6}{$D^-$}{3,-5}
        \end{annotate}
        &
        \begin{annotate}{\includegraphics[width = 0.4\linewidth]{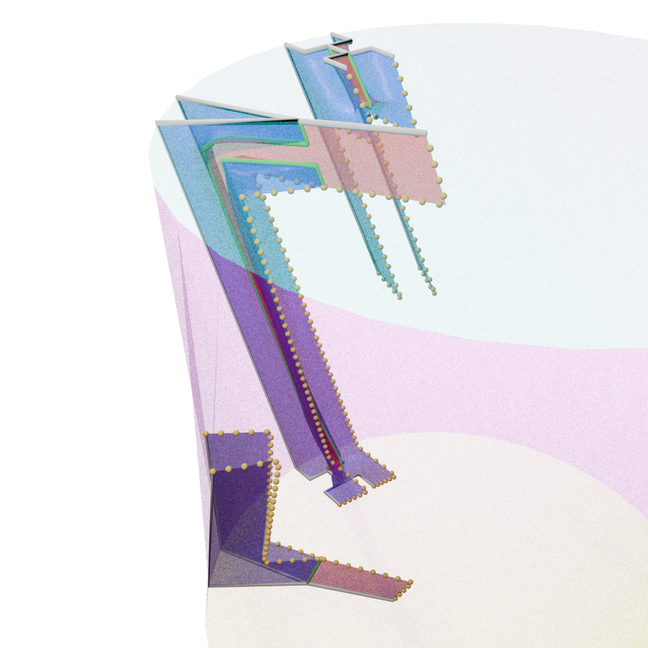}}{0.4}

        \end{annotate}
        
    \end{tabular}
    \caption{Several views of a snake move where $q\neq \pm1 \mod 3n$. Row 1 shows a snake that adds two points of intersection of the singular locus of a $\mathbb{Z}_3$-folded surface with an arbitrary plane. Row 2 and row 3 show snakes that slither across the collar of $(\partial X,\partial S_2)$ and induce the snake move of meridional intersection patterns in Figure~\ref{fig:snakemove}. In both cases we depict $\tilde{X}\subset B_2$ in a neighbourhood of a disk in $\partial B_2$ that intersects $f^{-1}(\partial S_2)$, $D^+$ and $D^-$. In row 3 we indicate the intersection of $\tilde{X}$ with the internal boundary of this neighbourhood with a yellow dotted line.}
    \label{fig:snakemodel}
\end{figure}

Note that the two snake moves for meridional intersection patterns of $L(3n,q)$ depend on the value of $q$ since this determines how the collar of $(\partial S_2,X\cap \partial S_2)$ and the upper and lower collars of $(D,\Gamma_X)$ match up. Given a meridional intersection pattern $(D,\Gamma_X)$ of a $\mathbb{Z}_3$-folded surface $X$ and a meridional intersection pattern $(\Gamma',D)$ obtained from $\Gamma_X$ by a snake move, then there exists a $\mathbb{Z}_3$-folded surface $X'$ ambient isotopic to $X$ such that $\Gamma' = X'\cap D$ is a meridional intersection pattern of $X'$. In particular, $(\Gamma',D)$ is a meridional intersection pattern of a $\chi$-optimal $\mathbb{Z}_3$-folded surface if $(D,\Gamma_X)$ is. 

Now label the intervals of $\partial D \setminus X$ in the cyclic order induced by the orientation of $D$ as $\{I_i\}_{i\in \mathbb{Z}_{3n}}$. Lemma~\ref{lem:unbalancedquads} gives one application of these snake moves that will be used in Section~\ref{sec:recursionrelation}.

\begin{lem}\label{lem:unbalancedquads}
    Let $X$ be a $\chi$-optimal $\mathbb{Z}_3$-folded surface embedded in $L(3n,q) = S_1\cup S_2$ intersecting $S_1$ in a single meridional disk and with meridional intersection pattern $(D,\Gamma_X)$. Suppose for some $i$ and choice of sign $\pm$ there are two disks $C_i,C_{i\pm q} \in \pi_0(D\setminus \Gamma_X)$ with the following properties.

    \begin{itemize}
        \item $\partial C_i$ is the closure of the disjoint union of four components of $\text{Int}(\Gamma_X)$ and $I_i$.
        \item $\partial C_{i\pm q}$ is the closure of the disjoint union of two components of $\text{Int}(\Gamma_X)$ and $I_{i\pm q}$.
    \end{itemize}

    It follows that there exists another $\chi$-optimal $\mathbb{Z}_3$-folded surface $X'$ such that $X\cap S_1 = X' \cap S_1$ and $\Gamma_X' := X'\cap D$ is a meridional intersection pattern of $X'$ with at least two fewer singular vertices.

\end{lem}

\begin{proof}
    If this is the case for a meridional intersection pattern of a $\chi$-optimal $\mathbb{Z}_3$-folded surface then perform a snake move arising from $C_{i\pm q}$ that separates $C_i$ into two disks, one of which is contained in $\text{Int}(D)$ and cobounds a $4$-cycle in the resulting meridional intersection pattern. Perform a reduced $4$-compression of the kind described in Lemma~\ref{lem:compressioninterior}. The snake move increases the number of singular vertices of the meridional intersection pattern by two and the reduced $4$-compression decreases the number of singular vertices by at least four giving the desired net reduction in the number of singular vertices of the meridional intersection pattern. Figure~\ref{fig:unbalancedquads} shows this diagrammatically. Note that the surface removal step used to define the reduced $4$-compression cannot eliminate the single meridional disk in the intersection with $S_1$ without changing the homology class.

\end{proof}

\begin{figure}
    \centering
    \includegraphics[width=0.8\linewidth]{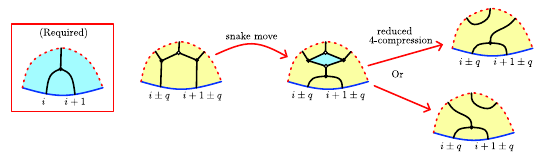}
    \caption{The snake move and subsequent $4$-compression move that results in a reduction in the size of $\Sing(\Gamma_X)$ in meridional intersection patterns of the kind described in Lemma~\ref{lem:unbalancedquads}.}
    \label{fig:unbalancedquads}
\end{figure}

\subsection{A recursion relation for $L(3n,1)$}\label{sec:recursionrelation}

Let $n\in \mathbb{Z}^{>0}$. Since $H_2(L(3n,1);\mathbb{Z}_3) = \mathbb{Z}_3$, $\chi_3(x) = \chi_3(-x)$ for all $x\in H_2(L(3n,1);\mathbb{Z}_3)$ and $\chi_3(0)=0$ it follows that $\chi_3:H_2(L(3n,1);\mathbb{Z}_3) \rightarrow \mathbb{Z}$ is fully determined by its value on any non-trivial class. Hence we define $\chi_3(L(3n,1)):= \chi_3(x)$ where $x$ is any non-trivial class in $H_2(L(3n,1);\mathbb{Z}_3)$.

In this section we will prove Lemma~\ref{lem:l3n1recursion}.

\begin{lem}\label{lem:l3n1recursion}
    $\chi_3(L(3n,1)) = \chi_3(L(3(n-1),1)) - 1$ for all $n>1$.
\end{lem}

The proof of Theorem~\ref{thm:l3n1norm} then follows from induction on Lemma~\ref{lem:l3n1recursion} and the base case:

\begin{lem}\label{lem:l3_1norm}
    $\chi_3(L(3,1)) = 1$
\end{lem}

\begin{proof}
    No component of a non-separating $\mathbb{Z}_3$-folded surface embedded in $L(3,1)$ can be a $2$-sphere because $H_2(L(3,1);\mathbb{Z})=0$. The case analysis in the proof of Lemma~\ref{lem:eulercharbounds} paired with the observation yields an upper bound $\chi_3(L(3,1))\leq 1$.
    
    Figure~\ref{fig:l3_1norm} shows a non-separating $\mathbb{Z}_3$-folded surface $X \subset L(3,1)$ such that $\chi(X) = 1$ and so $\chi_3(x)\geq 1$ completing the proof.
\end{proof}

\begin{figure}
    \centering
    \includegraphics[width=0.7\linewidth]{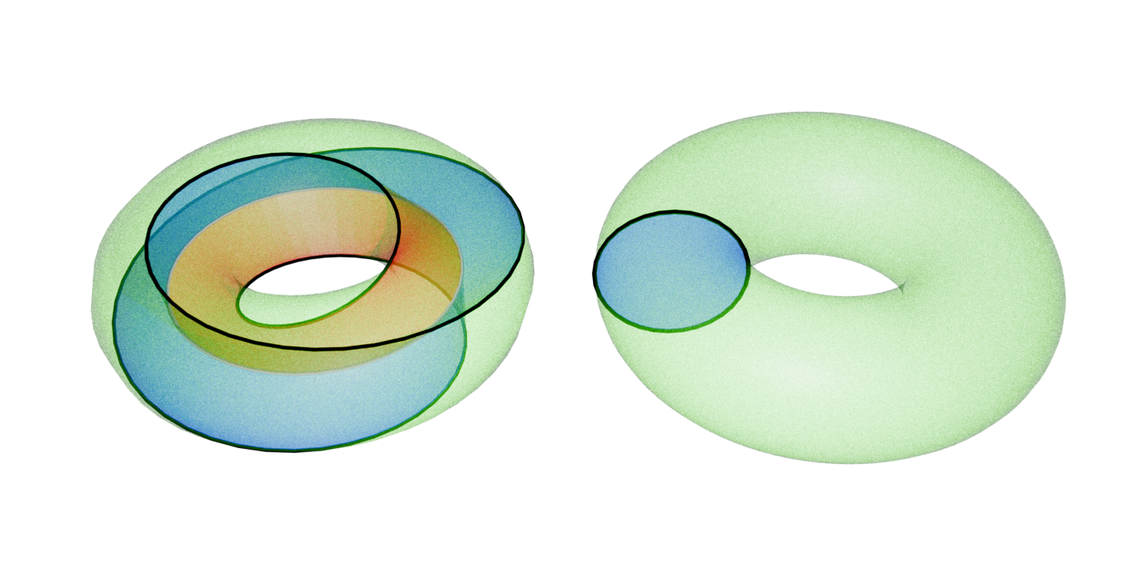}
    \caption{The $\mathbb{Z}_3$-crosscap that embeds in $L(3,1)$.}
    \label{fig:l3_1norm}
\end{figure}

The first ingredient in the proof of Lemma~\ref{lem:l3n1recursion} is some meridional intersection pattern combinatorics for which Definition~\ref{dfn:boundarytripods} is convenient.

\begin{dfn}[Boundary Tripods]\label{dfn:boundarytripods}
    Let $(D,\Gamma)$ be a meridional intersection pattern of a $\mathbb{Z}_3$-folded surface. A connected component $\Gamma_0$ of $\Gamma$ is a \emph{boundary tripod} if it has the following properties:

    \begin{itemize}
        \item $\Gamma_0$ is the graph $K_{3,1}$ (i.e. $\partial \Gamma_0$ consists of three vertices which cone to one singular vertex in $\text{Sing}(\Gamma_0)$).
        \item $\Gamma \setminus \Gamma_0$ is contained in one connected component of $D \setminus \Gamma_0$ (equivalently $\partial \Gamma_0$ is a closed interval in the cyclically ordered set $\partial \Gamma$).
    \end{itemize}
\end{dfn}

Figure~\ref{fig:boundarytripodbigon} shows what a boundary tripod looks like.
We can now prove Lemma~\ref{lem:minimalimpliesboundarytripod}.

\begin{lem}\label{lem:minimalimpliesboundarytripod}
    Let $(D,\Gamma_X)$ be a meridional intersection pattern of a $\chi$-optimal $\mathbb{Z}_3$-folded surface $X\subset L(3n,1) =S_1\cup S_2$ such that $X\cap S_1$ consists of a single meridional disk and $(|\Sing(\Gamma_X),|\pi_0(\Int(\Gamma_X))|)$ is minimal among such meridional intersection patterns (ordered lexicographically). It follows that $\Gamma_X$ has at least one \emph{boundary tripod}.
\end{lem}

\begin{proof}
    We will prove the contrapositive. Suppose $\Gamma_X$ has no boundary tripods. We seek to show that it is not minimal with respect to the lexicographic order. 
    
    First, if $D\setminus \Gamma_X$ admits any disk component that cobounds a simple closed curve or a cycle of $2$ or $4$ edges in $\Gamma_X$ it follows from Lemma~\ref{lem:compressioninterior} that it is not minimal. 
    
    Now assume that no component of $D\setminus \Gamma_X$ is a disk bounded by a simple closed curve or a cycle of length $2$ or $4$. Let $\Gamma_X'$ be a connected component of $\Gamma_X$ such that $\partial(\Gamma_X')$ is a closed interval in the cyclically ordered set $\partial(\Gamma_X)$ and define $B$, $f:B\rightarrow \mathbb{Z}^{>0}$, $H$, $h$, $b_i$ and $c_{2i}$ for $i\in \mathbb{Z}^{>0}$ exactly as in the proof of Lemma~\ref{lem:intervalbetterintersection}. Recall Equation~\ref{eq:cellcounting} which follows from an Euler characteristic argument by counting the cells of each dimension in the cell decomposition of $D$ induced by $\Gamma_X$.

    Since $X\cap S_1$ consists of a single meridional disk every vertex of $\partial \Gamma_X$ must have the same induced orientation. It follows from this that $h$ is even and $b_{2i-1}= 0$ for all $i\in \mathbb{Z}^{>0}$. We have also made an assumption that $c_2=c_4=0$. With these constraints Equation~\ref{eq:cellcounting} yields Equation~\ref{eq:inequalitycellcount}.

    \begin{equation}
        \label{eq:inequalitycellcount}
        \begin{aligned}
            1 &= \frac{4-h}{6} + \sum_{i=1}^{\infty}\left[\frac{4-i}{6}b_i +\frac{3-i}{3}c_{2i}\right] \\
                &\leq \frac{4-h}{6} + \sum_{i=1}^{\infty} \frac{4-i}{6}b_i \\
                &\leq \frac{2}{6} + \sum_{i=1}^{\infty} \frac{4-i}{6}b_i \\
                &= \frac{2}{6} + \sum_{i=1}^{\infty} \frac{4-2i}{6}b_{2i} \\
        \end{aligned}
    \end{equation}

    Rearranging we obtain $b_2 \geq 2 + \sum_{i=3}^{\infty}b_{2i}$. Note that the cells of $f^{-1}(2)$ cannot be adjacent in the total order of $B$ described in the proof of Lemma~\ref{lem:intervalbetterintersection} since otherwise it follows from Definition~\ref{dfn:boundarytripods} that $\Gamma_X'$ is a boundary tripod. If there is at least one element of $f^{-1}(\mathbb{Z}^{\geq 6})$ between consecutive elements of $f^{-1}(2)$ then $\sum_{i=3}^{\infty}b_{2i} \geq b_2 - 1$ which is a contradiction. We deduce from this that there must be cells in $f^{-1}(2)$ that are separated by an interval which is a subset of $f^{-1}(4)$. In particular we can locate a pair of cells $C_i$ and $C_{i\pm1}$ such that $f(C_i)= 4$ and $f(C_{i\pm1}) = 2$ and apply Lemma~\ref{lem:unbalancedquads} to reduce the number of singular vertices in $\Gamma_X$. This finishes the proof of the contrapositive.
    
\end{proof}

The remainder of the proof is more topological as follows:

\begin{proof}[Proof of Lemma~\ref{lem:l3n1recursion}]
    
    Let $X\subset L(3n,1) = S_1\cup S_2$ be $\chi$-optimal such that $X\cap S_1$ consists of a single meridional disk and $(D,\Gamma_X)$ is a meridional intersection pattern which is minimal in the sense of Lemma~\ref{lem:minimalimpliesboundarytripod} and accordingly let $\Gamma^{\psi}$ be a boundary tripod component of $\Gamma_X$. Figure~\ref{fig:boundarytripodbigon} shows how such a boundary tripod implies the existence of a disk $C$ in $\partial{B_2}\setminus \tilde{X}$ bounding a cycle of two edges in $\text{Int}(\tilde{X})$. Note that the map $q:B_2 \rightarrow S_2$ does not restrict to an embedding of $C$ and so the image of $C$ is not a $2$-compression disk of $X$. We can however use a collar of $(\partial B_2,\partial \tilde{X})$ in $(B_2,\tilde{X})$ to find a disk $C'$ parallel to $C$ in the interior of $B_2$ which embeds in $S_2$ (since the interior of $B_2$ embeds in $S_2$ as $S_2 \setminus (D\cup \partial S_2)$). A reduced $2$-compression of $X$ over the $2$-compression disk $q(C')$ yields another $\chi$-optimal $\mathbb{Z}_3$-folded surface $X'$ with a meridional intersection pattern $(D,\Gamma_{X'})$ which is some subcomplex of the meridional intersection pattern of $X$. This is because a $2$-compression of $C'$ which lies in $S_2\setminus (\partial S_2 \cup D)$ will yield a new $\mathbb{Z}_3$-folded surface with the same meridional intersection pattern, but the subsequent surface elimination step may remove some of the components of the interior of the $\mathbb{Z}_3$-folded surface that intersect $D$. Since $(X,\Gamma_X)$ already has the minimum possible number of singular vertices it follows that $\Gamma_X = \Gamma_{X'}$. Now let $\tilde{X}'$ be the lift of $X'\cap S_2$ over $B_2 \rightarrow S_2$.

\begin{figure}
    \centering
    \includegraphics[width=0.5\linewidth]{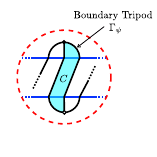}
    \caption{A boundary tripod in a meridional intersection pattern of $L(3n,1)$ which implies the existence of a bigon $C$ in $\partial B_2 \setminus \partial \tilde{X}$ cobounding two components of $\Int \partial \tilde{X}$.}
    \label{fig:boundarytripodbigon}
\end{figure}

    We will now describe a non-separating non-trivial $\mathbb{Z}_3$-folded surface $Y$ embedded in $L(3(n-1),1)$ such that $\chi(Y) = \chi(X')+1$. 

    We first describe $\tilde{Y} \subset B_2$ and then the map $B_2 \rightarrow S_2$. Let $E,E' \in \pi_0(\text{Int}(\tilde{X}'))$ be the two disks such that $(\partial E \cup \partial E') \cap \partial B_2 = \partial C$ (the fact that $E,E'$ are disks follows from the reduced $2$-compression we performed on $X$ to obtain $X'$). Let $\tilde{Y}$ be the $\mathbb{Z}_3$-folded surface in $B_2$ with underlying set $\tilde{X}' \setminus E'$ such that $\tilde{Y}\setminus \overline{E}$ inherits the orientation of $\text{Int}(\tilde{X}')\setminus (E'\cup E)$ in $\tilde{X}'$ and the orientation of $E$ in $\text{Int}(\tilde{Y})$ is the opposite of its orientation in $\text{Int}(\tilde{X}')$. A simple Euler characteristic computation on a cell complex of $\tilde{X}'$ in which $E'$ is a $2$-cell and $\partial E' \cap \partial C$ is a $1$-cell yields that $\chi(\tilde{Y}) = \chi(\tilde{X}')-1+1 = \chi(\tilde{X}')$.
    
    If we let $q_{X'}:B_2\rightarrow S_2$ be the quotient map such that $q_{X'}(\tilde{X}') = X'\cap S_2$ then $q_{X'}(\tilde{Y})$ is no longer a $\mathbb{Z}_3$-folded surface ($q_{X'}(\tilde{Y}) \cap \partial(S_2)$ is now a closed interval not a simple closed curve). To restore this property we will now modify $q_{X'}$. First let $D_Y \subset D$ be a closed disk such that $\partial D_Y$ is the union of an interval in $\partial D$ and an interval properly embedded in $D$ which separates the boundary tripod $\Gamma^{\psi}$ from $\Gamma_{X'} \setminus \Gamma^{\psi}$ in $D$ such that $D_Y\cap \Gamma_{X'} = \Gamma_{X'} \setminus \Gamma^{\psi}$ (see Figure~\ref{fig:newD_Yboundarypattern}). Lift $D_Y$ through $q_{X'}$ to a pair of disks $D_Y^{\pm} \subset D^{\pm}$. 

\begin{figure}
    \centering
    \includegraphics[width=0.8\linewidth]{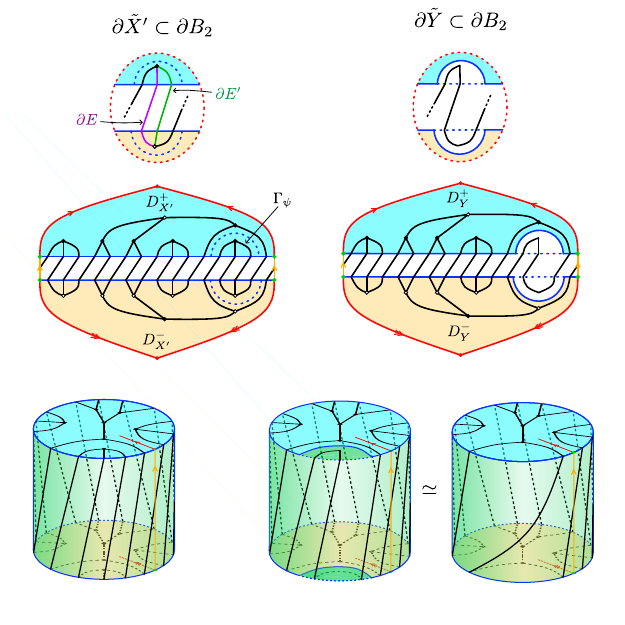}
    \caption{The two boundaries $\partial \tilde{Y} \subset \partial\tilde{X}'\subset \partial B_2$ and the disks $D_Y^{\pm}\subset D_{X'}^{\pm}$.}
    \label{fig:newD_Yboundarypattern}
\end{figure}
    
    Let $S_2'$ be the solid torus obtained by gluing $D_Y^{+}$ and $D_Y^{-}$ on the boundary of $B_2$ via the restrictions $q_{X'}:D_Y^{\pm}\rightarrow D_Y$ and let $q_{Y}:B_2 \rightarrow S_2'$ be the corresponding quotient map. To see that $q_Y(\tilde{Y})$ is a $\mathbb{Z}_3$-folded surface properly embedded in a solid torus note that $q_Y$ now identifies $\partial \tilde{Y} \cap D_Y^{+}$ and $\partial \tilde{Y} \cap D_Y^{-}$ via an orientation reversing map of $\mathbb{Z}_3$-folded $1$-manifolds. Figure~\ref{fig:newD_Yboundarypattern} also shows how $q_Y(\tilde{Y})\cap \partial S_2'$ is a $(3(n-1),1)$-torus knot. Let $S_1'$ be another solid torus and identify $\partial S_1'$ and $\partial S_2'$ by an orientation reversing homeomorphism such that this torus knot is identified with a meridional curve of $S_1'$. The resulting space is clearly $L(3(n-1),1)$. Let $Y$ be the $\mathbb{Z}_3$-folded surface embedded in $L(3(n-1),1)$ such that $Y\cap S_2' = q_Y(\tilde{Y})$ and $Y\cap S_1'$ is a single meridional disk properly embedded in $S_1'$ cobounding this meridional curve. Now note that since $X' \cap S_2$ is homeomorphic to the quotient of $\tilde{X}'$ by the identification of $q^{-1}(\Gamma_{X'}) \cap D^+$ and $q^{-1}(\Gamma_{X'}) \cap D^-$ with $\chi(\Gamma_{X'})$ it follows that $\chi(X' \cap S_2) = \chi(\tilde{X}') - \chi(\Gamma_{X'})$. Similarly $\chi(Y \cap S_2') = \chi(\tilde{Y}) - \chi(D_Y\cap Y)$. Since $D_Y\cap Y = \Gamma_{X'} \setminus \Gamma^{\psi}$ we have $\chi(D_Y\cap Y) = \chi(\Gamma_{X'}) - \chi(\Gamma^{\psi}) = \chi(\Gamma_{X'}) - 1$ and so $\chi(Y\cap S_2') = \chi(X' \cap S_2) + 1$. Now note that since $X'\cap S_1$ and $Y\cap S_1'$ are meridional disks of the solid tori $S_1$ and $S_1'$ respectively it follows that $\chi(Y) = \chi(Y\cap S_2')+1 = \chi(X' \cap S_2) + 2 = \chi(X)+1$. 
    
    For notational simplicity we now set $D_{X'} := D$, $D_{X'}^{\pm} := D^{\pm}$ and $\Gamma_Y := Y\cap D_Y$. We will now prove that $X'$ is non-separating in $L(3n,1)$ if and only if $Y$ is non-separating in $L(3(n-1),1)$. 
    
    Let $\sim_{X'}$ be an equivalence relation on $\pi_0(\partial B_2 \setminus \partial \tilde{X}')$ generated by the following relations:

    \begin{itemize}
        \item All $x\in \pi_0(\partial B_2 \setminus \partial \tilde{X}')$ such that $x\cap q_{X'}^{-1}(\partial S_2) \neq \emptyset$ belong to a single class.
        \item If there exists a component of $B_2\setminus \tilde{X}'$ that is a cobordism of $x$ and $y$ then $x\sim_{X'} y$.
        \item If there exists $z\in \pi_0(D_{X'} \setminus \Gamma_{X'})$ such that $z$ is contained in the interior of $D_{X'}$. Let $z^+,z^- \in \pi_0(\partial B_2 \setminus \partial \tilde{X}')$ such that $z^\pm \subset D_{X'}^{\pm}$ and $q_{X'}(z^{\pm}) = z$. Then $z^+\sim_{X'} z^-$. 
    \end{itemize}

    Define $\sim_{Y}$ an equivalence relation on $\pi_0(\partial B_2 \setminus \partial \tilde{Y})$ similarly swapping $S_2$ for $S_2'$.

    Note that the fact that $X'$ is non-separating in $L(3n,1)$ is equivalent to the fact that $X'\cap S_2$ is non-separating in $S_2$. When we lift to $\tilde{X}'$ this is equivalent to the condition that every component of $B_2\setminus \tilde{X}'$ intersects $\partial B_2$ and that $\pi_0(\partial B_2 \setminus \partial \tilde{X}')/\sim_{X'}$ consists of a single class.
    Similarly $Y$ is non-separating in $L(3(n-1),1)$ if and only if the boundary of every component of $B_2\setminus \tilde{Y}$ intersects $\partial B_2$ and $\pi_0(\partial B_2 \setminus \partial \tilde{Y})/\sim_{Y}$ consists of a single class.
    
    Since $\tilde{Y}\subset \tilde{X}'$ and $\tilde{X}'\setminus \tilde{Y}$ consists of a single disk $E'$ separating two components of $B_2\setminus \tilde{X}'$ which both intersect $\partial B_2$ it follows that every component of $B_2\setminus \tilde{Y}$ also intersects $\partial B_2$. 
    
    Let $h:\pi_0(\partial B_2 \setminus \partial \tilde{X}') \rightarrow \pi_0(\partial B_2 \setminus \tilde{Y})$ be the surjective function that maps each component of $\partial B_2 \setminus \partial \tilde{X}'$ to the unique component of $\partial B_2 \setminus \partial \tilde{Y}$ that contains it. We will show that $\pi_0(\partial B_2 \setminus \partial \tilde{Y})/\sim_{Y}$ consists of a single class by proving that $x \sim_{X'} y \Leftrightarrow h(x)\sim_{Y} h(y)$ for all $x,y \in \pi_0(\partial B_2 \setminus \partial \tilde{X}')$. It is sufficient to show that this is true for the generating relations with the following observations:

    \begin{itemize}
        \item $x$ intersects $q_{X'}^{-1}(\partial(S_2))$ if and only if $h(x)$ intersects $q_{Y}^{-1}(\partial(S_2'))$. This can be verified for each component that intersects $D^{+}_{X'}\setminus D^{+}_{Y}$ and $D^{-}_{X'}\setminus D^{-}_{Y}$ by means of Figure~\ref{fig:newD_Yboundarypattern} and it is true for all other components since $q_Y^{-1}(\partial S_2') = q_{X'}^{-1}(\partial S_2)\cup (D^{+}_{X'}\setminus D^{+}_{Y})\cup (D^{-}_{X'}\setminus D^{-}_{Y})$.
        
        \item If there exists a component of $B_2\setminus \tilde{X}'$ that is a cobordism of $x$ and $y$, then since $\tilde{Y} \subset \tilde{X}' \Rightarrow B_2\setminus \tilde{X}' \subset B_2\setminus \tilde{Y}$ there also exists a component of $B_2\setminus \tilde{Y}$ that is a cobordism of $x$ and $y$ and hence also $h(x)$ and $h(y)$. Moreover if $h(x)$ is cobordant to $h(y)$ in $B_2\setminus \tilde{Y}$ and $x$ is not cobordant to $y$ in $B_2\setminus \tilde{X}'$ then without loss of generality $x$ must be $C$ and $y$ must be cobordant to the component $F$ of $\partial B_2 \setminus \tilde{X}'$ such that $h(C) = h(F)$ (labelled in Figure~\ref{fig:newD_Yboundarypattern}). In this case we can use the fact that $F$ and $C$ both intersect $q_{X'}^{-1}(S_2)$ to deduce that $y\sim_{X'} F \sim_{X'} x$.
       
        \item If $z^{\pm}$ are components of $\partial B_2 \setminus \partial \tilde{X}'$ contained in the interior of $D^{\pm}_{X'}$, then the boundary of $z$ cannot intersect $\Gamma^{\psi}$. It follows that $z^{\pm}$ are both also contained in the interior of $D^{\pm}_Y$ and persist as components of $\partial B_2 \setminus \partial \tilde{Y}$. Conversely if $z^{\pm}$ are components of $\partial B_2 \setminus \partial \tilde{Y}$ contained in the interior of $D^{\pm}_Y$ then it is trivial that they are also components of $\partial B_2 \setminus \partial \tilde{X}'$ in the interior of $D^{\pm}_{X'}$ since $D^{\pm}_Y \subset D^{\pm}_{X'}$. Since $q_Y:D^{\pm}_Y\rightarrow D_Y$ is a restriction of $q_{X'}:D^{\pm}_{X'}\rightarrow D_{X'}$ it follows that $q_{X'}(z^+) = q_{X'}(z^-) \Leftrightarrow q_{Y}(z^+) = q_{Y}(z^-)$.
        
    \end{itemize}
    
    Since every component of $B_2\setminus \tilde{Y}$ also intersects $\partial B_2$ and $\pi_0(\partial B_2 \setminus \partial \tilde{Y})/\sim_{Y}$ consists of a single class it follows that $Y\cap S_2'$ is non-separating and hence $Y$ is non-separating. This means we can take $\chi(Y) = \chi(X')+1$ to be a lower bound on $\chi_3([Y])$. Recall that since $X$ was assumed to be $\chi$-optimal it follows that $X'$ is too since it is the result of a reduced $2$-compression. Hence we have proved Equation~\eqref{eq:firstineq}.

    \begin{equation}\label{eq:firstineq}
        \chi_3(L(3(n-1),1)) \geq \chi(Y) =\chi(X')+1 = \chi_3(L(3n,1))+1
    \end{equation}

    We will now prove that $\chi_3(L(3(n+1),1)) \geq \chi_3(L(3n,1))-1$ by constructing a non-separating $\mathbb{Z}_3$-folded surface $Z \subset L(3(n+1),1)$ such that $\chi(Z) = \chi(X)-1$. This is philosophically the dual of the construction of $Y$ above. We start by selecting a point $x$ in the interior of one of the intervals in the $q_{X}^{-1}(X \cap \partial S_2)$ and considering a regular neighbourhood $N(x)$ of $x$ in $B_2$. Since $x$ lies on $\partial B_2$,  $N(x)$ is a ball such that $\partial N(x)\setminus \partial B_2$ is a disk properly embedded in $B_2$, intersected transversally by $\tilde{X}$. Let $E$ be the disk $\tilde{X}\cap N(x)$ and $E'$ be the closure of either of the disk components of $\partial N(x)\setminus \partial B_2 \setminus \tilde{X}$. Let $\tilde{Z}$ be the $\mathbb{Z}_3$-folded surface properly embedded in $B_2$ with  underlying set $\tilde{X}\cup E'$ where the orientation of $E$ has been inverted and the orientation of $E'$ is set such that $E$ and $E'$ induce a consistent orientation on their common boundary $\partial N(x)\setminus \partial B_2 \cap \tilde{X}$. A similar Euler characteristic argument to before yields that $\chi(\tilde{Z}) = \chi(\tilde{X})$. The intersection of each of these components with $\partial B_2$ is shown in Figure~\ref{fig:newD_Zboundarypattern}.

    Now let $D_0^{\pm}$ be a pair of disks contained in $q_X^{-1}(\partial S_2)$ as shown in Figure~\ref{fig:newD_Zboundarypattern} such that $\partial D_0^+ \cap \partial D_X^+$ is an interval identified with $\partial D_0^- \cap \partial D_X^-$ and $D_0^+ \cap \tilde{Z}$ and $D^-_0 \cap \tilde{Z}$ are both boundary tripods. Fix an orientation reversing homeomorphism of these boundary tripods that preserves the order of their boundary vertices. Using Alexander's trick over the three disks in the complement of these tripods we can extend the homeomorphisms $\partial D_0^+ \cap \partial D_X^+ \rightarrow \partial D_0^- \cap \partial D_X^-$ and $\tilde{Z}\cap D_0^+ \rightarrow \tilde{Z} \cap D_0^-$ to a homeomorphism $D_0^+ \rightarrow D_0^-$. Let $D_Z^{\pm} := D_0^{\pm}\cup D_X^{\pm}$ and combine the maps of $D_0^{\pm}$ and $D_X^{\pm}$ into an orientation reversing homeomorphism of $D_Z^{\pm}$. Let $S_2'$ be the solid torus obtained by a quotient of $B_2$ by identifying $D_Z^{+}$ and $D_Z^-$ in this way and let $q_Z: B_2 \rightarrow S_2'$ be the corresponding quotient map. Note that $q_Z(\tilde{Z})$ is a $\mathbb{Z}_3$-folded surface properly embedded in the solid torus $S_2'$ since the identification of $D_Z^+$ with $D_Z^-$ restricts to an orientation reversing map of $\mathbb{Z}_3$-folded $1$-manifolds $D_Z^+ \cap \tilde{Z}$ and $D_Z^- \cap \tilde{Z}$. Moreover $\partial(q_Z(\tilde{Z}))$ is a $(3(n+1),1)$-torus knot on the boundary torus $\partial S_2'$. As in the case of $Y$ above we let $S_1'$ be another oriented solid torus and glue $\partial S_1'$ and $\partial S_2'$ via an orientation reversing map such that the $\partial (q_Z(\tilde{Z}))$ is identified with a meridional curve on $\partial S_1'$. The resulting space is the lens space $L(3(n+1),1)$. Let $Z \subset L(3(n+1),1)$ be the $\mathbb{Z}_3$-folded surface such that $Z\cap S_2' = q_Z(\tilde{Z})$ and $Z\cap S_1'$ is a meridional disk properly embedded in $S_1'$ cobounding $\partial(q_Z(\tilde{Z}))$. 

    The same arguments regarding the relationship between $\chi(\tilde{X})$, $\chi(\Gamma_X)$, $\chi(X\cap S_2)$ and $\chi(X)$ with the corresponding quantities for $Z$ apply as previously in the case of $X'$ and $Y$ leading to $\chi(Z) = \chi(X)-1$. Moreover what we have just done is essentially the inverse of the construction of $Y$ from $X'$. By taking $Z\mapsto X'$, $X\mapsto Y$ and $n+1 \mapsto n$ we can use the same arguments above to show that $Z$ is non-separating in $L(3(n+1),1)$ since in particular $X'$ was non-separating if and only if $Y$ was non-separating. This implies Equation~\eqref{eq:secondineq}.
     
    \begin{equation}\label{eq:secondineq}
        \chi_3(L(3(n+1),1))\geq \chi(Z) = \chi(X)-1 = \chi_3(L(3n,1))-1
    \end{equation}

    Combining the inequalities in Equations~\eqref{eq:firstineq} and~\eqref{eq:secondineq} delivers the desired result that $\chi_3(L(3n,1)) = \chi_3(L(3(n-1),1))-1$.

\begin{figure}
    \centering
    \includegraphics[width=0.8\linewidth]{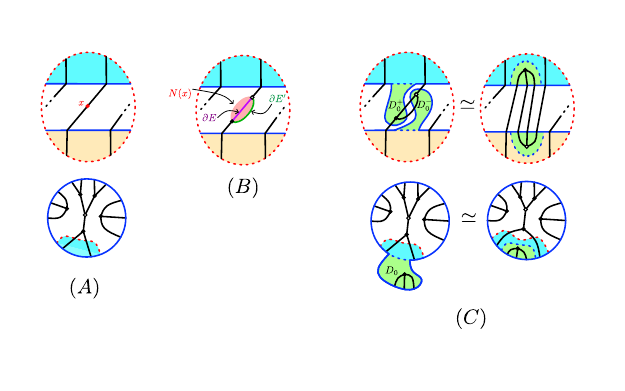}
    \caption{The two boundaries $\partial\tilde{X} \subset \partial \tilde{Z} \subset \partial B_2$ and the disks $D_Z^{\pm} =  D_{X}^{\pm} \cup D_0^{\pm}$.}
    \label{fig:newD_Zboundarypattern}
\end{figure}

\end{proof}

We expect that it should be possible to determine $\chi_3(L(3n,q))$ using similar techniques but we leave this to future work.

\section{A naïve complexity bound for $L(3n,1)$}\label{sec:naivecomplexity}

We now give Corollary~\ref{cor:naivecomplexitybound} which demonstrates the application of this theory of $\mathbb{Z}_p$-folded surfaces to the generalised triangulation complexity of $3$-manifolds. The application of the theory of canonical $\mathbb{Z}_3$-folded surfaces to study the complexity of lens spaces is joint work with Jonathan Spreer and Corollary~\ref{cor:naivecomplexitybound} should be considered as a preview of more sophisticated techniques from our forthcoming paper.

\begin{cor}[A lower bound on the complexity of $L(3n,1)$]\label{cor:naivecomplexitybound}
    It follows from Theorem~\ref{thm:l3n1norm} that $c(L(3n,1))\geq n$, where $c(L(3n,1))$ is the minimum number of tetrahedra required to form a generalised triangulation of $L(3n,1)$.
\end{cor}
\begin{proof}
    First note that due to~\cite[Theorem~6.1]{jaco20020efficienttriangulations3manifolds} it follows that any minimal triangulation $\mathcal{T}$ of $L(3n,1)$ with $n>1$ has a single vertex and it is well-known that $c(L(3,1))=2$. Let $\varphi \in H^1(L(3n,1);\mathbb{Z}_3)$ be the unique non-trivial cohomology class (up to a sign inversion). Up to normal isotopy and symmetries the intersection of $S_{\varphi}$ with each tetrahedron has four distinct types. A type $A$ tetrahedron is intersected in a quadrilateral normal disk. A type $B$ tetrahedron is intersected in a triangle normal disk. A type $C$ tetrahedron does not intersect $S_{\varphi}$ and the intersection of a type $D$ tetrahedron with $S_{\varphi}$ is homeomorphic to $K_{3,1}\times[0,1]$. These tetrahedron types are pictured in Figure~\ref{fig:tetrahedrontypes3}. Let $A, B, C, D$ be the number of type $A,B,C$ and $D$ tetrahedra in $\mathcal{T}$ respectively. We now follow identical reasoning to that used in the $\mathbb{Z}_2$ case for canonical non-orientable surfaces in~\cite{Jaco_2009}. We can partition the edges of $\mathcal{T}$ into \textit{even edges} that are not intersected by $S_{\varphi}$ and \textit{odd edges} that intersect $S_{\varphi}$ in a single point. Let the number of even edges be $\mathfrak{e}$ and the number of odd edges be $\mathfrak{o}$. Since $L(3n,1)$ is a closed $3$-manifold it follows that $\chi(S_{\varphi}) = \chi(L(3n,1) \setminus S_{\varphi})$. By pieceing together retractions for each tetrahedron type one can show that $L(3n,1) \setminus S_{\varphi}$ deformation retracts onto the subcomplex of $\mathcal{T}$ spanned by the even edges. This subcomplex has a single vertex, $\mathfrak{e}$ edges, $\frac{B+4C}{2}$ triangles and $C$ tetrahedra. Hence we deduce that $\chi(S_{\varphi}) = 1 - \mathfrak{e} + \frac{B}{2} + C$. It follows from Corollary~\ref{cor:chi3bound} that $\chi(S_{\varphi})\leq \chi_3(\varphi^*)$. Theorem~\ref{thm:l3n1norm} implies that $\chi_3(\varphi^*) = 2-n$ and hence $\chi(S_{\varphi})\leq 2-n$. Combining these two relations gives $1 - \mathfrak{e} + \frac{B}{2} + C \leq 2-n$ and hence $\mathfrak{e} \geq n-1$. Since $\varphi$ is non-trivial it follows that $\mathfrak{o}\geq 1$. Moreover since the lens spaces $L(3n,1)$ for $n>1$ have trivial integral second homology it follows that $\Sing(S_{\varphi}) \neq \emptyset$. Consider a triangle in $\mathcal{T}$ that intersects the singular locus of $S_{\varphi}$. If $\mathfrak{o}=1$ then all three edges of this triangle are identified such that the image of this triangle in $L(3n,1)$ is a $\mathbb{Z}_3$-crosscap $X$ (see Figure~\ref{fig:pcrosscaps}). $\mathbb{Z}_3$-crosscaps are always non-separating since they have trivial integral second homology. The Euler characteristic of a $\mathbb{Z}_3$-crosscap is $1$ which would imply that $\chi_3([X])\geq 1$, but Theorem~\ref{thm:l3n1norm} implies that $\chi_3([X]) = 2-n$. Our assumption that $n>1$ implies that $\chi_3([X])<1$ so we have a contradiction and we conclude that there must be at least two distinct odd edges. Hence we deduce that $\mathfrak{o}\geq 2$ and so $\mathfrak{e}+\mathfrak{o} \geq n+1$. Since $L(3n,1)$ is closed $\chi(L(3n,1))= 0$ and Equation~\eqref{eq:countingcellscanonical3fold} follows.

    \begin{equation}\label{eq:countingcellscanonical3fold}
        1 - \mathfrak{e} - \mathfrak{o} + 2c(L(3n,1)) - c(L(3n,1)) = \chi(L(3n,1))=0
    \end{equation}

    The required result $c(L(3n,1)) = \mathfrak{e} + \mathfrak{o} -1 \geq n$ follows immediately.
\end{proof}

\begin{figure}
    \centering
    \includegraphics[width=0.8\linewidth]{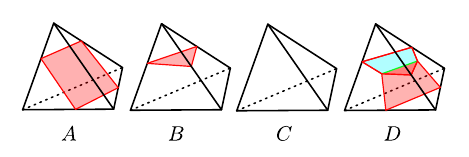}
    \caption{The four intersection types of the canonical $\mathbb{Z}_3$-folded surface with tetrahedra in a one-vertex triangulation (up to symmetries of a tetrahedron).}
    \label{fig:tetrahedrontypes3}
\end{figure}

More sophisticated arguments can be used to determine the exact complexity of the lens spaces $L(3n,1)$ using canonical $\mathbb{Z}_3$-folded surfaces. We will present this joint work in a forthcoming paper.

\bibliographystyle{alpha}
\bibliography{refs}

\end{document}